\documentclass{article}
\usepackage{amsmath,amssymb,amsthm,scrtime,bbm,verbatim,mathrsfs}
\usepackage{graphics}
\usepackage{amsfonts,xcolor}
\usepackage{tikz}
\usepackage{bm}
\usepackage{csquotes}
\MakeOuterQuote{"}

\newcommand{\E}{\mathbb{E}}
\renewcommand{\P}{\mathbb{P}}
\newcommand{\N}{\mathbb{N}}
\newcommand{\Z}{\mathcal{Z}}

\newcommand{\eps}{\epsilon}
\newcommand{\sss}{\scriptscriptstyle}
\newcommand{\F}{\mathcal{F}}
\newcommand{\eset}{{\sf{E}}}
\newcommand{\vset}{{\sf{V}}}
\newcommand{\graph}{{\sf{G}}}
\newcommand{\connect}{{\sf{C}}}
\newcommand{\eseteps}{{\sf{E}}^{\eps}}
\newcommand{\vseteps}{{\sf{V}}^{\eps}}
\newcommand{\grapheps}{{\sf{G}}^{\eps}}

\newcommand{\PA}{{\rm{PA}}}
\newcommand{\CL}{{\rm{CL}}}

\newcommand{\CM}{{\rm{CM}}}
\newcommand{\path}{{\sf{p}}}
\newcommand{\dist}{{\sf{d}}}
\newcommand{\pc}{p_{\rm c}}
\newcommand{\loc}{\lambda} %location		
\newcommand{\mar}{\alpha} %mark
\newcommand{\IBP}{{\mathrm{IBP}}} %name branching process
\newcommand{\fst}{\lfloor \eps n\rfloor} %first vertex in damages graph
\newcommand{\rightp}{{\rm{r}}} %mark: r=parent on right
\newcommand{\leftp}{\ell} %mark: l=parent on left
\newcommand{\maxdeg}{{\sf{M}}} % maximal indegree
\newcommand{\typespace}{\Phi} % typespace
\newcommand{\type}{\phi} % type
\newcommand{\Asimple}{\bar{A}} %simplified version of A
\newcommand{\p}{\mathfrak{p}} %typical element of point process
\newcommand{\q}{\mathfrak{q}} %typical element of point process
\newcommand{\tree}{\mathbb{T}^{\eps}} %tree = intermediate step in coupling
\newcommand{\Ftree}{\mathbb{T}} %tree without the epsilon
\newcommand{\stoptime}{T} %stopping time in Coupling
\newcommand{\threshold}{\hat{n}} %step from which f(0)>1 doesn't matter anymore
\newcommand{\pathset}{\mathcal{S}} % set of paths in Distances Section
\newcommand{\eigenfct}{\varphi} % eigenfunction of spectral radius
\newcommand{\inMC}{\mathcal{Y}} %inhomogeneous Markov chain, component

\newcommand{\heap}[2]  {\genfrac{}{}{0pt}{}{#1}{#2}}
\newcommand{\sfrac}[2] {\mbox{$\frac{#1}{#2}$}}

\numberwithin{equation}{section}

\newtheorem{theorem}{Theorem}[section]
\newtheorem{lemma}[theorem]{Lemma}

\newtheorem{corollary}[theorem]{Corollary}
\newtheorem{proposition}[theorem]{Proposition}

\theoremstyle{definition}
\newtheorem{definition}[theorem]{Definition}

\title{Vulnerability of robust preferential attachment networks}
\author{Maren Eckhoff\thanks{Department of Mathematical Sciences, University of Bath, Bath, BA2 7AY, United Kingdom.\newline Email: {\tt m.eckhoff@bath.ac.uk}, {\tt maspm@bath.ac.uk}}
\and
Peter M\"orters$^*$}
\date{}

\begin{document}
\maketitle

\begin{abstract}
\noindent
Scale-free networks with small power law exponent are known to be robust, meaning that their qualitative topological structure
cannot be altered by random removal of even a large proportion of nodes. 
By contrast, it has been argued in the science literature that such networks are highly vulnerable to a targeted attack, and removing a small 
number of key nodes in the network will dramatically change the topological structure.  Here we analyse a class of preferential attachment networks 
in the robust regime and prove four main results supporting this claim: After removal of an arbitrarily small proportion $\eps>0$ of the oldest 
nodes (1) the asymptotic degree distribution has exponential instead of power law tails; (2) the largest degree in the network drops from
being of the order of a power of the network size~$n$ to being just logarithmic in~$n$; (3) the typical distances in the network increase from order $\log\log n$ to
order $\log n$; and (4) the network becomes vulnerable to random removal of nodes. Importantly, all our results explicitly quantify the dependence on the 
proportion $\eps$ of removed vertices. For example, we show that the critical proportion of nodes that have to be retained
for survival of the giant component undergoes a steep increase as $\eps$ moves away from zero, and a comparison of this result with similar ones
for other networks reveals the existence of two different universality classes of robust network models. The key technique in our proofs is a local approximation of the network 
by a branching random walk with two killing boundaries, and an understanding of the  particle genealogies in this process, which enters into estimates 
for the spectral radius of an associated~operator.
\end{abstract}

\bigskip

\noindent
{\bf{Key words:}} Power law, small world,  scale-free network, preferential attachment, Barab\'asi-Albert model, percolation, maximal degree,
diameter, network distance, robustness, vulnerability, multitype branching process, killed branching random walk, spectral radius estimation.\\[0.5cm]

\noindent
{\bf{Mathematics Subject Classification (2010):}} Primary 05C80; Secondary 60J85, 60K35, 90B15\\[0.5cm]

\newpage

\section{Introduction}

\subsection{Motivation}

The problem of resilience of networks to either random or targeted attack is crucial to many instances of real world networks, ranging from
social networks (like collaboration networks) via technological networks (like electrical power grids) to communication networks (like the 
world-wide web). 
Of particular importance is whether the connectivity of a network relies on a small number of hubs and whether their loss will cause a large-scale breakdown.
Albert, Albert and Nakarado~\cite{AAN04} argue that ``the power grid is robust to most perturbations, 
yet disturbances affecting key transmission substations greatly reduce its ability to function''.
Experiments of Albert, Jeong, and Barab\'asi~\cite{AJB00}, 
Holme, Kim, Yoon and Han~\cite{HolKimYonHan02} and more recently of Mishkovski, Biey and Kocarev~\cite{MisBieKoc11} find robustness under random attack but vulnerability to the removal of a small number of key nodes in several other networks. 
The latter paper includes a study of data related to the human brain, as well as street, collaboration and power grid networks.
One should expect this qualitative behaviour across the range of real world networks and 
it should therefore also be present in the key mathematical models of large complex networks. 
\smallskip\pagebreak[3]

A well established feature of many real world networks is that in a suitable range of values~$k$ the proportion of nodes with degree~$k$ 
has a decay of order $k^{-\tau}$ for a power law exponent~$\tau$. The robustness of networks with small power law exponent under 
random attack has been observed heuristically by Callaway et al.~\cite{CNSW00} and Cohen et al.~\cite{CEbAH01}, but there seems to be controversy in these early papers
about the extent of the vulnerability in the case of targeted attack, see the discussion in \cite{DorMen01} and \cite{CEbAHReply}. 
As Bollob{\'a}s and Riordan~\cite[Section~10]{BR03} point out, such heuristics, informative as they may be, are often quite far away from a mathematical 
proof that applies to a given model. In their seminal paper~\cite{BR03} they provide the first rigorous proof of robustness in the case of a specific 
preferential attachment model with power law exponent~$\tau=3$, and later Dereich and M\"orters~\cite{DerMoe13} proved for a class of preferential attachment 
models with tunable power law exponent that networks are robust under random attack if the power law exponent satisfies $\tau\leq 3$, but not when $\tau>3$,
thus revealing the precise location of the phase transition in the behaviour of preferential attachment networks.  However, the question of vulnerability of robust 
networks when a small number  of privileged nodes is removed has not been studied systematically in the mathematical  literature so far.
\smallskip

It is the aim of the present paper to give evidence for the vulnerability of robust networks by providing rigorous proof that preferential attachment networks 
in the robust regime $\tau\leq 3$ undergo a radical change under a targeted attack, i.e.\ when an arbitrarily small 
proportion $\eps>0$ of the most influential nodes in the network is removed. Our main results, presented in Section~\ref{ss-13},
show how precisely this change affects the degree structure, the length of shortest paths and the connectivity in the network. The results  take the form
of limit theorems revealing explicitly the dependence of the relevant parameters on $\eps$. Not only does this provide further insight into the topology of the network 
and the behaviour as $\eps$ tends to zero, it also allows a comparison to other network models, and thus exposes two classes of robust networks with rather 
different behaviour, see Section~\ref{ss-15}. Our mathematical analysis of the network is using several new ideas and combines  probabilistic and combinatorial arguments with analytic techniques informed by new probabilistic 
insights. It is crucially based on the local approximation of preferential attachment networks by a branching random walk  with a killing boundary recently found in~\cite{DerMoe13}.  
In this approximation the removal of a proportion of old vertices corresponds to the introduction of a second killing boundary. On the one hand  this adds an additional level  
of complexity to the process, as the mathematical understanding of critical phenomena in branching  models on finite intervals is only just emerging, see for example~\cite{HarHesKyp13pre}. 
On the other hand compactness of the typespace for this branching process opens up new avenues that are exploited, for example, in the form of spectral estimates
based on rather subtle information on the shape of principal eigenfunctions of an operator associated with the branching process. 

\pagebreak[3]

\subsection{Mathematical framework}

The established mathematical model for a large network is a sequence $(\graph_n \colon n \in \N)$ of (random or deterministic)
graphs $\graph_n$ with vertex set $\vset_n$ and an edge set $\eset_n$ consisting of (ordered or unordered) edges between the vertices.
We assume that the size $|\vset_n|$ of the vertex set is increasing to infinity in probability, so that results about the limiting behaviour in the sequence of
graphs may be seen as predictions for the behaviour of large networks. In all cases  of interest here the asymptotic number of edges per vertex is finite so that
the topology in a bounded neighbourhood of a typical vertex is not affected by the overall size of the graph. As a result, metrics based on the local 
features of the graph typically stabilize. An important example for this is the proportion of vertices with a given degree in $\graph_n$, which in
the relevant models converges and allows us to talk about the asymptotic degree distribution. The mathematical models of power law networks therefore have an asymptotic degree distribution with the probability of degree $k$ decaying like~$k^{-\tau}$, as $k\to\infty$, for some $\tau>1$. 
Our focus here is on the global properties emerging in network models with asymptotic power law degree distributions.
\smallskip

A crucial global feature of a network is its connectivity, and in particular the existence of a large connected component. To describe this, we denote
by $\connect_n$ a connected component in $\graph_n$ with maximal number of nodes. The graph sequence $(\graph_n \colon n \in \N)$ has \emph{a giant component} if there exists 
a constant $\zeta>0$ such that
\[
\frac{|\connect_n|}{\E|\vset_n|} \to \zeta \qquad \text{as }n \to \infty,
\]
where the convergence holds in probability. 
We remark that for the models usually considered the issue is not the convergence itself but the
positivity of the limit~$\zeta$. If a giant component exists and the length of the shortest path between any two vertices
in the largest component of $\graph_n$ is asymptotically bounded by a multiple of $\log n$, the network
is called \emph{small}. If it is asymptotically bounded by a constant multiple of $\log\log n$, the
network is called \emph{ultrasmall}.
\smallskip

To model a \emph{random attack} on the network we keep each vertex in $\graph_n$ independently with probability~$p\in[0,1]$ and otherwise we remove it from 
the vertex set together with all its adjacent edges, i.e.\ we run \emph{vertex percolation} on $\graph_n$ with \emph{retention probability} $p$. 
The resulting graph is denoted by $\graph_n(p)$. A simple coupling argument shows that there exists a critical parameter $\pc\in[0,1]$ such that 
the sequence $(\graph_n(p) \colon n \in \N)$ has a giant component if $\pc<p\leq 1$, and it does not have a giant component if $0\leq p<\pc$.
If $\pc=0$, i.e.\ if the giant component cannot be destroyed by percolation with any positive retention parameter, then the network is called \emph{robust}.
To study resilience of networks to \emph{targeted attack} we look at models in which the construction of the network favours certain vertices in such a way 
that these privileged vertices have a better chance of getting a high degree than others. When $\graph_n$ is a network on $n$ vertices, we label these by 
$1$ to $n$ and assume that vertices are ordered in decreasing order of privilege. This assumption allows an attacker to target the most
privileged vertices without knowledge of the entire graph.  The \emph{damaged graph} $\grapheps_n$, for some $\eps\in (0,1)$, is obtained from $\graph_n$ by 
removal of all vertices with label less or equal than $\eps n$ together with all adjacent edges. In particular, the new vertex set is 
$\vseteps_n=\{\fst+1,\ldots,n\}$.
\smallskip

We investigate the problem of vulnerability of random networks to targeted attack  in the context of \emph{preferential attachment networks}. This class
of models has been popularised by Barab\'asi and Albert~\cite{BA99} and has received considerable attention in the scientific literature. The idea is that
a sequence of graphs is constructed by successively adding vertices. Together with a new vertex, edges are introduced by connecting it
to existing vertices at random with a probability depending on the degree of the existing node; the higher the degree the more likely the connection. 
Despite the relatively simple principle on which this model is based it shows a good match of global features with real networks. For example, the 
asymptotic degree distributions follow a power law, and if the power law exponent satisfies $\tau< 3$, then the network is robust and ultrasmall. Variations in the 
attachment probabilities allow for tuning of the power law exponent~$\tau$.
\smallskip

The first mathematically rigorous study of resilience in preferential attachment networks was performed by 
Bollob{\'a}s and Riordan~\cite{BR03} for the so-called LCD model. This model variant has the advantage of having
an explicit static description, which makes it much easier to analyse than models that have only a dynamic description.
It also has a fixed power law exponent $\tau=3$, hence, 
Bollobas and Riordan~\cite{BR03} prove only results for this specific exponent. They show that in this case the network is 
robust and identify a critical proportion $\eps_{\rm c}<1$ such that the removal of the oldest $\fst$ oldest vertices 
leads to the destruction of the giant component if and only if $\eps\ge\eps_{\rm c}$. Note that this is not in line with the notion of 
vulnerability that we are interested in,  as we only want to remove a \emph{small} proportion of old vertices.
\smallskip

In the present paper, we consider the question of vulnerability in the following model variant, introduced in~\cite{DerMoe09}. Let $\N_0$
be the set of nonnegative integers and fix a function $f\colon \N_0 \to (0,\infty)$, 
which we call the \emph{attachment rule}. The most important case is if $f$ is affine, i.e.\ $f(k)=\gamma k +\beta$ for parameters $\gamma\in[0,1)$ and 
$\beta>0$, but non-linear functions are allowed. Given an attachment rule $f$, we define a growing sequence $(\graph_n \colon n\in \N)$ of random graphs by the 
following dynamics:

\begin{itemize}
\item Start with one vertex labelled $1$ and no edges, i.e.\ $\graph_1$ is given by $\vset_1:=\{1\}$,~$\eset_1:=\emptyset$;
\item Given the graph $\graph_n$, we construct $\graph_{n+1}$ from $\graph_n$ by adding a new vertex labelled $n+1$ and, for each $m \le n$ independently, 
inserting the directed edge $(n+1,m)$ with probability
\[
\frac{f(\text{indegree of } m \text{ at time }n)}{n} \land 1.
\] 
\end{itemize}
Formally we are dealing with a sequence of directed graphs but all edges point from the younger to the older vertex. 
Hence, directions can be recreated from the undirected, labelled graph. For all structural questions, particularly regarding 
connectivity and the length of shortest paths, we regard $(\graph_n \colon n\in\N)$ as an undirected network.
Dereich and M\"orters consider in \cite{DerMoe09, DerMoe13} concave attachment rules~$f$. Denoting the asymptotic slope of $f$ by
\begin{equation}\label{gammadef}
\gamma:=\lim_{k \to \infty} \frac{f(k)}{k},
\end{equation} 
they show that for $\gamma \in (0,1)$ the sequence $(\graph_n \colon n\in\N)$ has an asymptotic degree distribution which follows a
power law with exponent 
\[
\tau=\frac{\gamma+1}{\gamma}.
\] 
In the case $\gamma <\frac{1}{2}$, or equivalently $\tau> 3$, there exists a critical percolation parameter $\pc>0$ such that $(\graph_n(p) \colon n\in\N)$ has a giant component 
if and only if $p>\pc$.\footnote{The results of \cite{DerMoe13} are formulated for edge percolation, whereas we consider vertex percolation. It is not hard to see that for the existence 
or nonexistence of the giant component this makes no difference.}  If however $\gamma\ge \frac{1}{2}$, or equivalently $\tau\leq 3$,  
the sequence $(\graph_n(p)  \colon n\in\N)$ has a giant component for all $p \in (0,1]$, i.e.\ $(\graph_n \colon n\in\N)$ is  robust. This is the regime of interest in
the present paper.

\subsection{Statement of the main results}\label{ss-13}

Our focus is on the case of an affine attachment rule $f(k)=\gamma k +\beta$ with $\beta >0$ and $\gamma \in [\frac{1}{2},1)$. Recall that for this choice
the preferential attachment network is robust. We use the symbol $a(\eps)\asymp b(\eps)$ to indicate that there are constants $0<c<C$ and some $\eps_0>0$ such 
that $c b(\eps)\leq a(\eps) \leq C b(\eps)$ for all $0<\eps<\eps_0$.

\begin{theorem}{\bf{(Loss of connectivity)}}\label{thm:pc}
For any $\eps \in (0,1)$, there exists $\pc(\eps)\in (0,1]$ such that the damaged network
\begin{equation} \label{lin_pc_exist}
(\grapheps_n(p) \colon n\in\N)  \text{ has a giant component } \Leftrightarrow \; p> \pc(\eps).
\end{equation}
If $\gamma = \frac{1}{2}$ then 
\begin{equation}\label{pccrit}
\pc(\eps)  \asymp \frac{1}{\log (1/\eps)}.
\end{equation}
If $\gamma > \frac{1}{2}$ then, as $\eps \downarrow 0$,
\[
\pc(\eps) =\frac{2\gamma-1}{\sqrt{\beta (\gamma+\beta)}}\eps^{\gamma -1/2} \big[1+O\big(\eps^{\gamma -1/2} (\log \eps)\big)\big].
\]
\end{theorem}

Theorem~\ref{thm:pc} shows that the removal of an arbitrarily small proportion of old nodes makes the network vulnerable 
to percolation, but does not destroy the giant component. The steep increase of $\pc(\eps)$ as $\eps$ leaves zero shows that, even when a small proportion 
of old nodes has been removed in the network,  the 
removal of further old nodes is much more destructive than the removal of a similar proportion of randomly chosen nodes.
\smallskip

As for small $\eps$ the
critical value $\pc(\eps)$ is strictly decreasing in $\gamma$, this effect is stronger the closer $\gamma$ is to $\frac{1}{2}$. This result might be perceived as slightly counterintuitive since the preferential attachment becomes stronger as $\gamma$ increases and therefore we might expect older nodes to be more privileged and a targeted attack to be more effective than in the small $\gamma$ regime. However, the effect of the stronger preferential attachment is more than compensated by the fact that networks with a small value
of $\gamma$ have a (stochastically) smaller number of edges and are therefore a-priori more vulnerable. Note also that $\pc(\eps)$ may be equal to $1$ if $\eps$ is not sufficiently small 
in which case \eqref{lin_pc_exist} implies that the damaged network has no giant component. 
%
%\smallskip
%
In the case $\gamma=\frac{1}{2}$ the implied constants in \eqref{pccrit} can be made explicit as \smash{$c=\frac{1}{\gamma+\beta}$} and 
\smash{$C=\frac{1}{\beta}$}, but we cannot
show that they match asymptotically. However, we conjecture that
\[
\pc(\eps)  \sim \frac{1}{\sqrt{\beta (\gamma+\beta)}} \,  \frac{1}{\log(1/\eps)} \qquad \text{as } \eps \downarrow 0,
\]
meaning that the ratio of the left- and right-hand sides converges to one.
\smallskip

To gain further insight into the topology of the damaged graph, we now look at the asymptotic indegree distribution and at the largest 
indegree in the network. Recall from~\cite{DerMoe09} that outdegrees are asymptotically Poisson distributed and therefore indegrees are
solely responsible for the power law behaviour as well as the dynamics of maximal degrees. From here onwards we additionally assume 
that~$\beta\leq 1$. 
\smallskip

For a probability measure $\nu$ on the nonnegative integers we write 
$\nu_{\ge k}:=\nu(\{k,k+1,\ldots\})$ and $\nu_k:=\nu(\{k\})$. We denote the indegree of vertex $m$ at time $n \ge m$ by $\Z[m,n]$ and the empirical indegree distribution in 
$\grapheps_n$ by $X^{\eps}(n)$, that is,
\[
X_k^{\eps}(n)=\frac{1}{n-\fst} \sum_{m=\fst+1}^n \mathbbm{1}_{\{k\}}(\Z[m,n]), \qquad \mbox{ for } k \in \N_0.
\]
We write $\maxdeg(\graph)$ for the maximal indegree in a directed graph $\graph$. For $s,t >0$, let
\[
B(s,t):=\int_0^1 x^{s-1} (1-x)^{t-1} \, dx
\]
denote the beta function at $(s,t)$.
Before we make statements about the network after the targeted attack, we recall the situation in the undamaged network. 
In this case Dereich and M\"orters in \cite{DerMoe09} show that, almost surely,
\[
\lim_{n \to \infty} X^0(n) =\mu
\]
in total variation norm. The limit is the probability measure~$\mu$ on the nonnegative integers given by
\[
\mu_{\ge k}= \frac{B(k+\frac{\beta}{\gamma},\frac{1}{\gamma})}{B(\frac{\beta}{\gamma}, \frac{1}{\gamma})} \qquad \mbox{ for } k \in \N_0,
\]
and satisfies $\lim_{k \to \infty} \log\mu_{\ge k}/\log k =-1/\gamma$. Moreover, the maximal indegree satisfies, in probability,
\[
\frac{\log \maxdeg(\graph_n)}{\log (n^{\gamma})}\to 1  \qquad \text{as }n\to \infty.
\]
Our theorem shows that in the damaged network the asymptotic degree distribution is no longer a power law but has exponential tails. The maximal degree 
grows only logarithmically, not polynomially.
\medskip

\begin{theorem}{\bf{(Collapse of large degrees)}}\label{thm:degrees}
Let $\eps \in (0,1)$. Almost surely,
\[
\lim_{n \to \infty} X^{\eps}(n) =\mu^{\eps}
\]
in total variation norm. The limit is the probability measure~$\mu^{\eps}$ on the nonnegative integers given by
\[
\mu_{\ge k}^{\eps}= \int_{\eps}^1 \frac{1}{1-\eps} B\big(k,\tfrac{\beta}{\gamma}\big)^{-1} \int_{y^{\gamma}}^1 x^{\frac{\beta}{\gamma} -1} (1-x)^{k-1} \, dx \, dy \qquad \mbox{ for } k \in \N_0.
\]
It satisfies $\lim_{k \to \infty} \log\mu_{\ge k}^{\eps}/k = \log(1-\eps^{\gamma})$. Moreover, the maximal indegree satisfies, in probability,
\begin{equation}\label{eq:MaxDegThm}
\frac{\maxdeg(\grapheps_n)}{\log n}\to -\frac{1}{\log(1-\eps^{\gamma})} \qquad \text{as }n\to \infty.
\end{equation}
\end{theorem}
\pagebreak[3]

It is worth mentioning that $\mu=\mu^{0}$, so Theorem~\ref{thm:degrees} remains valid for $\eps =0$.
Theorem~\ref{thm:degrees} shows in particular that by removing a proportion of the oldest vertices we have removed all
vertices with a degree bigger than a given constant multiple of $\log n$. This justifies the comparison of our vulnerability
results with empirical studies of real world networks such as~\cite{AAN04}, in which all nodes whose degree exceeds a given 
threshold are removed. Note also that as $\eps\downarrow 0$ the right-hand side in \eqref{eq:MaxDegThm}
is asymptotically equivalent to $\eps^{-\gamma}$ and the growth of the maximal degree 
is the faster the larger $\gamma$. 
\bigskip

\pagebreak[3]

Preferential attachment networks are  ultrasmall for sufficiently small power law exponents. For our model, 
M\"onch~\cite{Moench13pre}, see also~\cite{DerMoeMoe12}, has shown that, denoting by $\dist_{\graph}$ the graph distance in the graph~$\graph$, 
for independent random vertices $V_n,W_n$ chosen uniformly from~$\connect_n$, we have
\begin{alignat*}{3}
&\text{if }\gamma = \frac{1}{2}\text{ then}\qquad \qquad & \qquad &\dist_{\graph_n}(V_n,W_n) \sim \frac{\log n}{\log \log n},\qquad\qquad &&\qquad\\
&\text{if }\gamma > \frac{1}{2}\text{ then}\qquad \qquad &  \qquad&\dist_{\graph_n}(V_n,W_n) \sim \frac{4\log \log n}{\log(\gamma/(1-\gamma))},\qquad\qquad&& \qquad
\end{alignat*}
meaning that the ratio of the left- and right-hand side converges to one in probability as $n \to \infty$. Removing an arbitrarily small proportion of old 
vertices however leads to a massive increase in the typical distances, as our third main theorem reveals. We say that a sequence of events $({\mathcal E}_n \colon n\in\N)$
holds \emph{with high probability} if $\P({\mathcal E}_n)\to 1$ as $n\to\infty$.
\smallskip

\begin{theorem}{\bf{(Increase of typical distances)}}\label{thm:distances}
Let $\eps>0$ be sufficiently small so that $\pc(\eps)<1$, and let $V_n,W_n$ be independent, uniformly chosen vertices 
in $\vseteps_n$. Then, for all $\delta >0$,
\[
\dist_{\grapheps_n}(V_n,W_n) \ge \frac{1-\delta}{\log (1/\pc(\eps))}\,  \log n \qquad \text{with high probability}. 
\]
\end{theorem}

Our proof gives the result for all values~$\gamma \in [0,1)$ but if $\gamma<\frac12$ even without removal of old vertices
the typical distances in the network are known to be of order $\log n$, so that this is not surprising. We believe that there is an upper bound
matching the lower bound above, but the proof would be awkward and the result much less interesting.
\smallskip
 
In the next two sections we discuss some further ramifications of our main results.

\subsection{Non-linear attachment rules}\label{sec:nonlin_result}

So far we have presented results for the case of affine attachment rules~$f$, given by
$f(k)=\gamma k + \beta$. While the fine details of the network behaviour often depend on the exact model 
definition, we expect the principal scaling and macroscopic features to be independent of these details. To 
investigate this \emph{universality} we now discuss to what extent Theorem~\ref{thm:pc} remains true when we 
look at more general non-linear attachment rules~$f$. 
\smallskip 

\pagebreak[3]

We consider two classes of attachment rules.
\begin{itemize}
\item[(1)] A function $f\colon \N_0 \to (0,\infty)$ is called a \emph{L-class} attachment rule if there exists $\gamma \in [0,1)$ and $0< \beta_l \le \beta_u$ such that $\gamma k +\beta_l \le f(k) \le \gamma k  +\beta_u$ for all $k$. Note that the parameter $\gamma$ for a L-class rule is uniquely defined by~\eqref{gammadef}.
\item[(2)]
A concave function $f \colon  \N_0 \to (0,\infty)$ with $\gamma:=\lim_{k \to \infty} f(k)/k \in [0,1)$ is called a \emph{C-class} attachment rule.
Note that concavity of $f$ implies that the limit above exists and that $f$ is non-decreasing. 
\end{itemize}

The asymptotic slope of the attachment rule determines the key features of the model. For example,  Dereich and M\"orters~\cite{DerMoe09}
show that for certain $C$-class attachment rules with $\gamma >0$ the asymptotic degree distribution is a power law with exponent $\tau=1+1/\gamma$. 
The following theorem shows that $\gamma$ also determines the scaling of the critical percolation probability for the damaged network .

\begin{theorem}{\bf{(Loss of connectivity, non-linear case)}}\label{thm:genf}
Let $f$ be a L-class or C-class attachment rule. For all $\eps \in (0,1)$,  
\[
\pc(\eps):=\inf\big\{p\colon (\grapheps_n(p) \colon n\in\N) \text{ has a giant component}\big\}>0.
\]
Moreover, if $f$ is in the L-class and
\begin{alignat*}{3}
&\text{if }\gamma = \frac{1}{2}\text{ then}\qquad \qquad & \qquad &\lim_{\eps\downarrow 0}\frac{\log \pc(\eps)}{\log \log(1/\eps)} = -1,\qquad\qquad &&\qquad\\
&\text{if }\gamma > \frac{1}{2}\text{ then}\qquad \qquad &  \qquad&\lim_{\eps \downarrow 0}\frac{\log \pc(\eps)}{\log \eps} = \gamma -\frac{1}{2}.\qquad\qquad&& \qquad
\end{alignat*}\nopagebreak
If $f$ is in the C-class, the statement remains true in the case $\gamma>\frac12$, and in the case $\gamma=\frac{1}{2}$ if the limit is replaced by a 
$\limsup_{\eps \downarrow 0}$ and the equality by `$\le$'.
\end{theorem}
\pagebreak[3]

Theorem~\ref{thm:genf} implies that the damaged network $(\grapheps_n \colon n\in\N)$ is not robust. But as $\lim_{\eps \downarrow 0}\pc(\eps)= 0$
it is still `asymptotically robust' for $\eps \downarrow 0$ in the sense that when less than order $n$ old vertices are destroyed, the critical percolation
parameter remains zero. We formulate this as a corollary. For two graphs $\graph=(\vset,\eset)$ and $\tilde{\graph}=(\tilde{\vset},\tilde{\eset})$, we write $\graph \ge \tilde{\graph}$ if there is a coupling such that $\vset \supseteq \tilde{\vset}$ and $\eset \supseteq \tilde{\eset}$.

\begin{corollary}\label{cor:optimal}
Let $(m_n \colon n \in \N)$ be a sequence of natural numbers with $\lim_{n \to \infty} m_n/n=0$. The network $(\graph_n^{\sss(m_n)} \colon n\in\N)$,
consisting of the graphs  $\graph_n$ damaged by removal of the oldest $m_n$ vertices along with all adjacent edges, is robust.
\end{corollary}

\begin{proof} Let $p\in (0,1)$. By Theorem~\ref{thm:genf}, there exists $\eps>0$ such that $\pc(\eps)<p$. Choose $n_0 \in \N$ such that $m_n/n<\eps$ for all $n \ge n_0$. Then 
$\graph_n^{\sss(m_n)} \ge \grapheps_n$ for all $n \ge n_0$, implying $\graph_n^{\sss(m_n)}(p) 
\ge \grapheps_n(p)$. Since $(\grapheps_n(p) \colon n\in\N)$ has a giant component, so does $(\graph_n^{\sss(m_n)}(p) \colon n\in\N)$.
\end{proof}

Theorem~\ref{thm:genf} is derived from Theorem~\ref{thm:pc} using the monotonicity of the network in the attachment rule. Its appeal 
lies in the large class of functions to which it applies. The L-class attachment rules are all positive, bounded perturbations 
of linear functions. In Figure 1 we see several examples: On the left a concave function which is also in the C-class, then a convex function and a function which is convex in one and concave in another part of its domain. The latter 
examples are not monotone, and all three are asymptotically vanishing perturbations of an affine attachment rule. The example of an 
L-class attachment rule on the right shows that this may also fail.

\begin{center}
\includegraphics[scale=0.5]{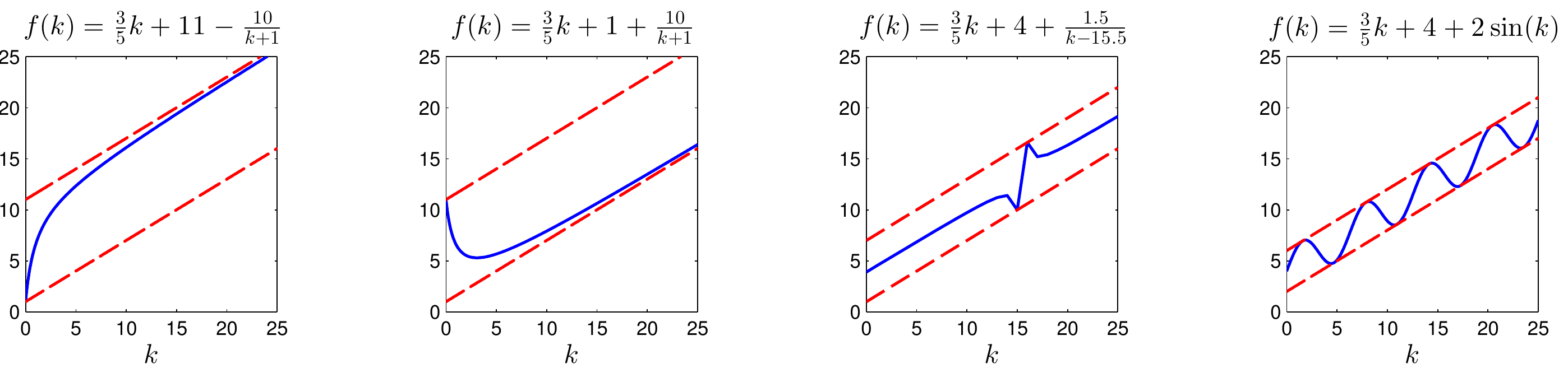}\\
{\small {\bf{Figure 1:}} Examples for L-class attachment rules. The blue curve is the attachment rule, the red, dashed lines are linear lower and upper bounds.}
\end{center}
%[scale=0.4, height=6cm]{Lclass_examples.pdf}
The C-class attachment rules are always non-decreasing as positive concave functions and always have a linear lower bound with the same asymptotic slope $\gamma$ as the function itself. However, when the perturbation $k \mapsto f(k)-\gamma k$ is not bounded, then there exists no linear function with slope $\gamma$ which is an upper bound to the attachment rule; any linear upper bound will be steeper. Two examples are displayed in Figure 2.

\begin{center}
\includegraphics[scale=0.5]{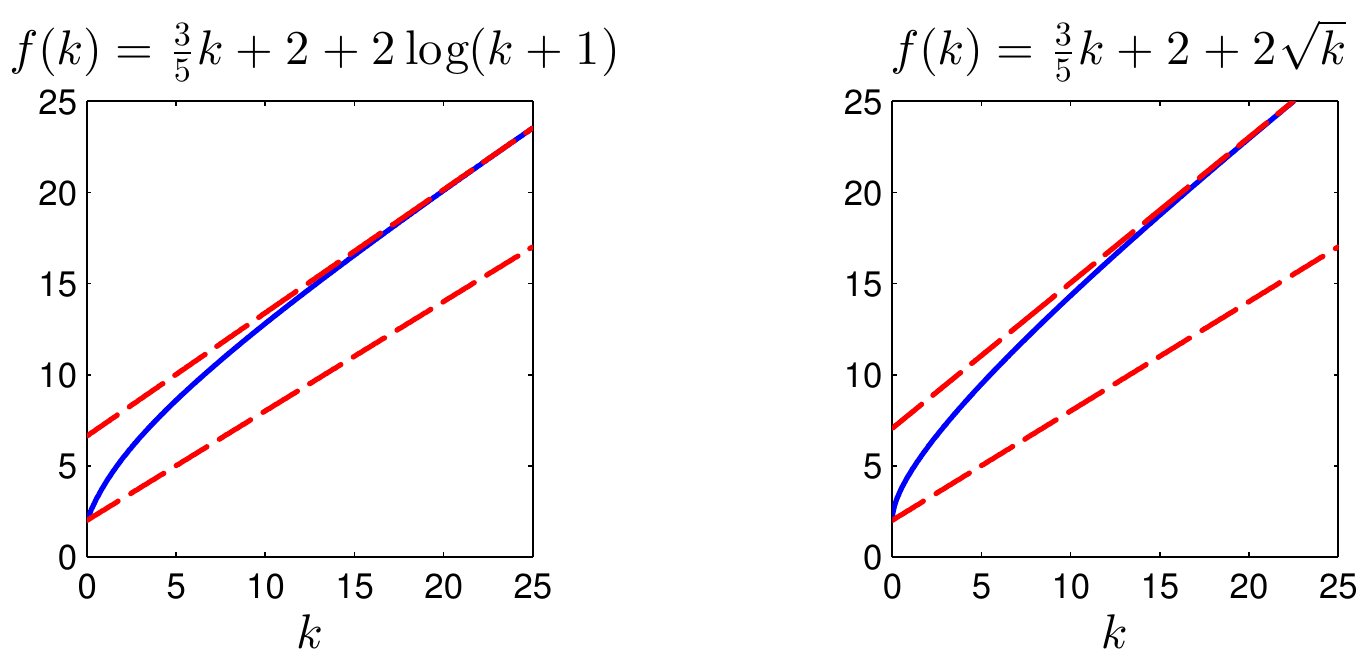}\\
\small{ {\bf{Figure 2:}} Examples for C-class attachment rules. The blue curve is the attachment rule, the red, dashed lines are linear lower and upper bounds. The slope of the upper bound is strictly larger than $\gamma$.}
\end{center}

\subsection{Vulnerability of other network models}\label{ss-15}

We would like to investigate to what extent our results are common to robust random network models rather than specific to preferential attachment networks. Again our focus is on 
Theorem~\ref{thm:pc} and we look at two types of networks, the configuration model and the inhomogeneous random graphs. Both types have an explicit static description 
and are therefore much easier to analyse than the preferential attachment networks studied in our main theorems.

\subsubsection{Configuration model} \label{sec:CM_results}

A targeted attack can be planned particularly well when the degree sequence of the network is known. A random graph model with fixed degree sequence is given by the configuration model. For $n \in \N$, let ${\bf{d_n}}=(d_i^n)_{i=1}^n \in \N^n$ with $\sum_{i=1}^n d_i^n$ even. To simplify notation, we write $d_i$ instead of $d_i^n$. The multigraph $\graph_n^{\sss (\CM)}$ on vertex set $\{1,\ldots,n\}$ is constructed as follows: to every vertex~$i$ attach $d_i$ half-edges. Combine the half-edges into pairs by a uniformly random matching of the set of all half-edges. Each pair of half-edges is then joined to form an edge of $\graph_n^{\sss (\CM)}$.
The configuration model received a lot of attention in the literature, see \cite{Hof13pre} and the references therein. 
A good targeted attack in the configuration model is the removal of the vertices with the highest degree and we denote by $\graph_n^{{\sss(\CM)},\eps}$ the network after removal of the $\fst$ vertices with the largest degree.
\smallskip  

Let $n_k=|\{i \le n\colon d_i=k\}|$ be the number of vertices with degree $k$ and assume that there exists a $\N$-valued random variable $D$ with $0<\E D<\infty$ and $\P(D=2)<1$, such that
\begin{equation}\label{as:CM}
n_k/n \to \P(D=k) \quad \mbox{ for all} \; k \in \N \qquad \text{and}\qquad \frac{1}{n}\sum_{k=1}^{\infty} kn_k \to \E D \qquad \text{as }n\to \infty.
\end{equation}
In particular, the law of $D$ is the weak limit of the empirical degree distribution in $(\graph_n^{\sss (\CM)}\colon n \in \N)$ and the network is robust if $\E[D^2]=\infty$. Our focus is on the case that the distribution of $D$ is a power law with exponent $\tau=1+1/\gamma$, $\gamma \ge \frac{1}{2}$.

\begin{theorem}\label{thm:CM}
Let $\eps \in (0,1)$, $\gamma \ge \frac{1}{2}$ and suppose there is a constant $C>0$ such $\P(D > k) \sim C k^{-1/\gamma}$ as $k\to \infty$. Then there exists $\pc^{\sss (\CM)}(\eps)>0$ such that
\[
(\graph_n^{{\sss (\CM)},\eps}(p) \colon n \in \N) \text{ has a giant component } \Leftrightarrow \, p>\pc^{\sss (CM)}(\eps).
\]
Moreover,
\[
\pc^{\sss (\CM)}(\eps) \asymp \begin{cases} 
\frac{1}{\log(1/\eps)} & \quad \text{if }\gamma =\frac{1}{2},\\
\eps^{2\gamma -1} & \quad \text{if }\gamma >\frac{1}{2}.\end{cases}
\]
\end{theorem}

We observe the same basic phenomenon as in the corresponding preferential attachment models: While the undamaged network is robust, after removal
of an arbitrarily small proportion of privileged nodes the network becomes vulnerable to random removal of vertices. 
However, when $\gamma>\frac{1}{2}$, then the increase of the critical percolation parameter $\pc(\eps)$ as $\eps$ leaves zero is less steep than in the corresponding preferential attachment model.

Note that we make the assumption $0<\E D<\infty$. In the case  $\E D=\infty$  Bhamidi et al.\ \cite{BhaHofHoo10}
show a more extreme form of vulnerability, where the connected network can be disconnected with high probability by deleting a bounded number of vertices.
%\smallskip

\subsubsection{Inhomogeneous random graphs} \label{sec:IRGresult}
Inhomogeneous random graphs are a generalization of the classical Erd\H{o}s-R\'{e}nyi random graph. Let $\kappa \colon (0,1]\times (0,1] 
\to (0,\infty)$ be a symmetric kernel. The inhomogeneous random graph $\graph_n^{\sss (\kappa)}$ corresponding to kernel $\kappa$ has the vertex set $\vset_n =\{1,\ldots,n\}$ and any pair of distinct vertices $i$ and
$j$ is connected by an edge independently with probability
\begin{equation} \label{def:IRRG}
\P(\{i,j\} \text{ present in }\graph_n^{\sss (\kappa)}) = \tfrac{1}{n} \kappa\big(\tfrac{i}{n},\tfrac{j}{n}\big) \land 1.
\end{equation}
Many features of this class of models are discussed by Bollob{\'a}s, Janson and Riordan~\cite{BolJanRio07} 
and van der Hofstad~\cite{Hof13pre}.
The first inhomogeneous random graph model we consider is a version of the \emph{Chung-Lu model}, see for example~\cite{CL02, CL06}.
The relevant kernel is
\[
\kappa^{\sss(\CL)}(x,y) =x^{-\gamma} y^{-\gamma} \qquad\mbox{ for }  \, x,y \in (0,1].
\]
This is an example of a kernel of the form $\kappa(x,y)=\psi(x) \psi(y)$, for some $\psi$, which are called kernels of \emph{rank one}, see \cite{BolJanRio07}. Note that a similar factorisation occurs in the configuration model since the probability that vertices $i$ and $j$ are directly connected is roughly proportional to $d_i d_j$.
The network corresponding to $\kappa^{\sss (\CL)}$ has an asymptotic degree distribution which is a power law with exponent $\tau=1+1/\gamma$. 
\smallskip

\pagebreak[3]

The second inhomogeneous random graph model we consider is chosen such that the edge probabilities agree (at least asymptotically) with those in a preferential 
attachment network and the asymptotic degree distribution is a power law with exponent $\tau=1+1/\gamma$. The relevant kernel~is 
\[
\kappa^{\sss(\PA)}(x,y)=\frac{1}{(x \land y)^{\gamma} (x\vee y)^{1-\gamma}} \qquad\mbox{ for }  \, x,y \in (0,1].
\]
Note that if $\gamma\not=\sfrac{1}{2}$ this kernel is not of rank one, but \emph{strongly inhomogeneous}. These two kernels allow us to demonstrate the difference between \emph{factorising 
models} (like the configuration model) and preferential attachment models within one model class. 

We denote by $\graph_n^{\sss (\CL)}$, resp.~$\graph_n^{\sss (\PA)}$ the 
inhomogeneous random graphs with kernel~$\kappa^{\sss(\CL)}$, resp.~kernel $\kappa^{\sss(\PA)}$. If $\gamma \ge \frac{1}{2}$, then $(\graph_n^{\sss (\CL)}\colon n\in\N)$ and 
$(\graph_n^{\sss (\PA)}\colon n\in\N)$ are robust.  Since the kernels $\kappa^{\sss(\CL)}$ and $\kappa^{\sss(\PA)}$ are decreasing in both components, vertices with small labels 
are favoured in the corresponding models.
We denote by $\graph_n^{{\sss (\CL)},\eps}$, resp.~$\graph_n^{{\sss (\PA)},\eps}$, what remains of the graph $\graph_n^{\sss (\CL)}$, resp.~$\graph_n^{\sss (\PA)}$, after removal 
of all vertices with label at most $\eps n$ along with their adjacent edges.

The following theorem confirms that, like in the preferential attachment and in the configuration model, the removal of a positive fraction of key vertices makes the networks vulnerable to random removal of nodes.
Notice that the $\kappa^{\sss (\CL)}$ and $\kappa^{\sss (\PA)}$ agree for 
$\gamma =\frac{1}{2}$ so that we only have to state a result for $\graph_n^{{\sss (\CL)},\eps}$ in this regime.

\begin{theorem}\label{thm:IRG}
Let $\gamma \ge \frac{1}{2}$, $\star \in \{\CL,\PA\}$ and $\eps\in (0,1)$. There exists $\pc^{\sss(\star)}(\eps)>0$ such that
\begin{equation*} 
(\graph_n^{{\sss (\star)},\eps}(p)\colon n \in \N) \text{ has a giant component } \Leftrightarrow \; p> \pc^{\sss(\star)}(\eps).
\end{equation*}
Moreover, 
\begin{align*}
&\pc^{\sss (\CL)}(\eps)=\begin{cases}
\frac{1}{\log(1/\eps)} & \text{if }\gamma=\frac{1}{2},\\
(2\gamma -1) \eps^{2\gamma-1} [1+O(\eps^{2\gamma-1})] & \text{if } \gamma>\frac{1}{2},
\end{cases}\end{align*}
and
\[
\pc^{\sss(\PA)}(\eps) \asymp \eps^{\gamma -1/2}\qquad \; \qquad \qquad  \qquad \qquad \text{ if }\gamma >\sfrac{1}{2}.
\]
\end{theorem}

The fact that the Chung-Lu model is vulnerable to targeted attacks has also been remarked by van der Hofstad in Section~9.1 of \cite{Hof13pre}.
\smallskip

Summarising, we note that vulnerability to a targeted attack is a universal feature of robust networks, holding not only for preferential attachment networks but also for configuration 
models and various classes of inhomogeneous random graphs. In the case $2<\tau<3$, studying the asymptotic behaviour of the critical percolation parameter~$\pc(\eps)$ as a function 
of the proportion $\eps$  of removed vertices reveals \emph{two universality classes} of networks, distinguished by the critical exponent measuring the polynomial rate of decay of~$\pc(\eps)$ 
as $\eps\downarrow 0$. In terms of the power law exponent~$\tau$ this critical exponent equals \smash{$\frac{3-\tau}{\tau-1}$} in the case of the configuration model and the Chung-Lu model, but is 
only half this value 
in the case of preferential attachment networks and inhomogeneous random graphs with a strongly inhomogeneous kernel.  The same classification of networks has emerged in a different context 
in~\cite{DerMoeMoe12}, where it was noted that the typical distances in networks of the two classes differ by a factor of two. 
The key feature of the configuration model and the rank one inhomogeneous random graphs seems to be that the connection probability of two vertices factorises.
By contrast, the connection probabilities in preferential attachment networks have a more complex structure giving privileged nodes a 
stronger advantage.

\pagebreak[3]

\subsection{Main ideas of the proofs}\label{sec:idea}

Dereich and M\"orters \cite{DerMoe13} have shown that the (not too large) graph neighbourhood of a uniformly chosen vertex in $\graph_n$ 
can be coupled to a branching random walk on the negative halfline. Although we cannot make direct use of these coupling results in our proofs, 
it is helpful to formulate our ideas in this framework, which we briefly explain now. Points on the halfline correspond to vertices with the  origin 
representing the youngest vertex  and the distance between vertex~$i-1$ and vertex~$i$ represented as~$1/i$, so that as the
graph size increases vertices become dense everywhere on the halfline, and the limiting object becomes continuous. In this representation, 
the vertices with index  $1,\ldots, \lfloor\eps n\rfloor$, which we remove when damaging  the network, are asymptotically located to the left of the point $\log \eps$.
Hence, the graph neighbourhood of a uniformly chosen vertex in $\grapheps_n$ can be coupled to a branching random walk in which particles are killed 
(together with their offspring) when leaving the interval~$[\log\eps,0]$. As in~\cite{DerMoe13} the survival probability of this branching process is equal 
to the asymptotic relative size~$\zeta$ of the largest component. This allows us to determine, for example, the critical parameter for percolation from 
knowledge when the percolated branching process has a positive survival probability.%
\medskip%

\pagebreak[3]
It is instructive to continue the comparison of the damaged 
and undamaged networks in the setup of this branching process. In \cite{DerMoe13}, where  the undamaged network is analysed, the branching random walk 
has only \emph{one} killing boundary on the right. It turns out that on the set of survival the 
leftmost particle drifts away from the killing boundary, such that it does not feel the boundary anymore. As a consequence, the unkilled process 
carries all information needed to determine whether the killed branching random walk survives with positive probability and, therefore, whether 
the network has a giant component. The \emph{two} killing boundaries in the branching random walk describing the damaged network prevent us 
from using this analogy; every particle is exposed to the threat of absorption.
\medskip

To survive indefinitely, a genealogical line of descent has to move within the (space-time) strip \linebreak $[\log\eps,0]\times \N_0$. To understand the optimal 
strategy for survival observe that, in the network with strong preferential attachment, old vertices typically have a large degree and therefore are connected to 
many young vertices, while young vertices themselves have only a few connections. This means that in the branching random walk without killing, particles produce 
many offspring to the right, but only a few to the left. Hence, if a particle is located near the left killing boundary it is very fertile, but its 
offspring are mostly located further to the right and are therefore less fertile. A particle near the right killing boundary, however, has itself a small number 
of offspring, which then however have a good chance of being fertile. As a result, the optimal survival strategy for a particle is to have an ancestral line of 
particles whose locations are alternating between positions near the left and the right killing boundary.
This intuition is the basis for our proofs.
\medskip

Continuing more formally, for the proof of Theorem~\ref{thm:pc} it can be shown that positivity of the survival probability can be characterised in terms 
of the largest eigenvalue~$\rho_{\eps}$ of an operator that describes the spatial distribution of offspring of a given particle. More precisely, the branching 
random walk survives percolation with retention parameter~$p$ if its growth rate $p\rho_{\eps}$ exceeds the value one, so that $\pc(\eps)=1/\rho_{\eps}$. Our intuition allows us to 
guess the form  of the corresponding eigenfunction, which has its mass concentrated in two bumps near the left and right killing boundary. From this guess we obtain sufficiently 
accurate estimates for the largest eigenvalue, and therefore for the critical percolation parameter, as long as the preferential attachment effect is 
strong enough. This is the case if $\gamma\geq \frac12$, allowing us to prove Theorem~\ref{thm:pc}.
\medskip

By contrast, for $\gamma <\tfrac{1}{2}$ we know that the network is not robust, 
i.e.\ we have $\pc(0)>0$. It would be of interest to understand the behaviour of $\pc(\eps)-\pc(0)$ as $\eps \downarrow 0$. Our methods can be applied to this case, but the resulting
bounds are  very rough. The reason is that in this regime the preferential attachment is much weaker, and the intuitive idea underlying our estimates gives a less accurate picture.
\medskip

The idea of proof in Theorem~\ref{thm:distances} is based on the branching process comparison, too.  To bound the probability that two typical vertices $V$ and $W$ are connected by 
a path of length at most~$h$, we look at the expected number of such paths. This is given by the number of vertices in a ball of radius $h-1$ 
around $V$ in the graph metric, multiplied with the probability that such a vertex connects to~$W$. By our branching process heuristics, the size of the ball is 
also equal to the number of vertices in the first $h-1$  generations of the branching random walk, which can be determined as the largest eigenvalue~$\rho_{\eps}=1/\pc(\eps)$ of the associated operator to the power $h$. The probability of connecting 
any vertex with index at least $\eps n$ to $W$ is bounded by $f(m)/\eps n$, where $m$ is the maximal degree in the network. Since $m=o(n)$ by Theorem~\ref{thm:degrees}, this implies that the probability of a connection is bounded from above
by $\exp(h \log (1/\pc(\eps))- \log n +o(\log n))$ and therefore goes to zero if $h\le (1-\delta)\log n/\log (1/\pc(\eps))$, $\delta>0$, which yields the result.
%Note that this uses a bound on the maximal degree in the network, which follows easily from Theorem~\ref{thm:degrees}.
\medskip

\pagebreak[3]

Theorem~\ref{thm:degrees} is relatively soft by comparison. The independence of the indegrees of distinct vertices allows us to study them separately and we again use the continuous approximation to describe the expected empirical indegree evolution. The limit theorem for the empirical distribution itself follows from a standard concentration argument. The asymptotic result for the maximal degrees is only slightly more involved and is based on fairly standard extreme value arguments.

\subsection{Overview}

The outline of this article is as follows. We start with the main steps of the proofs in Section~2. The multitype branching process which locally approximates a connected component in the network is defined in Section~\ref{sec:BP_def} and its key properties are stated. The main part of the proof of Theorem~\ref{thm:pc} then follows in Section~\ref{sec:Linf}. The analysis of the multitype branching process is conducted in Section~\ref{sec:BP}. We study an operator associated with the process in Section~\ref{sec:A} and derive necessary and sufficient conditions for its survival in Section~\ref{sec:survival}.  Sections~\ref{sec:degrees} and \ref{sec:distances} are devoted to the study of the topology of the damaged graph. In Section~\ref{sec:degrees} the typical and maximal degree of vertices is analysed. In Section~\ref{sec:distances} typical distances are studied.  The couplings between the network and the approximating branching process that underlie our proofs
are provided in Section~\ref{sec:coup}.  We then look at model variations in~Section~5. The derivation of Theorem~\ref{thm:genf} from Theorem~\ref{thm:pc} is presented in Section~\ref{sec:Genf}. This is the only section which requires consideration of non-linear attachment rules. We finish in Section~\ref{sec:others} by studying the question of vulnerability in other network models.

\section{Connectivity and branching processes}

In this section, we restrict our attention to linear attachment rules $f(k)=\gamma k +\beta$, for $\gamma \in [0,1)$ and $\beta >0$, and let $\eps$ be a fixed value in $(0,1)$.
The goal of this section is to prove Theorem~\ref{thm:pc}. To this end, we couple the local neighbourhood of a vertex in $\grapheps_n$ to a multitype branching process. 
The branching process is introduced in Section~\ref{sec:BP_def} and Theorem~\ref{thm:pc} is deduced in Section~\ref{sec:Linf}. Properties of the branching processes which are 
needed in the analysis are proved in Section~\ref{sec:BP}. The proof of the coupling between network and branching process is deferred to Section~\ref{sec:coup}.

\subsection{The approximating branching process}\label{sec:BP_def}

Let $Z=(Z_t\colon t \ge 0)$ be a pure jump Markov process %started in zero
with generator
\[
Lg(k):=f(k)\big( g(k+1)-g(k)\big).
\]
That is, $Z$ is an increasing, integer-valued process, which jumps from $k$ to $k+1$ after an exponential waiting time with mean $1/f(k)$, independently of the previous jumps. 
We write $P$ for the distribution of $Z$ started in zero and $E$ for the corresponding expectation. By $(\hat{Z}_t\colon t\ge 0)$ we denote a version of the process started in $\hat{Z}_0=1$ under the measure $P$.
\smallskip

The process~$Z$ is used to define the offspring distributions of a multitype branching process with type space $\typespace:=[\log \eps,0]\times \{\leftp,\rightp\}$. A typical element of $\typespace$ is denoted by $\type=(\loc,\mar)$. The intuitive picture 
is that $\loc$ encodes the spatial position of the particle which we call \emph{location}. The second coordinate $\mar$ indicates which side of the particle its parent is located and we refer to $\mar$ as the \emph{mark}. The non-numerical symbols $\leftp$ and $\rightp$ stand for `left' and `right', respectively.
A particle of type $(\loc,\mar)\in \typespace$ produces offspring to its left of mark $\rightp$ with relative positions having the same distribution as those points of the Poisson point process $\Pi$ with intensity measure
\begin{equation}\label{eq:intensityPi}
e^{t} E[f(Z_{-t})]\mathbbm{1}_{(-\infty,0]}(t) \, dt= \beta e^{(1-\gamma) t} \mathbbm{1}_{(-\infty,0]}(t) \, dt
\end{equation}
that lie in $[\log \eps-\loc,0]$. 
The distribution of the offspring to the right depends on the mark of the parent. When the particle is of type $(\loc,\leftp)$, then the relative positions of the offspring follow the same distribution as the jump times of $(Z_t\colon t \in [0, -\loc])$, but when the particle is of type $(\loc,\rightp)$, then the relative positions follow the same distribution as the jump times of $(\hat{Z}_t\colon t \in [0,-\loc])$. All offspring on the right are of mark $\leftp$. Observe that the chosen offspring distributions ensure that new particles have again a location in $[\log \eps,0]$.
The offspring distribution to the right is \emph{not} a Poisson point process: The more particles are born, the higher the rate at which new particles arrive. 
\pagebreak[3]
\smallskip

We call the branching process thus constructed the \emph{idealized branching process} ($\IBP$). 
It can be interpreted as a labelled tree, where every node represents a particle and is connected to its children and (apart from the root) to its parent. We equip node $x$ 
with label $\type(x)=(\loc(x),\mar(x))$, where $\loc(x)$ denotes its location and $\mar(x)$ its mark, and write $|x|$ for the generation of $x$.
We show in Section~\ref{sec:coup} that the genealogical tree of the $\IBP$ is closely related to the neighbourhood of a uniformly chosen vertex in $\grapheps_n$. To obtain a branching process approximation to $\grapheps_n(p)$, we define the \emph{percolated $\IBP$} by associating to every offspring in the $\IBP$ an independent Bernoulli$(p)$ random variable. If the random variable is zero, we delete the offspring together with its line of descent. If it equals one, then the offspring is retained in the percolated $\IBP$.
\smallskip

Let $S^{\eps }$ be a random variable with
\begin{equation}\label{def:Seps}
{\P}(-S^{\eps}\le t)=\frac{1}{1-\eps} \big(1- e^{-t}\big), \qquad \mbox{ for } t \in [0,-\log \eps].
\end{equation}
Denote by $\zeta^{\eps}(p)$ the survival probability of the tree which is with probability $p$ equal to
the percolated $\IBP$ started with one particle of mark $\leftp$ and location $S^{\eps }$ and equals the 
empty tree otherwise.  Let $\connect_n^{\eps}(p)$ be a connected component in $\grapheps_n(p)$ of maximal size.
 
\begin{theorem}\label{thm:size_giant} For all $\eps \in (0,1)$ and $p \in (0,1]$, in probability,
\[
\frac{|\connect_n^{\eps}(p)|}{\E|\vseteps_n(p)|}  \to \zeta^{\eps}(p)/p \qquad \text{as }n\to \infty.
\]
\end{theorem}

The proof of Theorem~\ref{thm:size_giant} is postponed to Section~\ref{sec:coup}. The theorem offers a description of the largest component in the network in terms of the survival probability of the percolated $\IBP$. To make use of this connection, we have to understand the branching process.
\smallskip

For any bounded, measurable function $g$ on $\typespace$, and $\type \in \typespace$, let
\[
A_pg(\type):=E_{\type,p}\Big[\sum_{|x|=1}g(\loc(x),\mar(x))\Big],
\]
where the expectation $E_{\type,p}$ refers to the percolated $\IBP$ starting with a single particle of type $\type$, percolated with retention parameter~$p$.
We write $A=A_1$ for the operator corresponding to the unpercolated branching process and $E_{\type}:=E_{\type,1}$. Recall that all quantities associated with the $\IBP$, and in 
particular~$A_p$, depend on the fixed value of $\eps$. We denote by $C(\typespace)$ the complex Banach space of continuous functions on $\typespace$ equipped with the 
supremum norm. The following proposition, which summarizes properties of~$A_p$, is proved in Section~\ref{sec:A}.

\begin{proposition}\label{pro:A}
For all $\eps \in (0,1)$ and $p \in (0,1]$, the operator $A_p\colon C(\typespace) \to C(\typespace)$ is linear, strictly positive and compact with spectral radius $\rho_{\eps}(A_p) \in (0,\infty)$. Moreover, $A_p=pA$ and $\rho_{\eps}(A_p)=p\rho_{\eps}(A)$.
\end{proposition}

We have the following characterization of the survival probability of the percolated $\IBP$.

\begin{theorem}\label{survival_vs_spectral}
For all $\eps \in (0,1)$ and $p \in (0,1]$
\[
\zeta^{\eps}(p)>0 \qquad \Leftrightarrow \qquad \rho_{\eps}(A_p)>1.
\]
\end{theorem}

Theorem~\ref{survival_vs_spectral} is proved in Section~\ref{sec:survival}. Combined with Theorem~\ref{thm:size_giant} and Proposition~\ref{pro:A} it gives a characterisation of the critical percolation parameter for $(\grapheps_n(p)\colon n\in\N)$.

\begin{corollary}
The network $(\grapheps_n(p) \colon n\in\N)$ has a giant component if and only if $p>\rho_{\eps}(A)^{-1}$.
\end{corollary}

Notice that the corollary implies that $(\grapheps_n\colon n\in\N)$ has no giant component when $\rho_{\eps}(A)\le 1$. Moreover, the first statement of Theorem~\ref{thm:pc} 
follows from the corollary by taking $\pc(\eps)=\rho_{\eps}(A)^{-1}\land 1$.
\smallskip

To complete the proof of Theorem~\ref{thm:pc}, it remains to estimate the spectral radius $\rho_{\eps}(A)$. We use that 
(see, e.g., Theorem~45.1 in \cite{Heu82}) for a linear and bounded operator $A$ on a complex Banach space
\begin{equation}\label{eq:spectral}
\rho_{\eps}(A)=\lim_{n\to \infty} \|A^n\|^{\frac{1}{n}} =\inf\{\|A^n\|^{\frac{1}{n}}\colon n \in \N\}.
\end{equation}
By the definition of the Poisson point process $\Pi$ in \eqref{eq:intensityPi}, the intensity measure of $\Pi$ equals 
\[
\mathbbm{1}_{(-\infty,0]}(t)\, M(dt),\qquad \mbox{ for } M(dt):=\beta e^{(1-\gamma)t} \, dt.
\] 
We denote by $\Pi^{\leftp}$ the point process given by the jump times of $(Z_t\colon t \ge 0)$ and by $\Pi^{\rightp}$ the point process given by the jump times of $(\hat{Z}_t\colon t\ge 0)$.
A simple computation (cf.\ Lemma~1.12 in \cite{DerMoe13}) shows that with $M^{\mar}(dt):=a_{\mar}e^{\gamma t}\, dt$, where $a_{\leftp}=\beta $ and $a_{\rightp}=\gamma +\beta$, the intensity measure of $\Pi^{\mar}$ is given by $\mathbbm{1}_{[0,\infty)}(t)M^{\mar}(dt)$ for $\mar \in \{\leftp,\rightp\}$. Hence, for any bounded, measurable function $g$ on $\typespace$ and $(\loc,\mar) \in \typespace$, 
%denoting by $E_{(\loc,\mar)}$ expectation with respect to an IBP started with a single particle of type~$(\loc,\mar)$,
\begin{equation}\label{eq:A_explicit}
Ag(\loc,\mar)=E_{(\loc,\mar)}\Big[\sum_{|x|=1} g(\loc(x),\mar(x))\Big]
%&= E\Big[ \sum_{\heap{\p \in \Pi}{\log \eps -\loc \le \p}} g(\loc+\p,\rightp) +\sum_{\heap{\q \in \Pi^{\mar}}{\q \le -\loc}} g(\loc+\q,\leftp)\Big]\\
%&= \int_{\log \eps -\loc}^0 g(\loc+t,\rightp) \, M(dt) +\int_0^{-\loc} g(\loc+t,\leftp)\, M^{\mar}(dt)\\
= \int_{\log \eps -\loc}^0 g(\loc+t,\rightp) \beta e^{(1-\gamma)t}\,dt +\int_0^{-\loc} g(\loc+t,\leftp) a_{\mar} e^{\gamma t}\,dt.
\end{equation}

\subsection{Proof of Theorem~\ref{thm:pc}}\label{sec:Linf}

Subject to the considerations of the previous section, Theorem~\ref{thm:pc} follows from the following proposition.

\begin{proposition}\label{pro:pc}
\begin{itemize}
\item[(a)] If $\gamma = \frac{1}{2}$, then
$\displaystyle
\frac{1}{\gamma+\beta} \frac{1}{\log(1/\eps)} \le \rho_{\eps}(A)^{-1} \le \frac{1}{\beta} \frac{1}{\log(1/\eps)}.$
\item[(b)] If $\gamma >\frac{1}{2}$, then
\[
\Big(1+\log (\eps^{1-2\gamma})\eps^{\gamma-1/2}+\big[\log (\eps^{1-2\gamma})\eps^{\gamma-1/2}\big]^2\Big)^{-1/2} \le \sfrac{\sqrt{\beta (\gamma+\beta)}}{2\gamma-1} \eps^{-\gamma+1/2} \rho_{\eps}(A)^{-1} \le \big(1-\eps^{\gamma-1/2}\big)^{-1}.
\]
\end{itemize}
\end{proposition}
\smallskip

\begin{proof}[Proof of Proposition~\ref{pro:pc}\,(a)]
For $h_0 \in C([\log \eps,0])$, the space of continuous functions on $[\log \eps,0]$, let
\[
\Asimple h_0(\loc):=\int_{\log\eps-\loc}^0 h_0(\loc+t) e^{t/2} \, dt+\int_0^{-\loc} h_0(\loc+t) e^{t/2} \, dt\qquad \mbox{ for } \, \loc\in [\log \eps,0].
\]
Defining $h(\loc,\mar):=h_0(\loc)$ for all $(\loc,\mar) \in \typespace$, equation \eqref{eq:A_explicit} yields
\[
\beta \Asimple h_0(\loc) \le A h(\loc,\mar) \le (\gamma+\beta) \Asimple h_0(\loc) \qquad \mbox{ if }  h_0 \ge 0, (\loc,\mar) \in \typespace.
\]
In particular, by the monotonicity and linearity of $A$ and $\Asimple$,
\[
\beta^n \Asimple^n\mathbf{1}(\loc) \le A^n\mathbf{1}(\loc,\mar) \le (\gamma+\beta)^n \Asimple^n\mathbf{1}(\loc) \qquad \mbox{ for }  \, (\loc,\mar)\in \typespace, n \in \N,
\]
where $\mathbf{1}$ denotes the constant function with value $1$. From~\eqref{eq:spectral}, we obtain 
\[
\rho_{\eps}(A) \in [\beta\rho_{\eps}(\Asimple), (\gamma+\beta)\rho_{\eps}(\Asimple)].
\]
To complete the proof it suffices to show that $\rho_{\eps}(\Asimple)=\log(1/\eps)$, which we can achieve by `guessing' the 
principal eigenfunction of $\Asimple$. Indeed, the result follows from (\ref{eq:spectral}) and
\[ 
\Asimple^{n+1}\mathbf{1}(\loc)=2(1-\eps^{1/2})(\log(1/\eps))^{n} e^{-\loc/2} \qquad \mbox{ for }  \, \loc\in [\log\eps,0], n \in \N_0.
\]
We show the latter identity by induction over $n$. For $n=0$,
\[
\Asimple\mathbf{1}(\loc)=\int_{\log\eps-\loc}^0 e^{t/2}\, dt+\int_0^{-\loc} e^{t/2} \, dt=2(1-e^{-\loc/2}\eps^{1/2}+e^{-\loc/2}-1)=2(1-\eps^{1/2})e^{-\loc/2}.
\]
Moreover, with $h_0(\loc):=e^{-\loc/2}$ we have
\[
\Asimple h_0(\loc)=\int_{\log\eps-\loc}^0 e^{-(\loc+t)/2}e^{t/2} \, dt +\int_0^{-\loc} e^{-(\loc+t)/2} e^{t/2} \, dt= e^{-\loc/2} \log(1/\eps).
\]
Thus, $\loc \mapsto e^{-\loc/2}$ is an eigenfunction of $\Asimple$ with eigenvalue $\log(1/\eps)$ and the proof is complete.
\end{proof}

\begin{proof}[Proof of lower bound in Proposition~\ref{pro:pc}\,(b)]
We analyse the ancestral lines of particles in the branching process.  Going back two steps in the ancestral line of every particle alive
we can divide the population in four groups depending on the positioning of children relative to their parents: (1) in both steps the child is to the left of its parent, (2) in the first step the child
is to the left and in the second it is to the right of its parent, (3) first right, then left, (4) in both steps the child is to the right of its parent.
The cases are depicted in Figure~3.
\begin{center}
\begin{tikzpicture}

\foreach \i in {1, 2, 3, 4}
{\draw (4*\i-4,0) -- (4*\i-1,0);
\node at (4*\i-2.5,-0.5) {$B_{\i}$};
\node at (4*\i-4,-0.25) {\footnotesize{$\log \eps$}};
\node at (4*\i-1,-0.25) {\footnotesize{$0$}};
\draw (4*\i-4,0) -- (4*\i-4,0.1);
\draw (4*\i-1,0) -- (4*\i-1,0.1);
}

\draw[->, >=stealth] (2.5,0) to [bend right=35] (1.5,0);
\draw[->, >=stealth] (1.5,0) to [bend right=35] (0.5,0);

\draw[->, >=stealth] (5.5,0) to [bend right=25] (4.5,0);
\draw[->, >=stealth] (4.5,0) to [bend left=35] (6.5,0);

\draw[->, >=stealth] (9.5,0) to [bend left=25] (10.5,0);
\draw[->, >=stealth] (10.5,0) to [bend right=35] (8.5,0);

\draw[->, >=stealth] (12.5,0) to [bend left=35] (13.5,0);
\draw[->, >=stealth] (13.5,0) to [bend left=35] (14.5,0);

\end{tikzpicture}
{\small{{\bf{Figure 3:}} Possible genealogy of a particle contributing to the respective operators.}}
\end{center}

We denote by $B_i$, $i \in \{1,\ldots,4\}$, the operators corresponding to the four scenarios. Using the point processes $\Pi$, $\Pi^{\leftp}$ and $\Pi^{\rightp}$ this means,
for any bounded, measurable function $g$ on $\typespace$ and $(\loc,\mar) \in \typespace$,
\begin{align*}
B_1g(\loc,\mar)&:=\E\Big[\sum_{\heap{\p \in \Pi}{ \log \eps-\loc \le \p}}\sum_{\heap{\q \in \Pi}{\log\eps -\loc-\p\le \q}} g(\loc+\p+\q,\rightp)\Big], &
B_2g(\loc,\mar)&:=\E\Big[\sum_{\heap{\p \in \Pi}{\log \eps-\loc \le \p}}\sum_{\heap{\q \in \Pi^{\rightp}}{\q\le -(\loc+\p)}} g(\loc+\p+\q,\leftp)\Big],\\
B_3g(\loc,\mar)&:=\E\Big[\sum_{\heap{\p \in \Pi^{\mar}}{\p\le -\loc}}\sum_{\heap{\q \in \Pi}{\log\eps -\loc-\p\le \q}} g(\loc+\p+\q,\rightp)\Big], &
B_4g(\loc,\mar)&:=\E\Big[\sum_{\heap{\p \in \Pi^{\mar}}{\p\le -\loc}}\sum_{\heap{\q \in \Pi^{\leftp}}{\q\le -(\loc+\p)}} g(\loc+\p+\q,\leftp)\Big].
\end{align*}
Intuitively, going back the ancestral line of a typical particle in the population at a late time, for a few generations the ancestral particles 
may be in group~(4),  because of the high fertility of particles positioned near the left boundary of $[\log \eps,0]$. But this behaviour is not sustainable, as 
after a few generations in this group the offspring particle will typically be near the right end of the interval and will therefore
be pushed into the killing boundary  so that it is likely to die out. Over a longer period the ancestral particles are much more likely to be
in groups (2) and (3), as this behaviour is sustainable over long periods when the ancestral line is hopping more or less regularly between 
positions near the left and the right boundary of the interval~$[\log \eps,0]$. Our aim is now to turn this heuristics into useful bounds on high 
iterates of the operator~$A$.
\smallskip

It is useful to understand how the operators $B_i$ act on the constant function $\mathbf{1}$ as wells as on the functions $g_1(\loc,\mar):=e^{-\gamma \loc}$ and 
$g_2(\loc,\mar):=e^{-(1-\gamma)\loc}$. We can write 
\[
B_3g(\loc,\mar)=\int_0^{-\loc} \int_{\log \eps -\loc-t}^0 g(\loc+t+s,\rightp) \, M(ds)\, M^{\mar}(dt),
\]
where $M(dt)=\beta e^{(1-\gamma)t}\,dt$ and $M^{\mar}(dt)=a_{\mar} e^{\gamma t} \, dt$ with $a_{\mar}\le \gamma +\beta$
are the intensity measures of the point processes $\Pi$ and $\Pi^{\mar}$. From this we obtain, for $(\loc,\mar) \in \typespace$,
\begin{align*}
&B_3 \mathbf{1}(\loc,\mar)\le \int_{-\infty}^{-\loc} \int_{-\infty}^{0} \, M(ds)\, M^{\rightp}(dt)= \frac{\beta (\gamma+\beta)}{\gamma (1-\gamma)} e^{-\gamma \loc},\\
&B_3 g_1(\loc,\mar)\le \int_{-\infty}^{-\loc} \int_{\log \eps -t-\loc}^{\infty} e^{-\gamma(\loc+t+s)} \, M(ds)\, M^{\rightp}(dt)= \frac{\beta (\gamma+\beta)}{(2\gamma-1)^2} \eps^{1-2\gamma} e^{-\gamma \loc},\\
&B_3 g_2(\loc,\mar)\le \int_{-\infty}^{-\loc} \int_{\log \eps}^{0} e^{-(1-\gamma)(\loc+t+s)} \, M(ds)\, M^{\rightp}(dt)= \frac{\beta (\gamma+\beta)}{2\gamma-1} \log(1/\eps) e^{-\gamma \loc}.
\end{align*}
Moreover, similarly elementary calculations for $B_1$, $B_2$ and $B_4$ imply
\begin{align*}
B_1 \mathbf{1}(\loc,\mar)\le \sfrac{\beta^2}{(1-\gamma)^2},\quad
B_2 \mathbf{1}(\loc,\mar)\le\sfrac{\beta (\gamma+\beta)}{\gamma (2\gamma-1)} \eps^{1-2\gamma} e^{-(1-\gamma)\loc},\quad
B_4 \mathbf{1}(\loc,\mar)\le\sfrac{\beta (\gamma+\beta)}{\gamma} \log(1/\eps) e^{-\gamma \loc};
\end{align*}
and
\begin{align*}
& B_1 g_1(\loc,\mar)\le \sfrac{\beta^2}{2\gamma-1}\log(1/\eps) \eps^{1-2\gamma} e^{-(1-\gamma)\loc},
& &B_1 g_2(\loc,\mar)\le \beta^2 (\log \eps)^2 e^{-(1-\gamma)\loc},\\
& B_2 g_1(\loc,\mar)\le \sfrac{\beta (\gamma+\beta)}{2\gamma-1} \log(1/\eps)\eps^{1-2\gamma} e^{-(1-\gamma)\loc},
& &B_2 g_2(\loc,\mar)\le \sfrac{\beta (\gamma+\beta)}{(2\gamma-1)^2} \eps^{1-2\gamma} e^{-(1-\gamma)\loc},\\
& B_4 g_1(\loc,\mar)\le \beta (\gamma+\beta)(\log \eps)^2 e^{-\gamma \loc},
& &B_4 g_2(\loc,\mar)\le \sfrac{\beta (\gamma+\beta)}{2\gamma-1} \log(1/\eps) e^{-\gamma \loc}.
\end{align*}
Summarising, there exists $C_{\eps}>0$ such that $B_i\mathbf{1}(\type) \le C_{\eps} g_1(\type)$ for all $i \in \{1,\ldots,4\}$, $\type \in \typespace$,
and denoting
\[
b_{\rm sm}:=b_1:=b_4:= \beta (\gamma+\beta) (\log\eps)^2, \qquad b_{\rm bg}:=b_2:=b_3:= \tfrac{\beta (\gamma+\beta)}{(2\gamma-1)^2}\eps^{1-2\gamma},
\]
where bg stands for `big' and sm for `small', we have 
\begin{align*}
B_ig_1(\type)&\le b_{\rm bg} \log (\eps^{1-2\gamma})g_2(\type),\quad & B_ig_2(\type)&\le b_i g_2(\type) \quad &&\text{for }i\in \{1,2\}, \\
B_ig_1(\type)&\le b_i g_1(\type),\quad &B_ig_2(\type)&\le b_{\rm bg} \log (\eps^{1-2\gamma})\eps^{2\gamma-1} g_1(\type)\quad  &&\text{for }i\in \{3,4\}.
\end{align*}
Using that by definition $A^2=\sum_{i=1}^4 B_i$, our estimate for $B_i \mathbf{1}$ yields
\begin{equation}\label{1stAest}
A^{2(n+1)}\mathbf{1}(\type)=\sum_{i_0,\ldots,i_{n}\in \{1,\ldots,4\}}B_{i_{n}}\circ \cdots \circ B_{i_0}\mathbf{1}(\type)\le 4 C_{\eps} \sum_{i_1,\ldots,i_{n}\in \{1,\ldots,4\}} B_{i_n}\circ \cdots \circ B_{i_1}g_1(\type).
\end{equation}
Up to constants, the estimates for $B_3$ and $B_4$ preserve $g_1$ but change $g_2$ into $g_1$, whereas the estimates for $B_1$ and $B_2$ preserve $g_2$ and change $g_1$ into $g_2$. Hence, we split the sequence of indices into blocks containing only $1$ or $2$ and blocks containing only $3$ or $4$. We write $m$ for the number of blocks, $k_j$ for the length of block $j$ and $\bar{k}_j:=\sum_{i=1}^{j-1} k_i+1$ for the first index in block $j$. Then
\begin{align*}
\sum_{i_1,\ldots,i_{n}\in \{1,\ldots,4\}} B_{i_n}\circ \cdots \circ B_{i_1}g_1(\type)= \sum_{m=1}^{n+1} \sum_{\substack{k_1+\ldots+k_m=n\\ k_1 \in \N_0, k_2,\ldots,k_m \in \N}} \sum_{(i_1,\ldots,i_n)} B_{i_n}\circ \cdots \circ B_{i_1}g_1(\type),
\end{align*}
where the last sum is over all sequences of indices $(i_1,\ldots,i_n)$ which satisfy $i_{\bar{k}_j},\ldots ,i_{\bar{k}_{j+1}-1} \in \{3,4\}$ for $j$ odd and $i_{\bar{k}_j},\ldots, i_{\bar{k}_{j+1}-1} \in \{1,2\}$ for $j$ even. We insist that formally the first block contains the indices $3$ or $4$ --- the case that this does not hold is covered by $k_1=0$. Hence, in the first block, operators $B_3$ and $B_4$ encounter $g_1$, which is preserved. To determine the constants, we only have to keep track of how often $B_4$ is used; we call this number $l_1$. The first operator belonging to
a new block $j$ causes a factor $b_{\rm bg}\log (\eps^{1-2\gamma})$ and if the change is from a $\{1,2\}$ to a $\{3,4\}$ block, then an additional $\eps^{2\gamma-1}$ is obtained. For the subsequent
steps within block $j$, we again have to track how often the operator causing the smaller constant $b_{\rm sm}$, $B_1$ or $B_4$, is used. This number is called $l_j$. After applying all $n$ operators,
the function $g_1(\type) \mathbbm{1}_{\text{odd}}(m)+g_2(\type) \mathbbm{1}_{\text{even}}(m)$ remains and we bound it by $\eps^{-\gamma}$. This procedure yields
\begin{align*}
&\sum_{i_1,\ldots,i_{n}\in \{1,\ldots,4\}} B_{i_n}\circ \cdots \circ B_{i_1}g_1(\type)\\
&\le \sum_{m=1}^{n+1} \sum_{\substack{k_1+\ldots+k_m=n\\ k_1 \in \N_0, k_2,\ldots,k_m \in \N}} b_{\rm bg}^{m-1} (\log (\eps^{1-2\gamma}))^{m-1} \eps^{(2\gamma-1)(\lceil \frac{m}{2}\rceil-1)} \eps^{-\gamma}\\
&\phantom{\le \sum_{m=1}^{n+1} \sum_{\substack{k_1+\ldots+k_m=n\\ k_1 \in \N_0, k_2,\ldots,k_m \in \N}}} \sum_{l_1=0}^{k_1}\Big[ \binom{k_1}{l_1} b_{\rm sm}^{l_1} b_{\rm bg}^{k_1-l_1}\Big]\prod_{j=2}^m \sum_{l_j=0}^{k_j-1}\Big[ \binom{k_j-1}{l_j} b_{\rm sm}^{l_j} b_{\rm bg}^{k_j-1-l_j} \Big]\\
&=\eps^{-\gamma} \sum_{m=1}^{n+1} \sum_{\substack{k_1+\ldots+k_m=n\\ k_1 \in \N_0, k_2,\ldots,k_m \in \N}} b_{\rm bg}^{m-1} (\log (\eps^{1-2\gamma}))^{m-1} \eps^{(2\gamma-1)(\lceil \frac{m}{2}\rceil-1)} (b_{\rm sm}+b_{\rm bg})^{k_1} \prod_{j=2}^m (b_{\rm sm}+ b_{\rm bg})^{k_j-1}\\
&=\eps^{-\gamma} \sum_{m=1}^{n+1} \sum_{\substack{k_1+\ldots+k_m=n\\ k_1 \in \N_0, k_2,\ldots,k_m \in \N}}b_{\rm bg}^{m-1}(\log (\eps^{1-2\gamma}))^{m-1} \eps^{(2\gamma-1)(\lceil \frac{m}{2}\rceil-1)} (b_{\rm sm}+b_{\rm bg})^{n-(m-1)}.
\end{align*}
Given $m$, the number of configurations $k_1 \in \N_0, k_2,\ldots,k_m \in \N$ with $k_1+\ldots+k_m=n$ is the number of arrangements of $m-1$ dividers and
$n-(m-1)$ balls, which equals $\binom{n}{m-1}$. 
%To determine the number of summands, notice that
%\begin{align*}
%&|\{(k_1,\ldots,k_m)\in \N_0 \times \N^{m-1}\colon \sum_{i=1}^m k_i=n\}|\\
%&= \big|\big\{(k_2,\ldots,k_m)\in \big(\N_0\big)^{m-1}\colon \sum_{i=2}^m k_i=n-(m-1)\big\}\big|+\big|\big\{(k_1,\ldots,k_m)\in \big(\N_0\big)^{m}\colon \sum_{i=1}^m k_i=n-m\big\}\big|\\
%&= \binom{n-(m-1)+(m-2)}{m-2}+\binom{n-m+(m-1)}{m-1}=\binom{n}{m-1}.
%\end{align*}
Since $\lceil \frac{m}{2}\rceil -1 \ge \frac{m-1}{2}-\frac{1}{2}$, an application of the binomial theorem yields
\begin{equation}\label{2ndAest}
%\begin{split}
\sum_{i_1,\ldots,i_{n}\in \{1,\ldots,4\}} B_{i_n}\circ  \cdots \circ B_{i_1}g_1(\type)
%\le \eps^{-2\gamma+1/2} \sum_{m=1}^{n+1} \binom{n}{m-1} b_{\rm bg}^{m-1} (\log \eps^{1-2\gamma})^{m-1} \eps^{(2\gamma-1)\frac{m-1}{2}} (b_{\rm bg}+b_{\rm sm})^{n-(m-1)}\\
\le  \eps^{-2\gamma+1/2} \big(b_{\rm bg} \log(\eps^{1-2\gamma}) \eps^{\gamma-1/2}+b_{\rm bg}+b_{\rm sm}\big)^n.
%\end{split}
\end{equation}
Combining \eqref{1stAest} and \eqref{2ndAest}, we conclude that for all $\type \in \typespace$
\begin{align*}
A^{2(n+1)}\mathbf{1}(\type)&\le 4 C_{\eps} \eps^{-2\gamma+1/2} b_{\rm bg}^n \big(\log(\eps^{1-2\gamma}) \eps^{\gamma-1/2}+1+\tfrac{b_{\rm sm}}{b_{\rm bg}}\big)^n.
\end{align*}
Now (\ref{eq:spectral}) yields, for all $\eps \in (0,1)$,
\[
\rho_{\eps}(A)^{-1}\ge \sfrac{2\gamma -1}{\sqrt{\beta (\gamma+\beta)}} \eps^{\gamma-1/2} \Big(1+\log (\eps^{1-2\gamma})\eps^{\gamma-1/2}+\big[
\log (\eps^{1-2\gamma})\eps^{\gamma-1/2}\big]^2\Big)^{-1/2}.\qedhere
\]
\end{proof}

The insight gained in the proof of the lower bound, enables us to `guess' an approximating eigenfunction, which is the main ingredient in the proof of the upper bound.

\begin{proof}[Proof of upper bound in Proposition~\ref{pro:pc}\,(b)]
Let $c_{\rightp}:=1$ and $c_{\leftp}:=\beta/(\gamma+\beta)$ and, for $(\loc,\mar) \in \typespace$, let 
\[
g_{{\rm{e}}}(\loc,\mar):=c_{\mar}\eps^{\gamma} e^{-\gamma \loc}  \mathbbm{1}_{[\log \eps, \frac{\log \eps}{2}]}(\loc) +\sqrt{\beta / (\gamma+\beta)} \eps^{1/2} e^{-(1-\gamma)\loc} \mathbbm{1}_{(\frac{\log \eps}{2},0]}(\loc).
\]
Notice that $a_{\mar}/c_{\mar}=\gamma +\beta$ for $\mar \in \{\leftp,\rightp\}$. If $(\loc,\mar) \in [\log \eps,\frac{\log \eps}{2}]\times \{\leftp,\rightp\}$, then
\begin{align*}
Ag_{{\rm{e}}}(\loc,\mar) &\ge E_{(\loc,\mar)}\Big[\sum_{\heap{|x|=1}{\loc(x) >\frac{\log \eps}{2}}} g_{{\rm{e}}}(\loc(x),\leftp)\Big] 
=a_{\mar}\sqrt{\beta/(\gamma+\beta)} \eps^{1/2} e^{-(1-\gamma)\loc}\int_{-\loc+\frac{\log \eps}{2}}^{-\loc} e^{(2\gamma-1)t } \, dt\\
& = c_{\mar}\sfrac{\sqrt{\beta (\gamma+\beta)}}{2\gamma-1} \eps^{1/2}  e^{-\gamma \loc} \big[1-\eps^{\gamma-1/2}\big]
= \sfrac{\sqrt{\beta (\gamma+\beta)}}{2\gamma-1} \eps^{-\gamma+1/2}  \big[1-\eps^{\gamma-1/2}\big] g_{\rm{e}}(\loc,\mar).
\end{align*}
If $(\loc,\mar) \in (\frac{\log \eps}{2},0]\times \{\leftp,\rightp\}$, then
\begin{align*}
Ag_{{\rm{e}}}(\loc,\mar)&\ge E_{(\loc,\mar)}\Big[\sum_{\heap{|x|=1}{\loc(x) \le \frac{\log \eps}{2}}} g_{{\rm{e}}}(\loc(x),\rightp)\Big]
=\beta c_{\rightp} \eps^{\gamma} e^{-\gamma \loc} \int_{-\loc+\log \eps}^{-\loc+\frac{\log \eps}{2}} e^{(1-2\gamma)t} \,dt\\
& =\sfrac{\beta}{2\gamma-1} \eps^{\gamma} e^{-\gamma \loc}  e^{(1-2\gamma) (\log \eps-\loc)} \big[1-e^{-(1-2\gamma)\frac{\log \eps}{2}}\big]
= \sfrac{\sqrt{\beta (\gamma+\beta)}}{2\gamma-1} \eps^{-\gamma+1/2}  [1-\eps^{\gamma-1/2}] g_{\rm{e}}(\loc,\mar).
\end{align*}
By monotonicity of $A$ this implies
\[
\|A^n\|\ge \big(\sfrac{\sqrt{\beta (\gamma+\beta)}}{2\gamma-1} \eps^{1/2-\gamma}[1-\eps^{\gamma -1/2}]\big)^n.
\] 
Taking the $n$-th root on both sides, an application of (\ref{eq:spectral}) yields the required bound for~$\rho_{\eps}(A)$.
\end{proof}

\subsection{A multitype branching process}\label{sec:BP}

In this section, we analyse the $\IBP$ and its relation to the associated operator~$A$. We begin by providing properties of~$A$ in Section~\ref{sec:A}, and then use these properties 
to prove necessary and sufficient conditions for the multitype branching process to survive with positive probability in Section~\ref{sec:survival}. Throughout, we use the notation introduced in Section~\ref{sec:BP_def} and  write $P_{\type,p}$ for the distribution of the percolated $\IBP$ with retention probability $p$ started with one particle of type $\type \in \typespace$,
abbreviating~$P_{\type}:=P_{\type,1}$.

\subsubsection{Proof of Proposition~\ref{pro:A}}\label{sec:A}

\begin{lemma}\label{lem:pos}
For all nonnegative $g \in C(\typespace)$ with $g\not\equiv 0$, we have $\displaystyle\min_{\type \in \typespace}A^2g(\type)>0$.
\end{lemma}

\begin{proof}
If $g \in C(\typespace)$, $g\ge 0$, $g \not\equiv 0$, then there exist $\log \eps \le \loc_1<\loc_2\le 0$ and $\mar_0 \in \{\leftp, \rightp\}$ such that $g$ is strictly positive on $[\loc_1,\loc_2]\times \{\mar_0\}$. Hence, it suffices to show that
\[
\min_{\type\in \typespace} P_{\type}\big(\exists x \colon |x|=2, \type(x) \in [\loc_1,\loc_2]\times\{\mar_0\}\big)>0.
\]
By the definition of the process any particle produces offspring in a given interval of positive length with, uniformly in the start type, strictly positive probability. 
The two steps allow the time needed to ensure that the relative position of the parent satisfies~$\alpha(x)=\mar_0$. 
\end{proof}

\begin{lemma} \label{lem:compact}
The operator $A\colon C(\typespace)\to C(\typespace)$ is compact.
\end{lemma} 

\begin{proof} According to (\ref{eq:A_explicit}), we can write for $g \in C(\typespace)$ and $(\loc,\mar) \in \typespace$,
\[
Ag(\loc,\mar) =\int_{\log \eps}^{0} g(t,\rightp) \kappa_{\leftp}(\loc,t)\, dt+ \int_{\log \eps }^0 g(t,\leftp) \kappa_{\rightp}(\loc,\mar,t) \,dt,
\]
with $\kappa_{\leftp}(\loc,t)=\mathbbm{1}_{[\log \eps,\loc]}(t) \beta e^{(1-\gamma)(t-\loc)}$ and $\kappa_{\rightp}(\loc,\mar,t)=\mathbbm{1}_{[\loc,0]}(t) a_{\mar} e^{\gamma (t-\loc)}$. 
Thus $A$ can be written as the sum of two operators, which are both compact by the Arzel\`a-Ascoli theorem. 
\end{proof}

We summarize some standard properties of compact, positive operators in the following proposition. 

\begin{proposition} \label{pro:decom}
Let $X$ be a complex Banach space and $A\colon X \to X$ be a linear, compact and strictly positive operator. 
\begin{itemize}
\item[(i)] The spectral radius of $A$, $\rho=\rho(A)$, is a strictly positive eigenvalue of $A$ with one dimensional eigenspace, generated by a strictly positive eigenvector $\eigenfct$. Eigenvalue $\rho$ is also the spectral radius of adjoint $A^*$ and the corresponding eigenspace is generated by a strictly positive eigenvector~$\nu_0$.
We rescale $\eigenfct$ and $\nu_0$ such that $\|\eigenfct\|=1$ and $\nu_0(\eigenfct)=1$ to make the choice unique.
\item[(ii)] There exists $\theta_0\in [0,\rho)$ such that $|\theta|\le \theta_0$ for all $\theta \in \sigma(A)\setminus\{\rho\}$, where $\sigma(A)$ is the spectrum of~$A$.
\item[(iii)] For any $\theta > \theta_0$ and $g \in X$, we have
$A^ng=\rho^n \nu_0(g) \eigenfct+O(\theta^n).$
\end{itemize}
\end{proposition}

\begin{proof} Statements~(i) and~(ii) are immediate from the Krein-Rutman theorem, see Theorem\ 3.1.3 (ii) in \cite{Pin95}, and
the general form of the spectrum of compact operators. Statement (iii) then follows from the spectral decomposition of a compact operator on a
complex Banach space. See, for example, Heuser~\cite{Heu82}, and there in particular Theorem~49.1 and Proposition~50.1.
\end{proof}

Now all results are collected to establish Proposition~\ref{pro:A}.

\begin{proof}[Proof of Proposition~\ref{pro:A}]
Identity $A_p=pA$ holds by definition and implies $\rho_{\eps}(A_p)=p\rho_{\eps}(A)$. Moreover, it is clear that it suffices to prove the first sentence of the statement for $p=1$. Linearity is immediate from the definition, positivity was shown in Lemma~\ref{lem:pos} and compactness is the content of Lemma~\ref{lem:compact}. The positive spectral radius follows immediately from 
Proposition~\ref{pro:decom} (i).
\end{proof}

\subsubsection{Proof of Theorem~\ref{survival_vs_spectral}}\label{sec:survival}

We start with a moment estimate for the total number of offspring of a particle. In the sequel, we write $|\IBP_n|$ for the number of particles in generation $n$ of the $\IBP$.

\begin{lemma} \label{lem:L2}We have $\sup_{\type\in \typespace}E_{\type}\big[|\IBP_1|^2\big]<\infty$.
\end{lemma}

\begin{proof}
Let $\Pi$, $Z$ and $\hat{Z}$ be independent realisations of the Poisson point process and the pure jump processes defined in Section~\ref{sec:BP_def}. Let $\type =(\loc,\mar) \in \typespace$. By the definition of the $\IBP$, 
\[
|\IBP_1|\overset{d}{=}\begin{cases}
\Pi([\log \eps-\loc,0])+Z_{-\loc} & \text{if }\mar=\leftp,\\
\Pi([\log \eps-\loc,0])+\hat{Z}_{-\loc} & \text{if }\mar = \rightp,\end{cases}
\]
where $\overset{d}{=}$ denotes distributional equality. Since $f$ is non-decreasing, $\hat{Z}$ stochastically dominates $Z$. In combination with Young's inequality this fact implies, 
for all $\type \in \typespace$,
\[
E_{\type}\big[|\IBP_1|^2\big] \le 2\Big( \E\big[\Pi([\log \eps,0])^2\big]+E\big[(\hat{Z}_{-\log \eps})^2\big]\Big).
\]
The first term on the right is finite because $\Pi$ is a Poisson point process with finite intensity measure. The second summand was computed in Lemma~1.12 of \cite{DerMoe13} and found to be finite.
\end{proof}

The next result is a classical fact about branching processes. We give a proof since we could not find a reference for the result in sufficient generality. See Theorem~III.11.2 in \cite{H63} for a special case.

\begin{lemma} \label{lem:noBound} For all $p \in [0,1]$, $N \in \N$ and $\type \in \typespace$,
\[
P_{\type,p}\big(1\le |\IBP_n| \le N \text{ infinitely often}\big)=0.
\]
\end{lemma}

\begin{proof}
We split the proof in two steps. First we show that $\delta:=\inf_{\type \in \typespace}P_{\type,p}(|\IBP_1| =0)>0$, then we conclude the statement from this result. 
By definition of the percolated IBP, for all $(\loc,\mar) \in \typespace$,
\begin{align*}
P_{(\loc,\mar),p}\big(|\IBP_1|=0\big)\ge P_{(\loc,\mar),1}\big(|\IBP_1|=0\big) &=\begin{cases}
P\big(\{ \Pi([\log \eps-\loc,0])=0 \}\cap\{ Z_{-\loc}=0\}\big) & \text{if }\mar=\leftp\\
P\big(\{ \Pi([\log \eps-\loc,0])=0 \}\cap\{ \hat{Z}_{-\loc}=0\}\big) & \text{if }\mar=\rightp \end{cases}\\
&\ge P\big(\Pi([\log \eps,0])=0\big) P\big(\hat{Z}_{-\log \eps}=0\big)>0.
\end{align*}
Since the lower bound is independent of $(\loc,\mar)$, the claim~$\delta>0$ is proved.

For the second step of the proof, we set $p=1$ to simplify notation. The proof for general $p$ is identical. Fix $N \in \N$, set $\tau_0:=0$ and, for $k \ge 1$, 
let $\tau_k:=\inf\{n > \tau_{k-1}\colon |\IBP_n| \in [1,N]\}$, where $\inf \emptyset :=\infty$. The strong Markov property implies, for all $\type \in \typespace$ 
and $k \in \N$,
\[
P_{\type}(\tau_k <\infty) \le P_{\type}(\tau_1 <\infty) \,\sup_{\nu} P_{\nu}(\tau_1<\infty)^{k-1},
\]
where the supremum is over all counting measure $\nu$ on $\typespace$ such that $\nu(\typespace) \in [1,N]$. Under $P_\nu$, $\nu=\sum_{i=1}^n \delta_{\type_i}$,  the branching process is started 
with $n$ particles of types $\type_1,\ldots,\type_n$. When all original ancestors have no offspring in the first generation, then the branching process suffers immediate extinction and 
$\tau_1=\infty$. Hence, for all such $\nu$,
\[
P_{\nu}(\tau_1<\infty)=1-P_{\nu}(\tau_1=\infty) \le 1- P_{\nu}\big(|\IBP_1|=0\big)\le 1-\delta^{\nu(\typespace)} \le 1-\delta^N.  
\]
We conclude, for all $\type \in \typespace$
\begin{align*}
P_{\type}\big(1\le |\IBP_n| \le N \text{ infinitely often}\big)&= \lim_{k \to \infty} P_{\type}(\tau_k <\infty)\le \lim_{k \to \infty} \sup_{\nu} P_{\nu}(\tau_1<\infty)^{k-1} \\
&\le \lim_{k \to \infty} (1-\delta^N)^{k-1}=0.\qedhere
\end{align*}
\end{proof}

With Lemma~\ref{lem:noBound} at hand, we can prove Theorem~\ref{survival_vs_spectral}.

\begin{proof}[Proof of Theorem~\ref{survival_vs_spectral}] Throughout the proof, we write $\rho:=\rho_{\eps}(A_p)$.
First suppose $\rho\le 1$. By Lemma~\ref{lem:noBound} $P_{\type,p}(\lim_{n \to \infty}|\IBP_n|\in \{0,\infty\})=1$. By Proposition~\ref{pro:decom}\,(iii) the assumption $\rho\le 1$ implies that 
\[
E_{\type,p}\big[|\IBP_n|\big]=A_{p}^n\mathbf{1}(\type)=\rho^n \nu_0(\mathbf{1})\eigenfct(\type)+o(1).
\]
Hence, $\sup_{n \in \N} E_{\type,p}[|\IBP_n|]<\infty$ and we conclude that $\lim_{n \to \infty} |\IBP_n| =0$ $P_{\type,p}$-almost surely for all $\type \in \typespace$ and, therefore, $\zeta^{\eps}(p)=0$.\smallskip

Now suppose that $\rho>1$ and denote  $W_n=\frac{1}{\rho^n} \sum_{|x|=n} \eigenfct(\type(x))$ for $n \in \N$.
Then $(W_n\colon n \in \N)$ is under~$P_{\type,p}$ a nonnegative martingale with respect to the filtration generated by the branching process. Hence, $W:=\lim_{n \to \infty} W_n$ 
exists almost surely. Given Lemma \ref{lem:L2} Biggins and Kyprianou show in Theorem~1.1 of \cite{BK04} that  $E_{\type,p}[W]=1$ and therefore, $P_{\type,p}(W>0)>0$. 
This implies in particular that the branching process survives with positive probability irrespective of the start type.
\end{proof}

We now investigate continuity of the survival probability as a function of the attachment rule. For this purpose we emphasise 
dependence on $f$ by adding it as an additional argument to several quantities. 

\begin{lemma}\label{continuityZeta}
Let $p \in (0,1]$. Then $\lim_{\delta \downarrow 0} \zeta^{\eps}(p,f-\delta)=\zeta^{\eps}(p,f)$.
\end{lemma}

\begin{proof}
Observe that  there exists a natural coupling of the $\IBP(f)$ with the $\IBP(f-\delta)$ such that every particle in the $\IBP(f-\delta)$ is also present in the $\IBP(f)$
and, hence, $\zeta^{\eps}(p,f-\delta)$ is increasing as $\delta \downarrow 0$. We can therefore assume that $\zeta^{\eps}(p,f)>0$, that is $\rho(f):=\rho_{\eps}(A_p,f)>1$, and by the continuity of $A_p$ in the attachment rule, there exists $\delta_0>0$ such that $\rho_{\eps}(A_p,f-\delta_0)>1$. In the proof of Theorem \ref{survival_vs_spectral} we have seen that this implies that the $\IBP(f-\delta_0)$ survives with positive probability, irrespective of the start type, and similar to Lemma \ref{lem:pos}, we conclude
\begin{equation}\label{eq:uniformExt}
\inf_{\type \in \typespace}P_{\type,p}\big(\IBP(f-\delta_0) \text{ survives}\big)>0.
\end{equation}
Recall the definition of the martingale~$(W_n\colon n  \in \N)$ and its almost sure limit~$W$ from the proof of Theorem~\ref{survival_vs_spectral}, which satisfies $E_{\type,p}[W]=1$ and
\[
W=\frac{1}{\rho(f)} \sum_{|x|=1} W(\type(x)) \qquad P_{\type,p}\text{-almost surely},
\]
where, conditionally on the first generation, $(W(\type(x))\colon |x|=1)$ are independent copies of the random variable~$W$ under $P_{\type(x),p}$.
In particular, $\type\mapsto P_{\type,p}(W=0)$ is a fixed point of the operator $Hg(\type)=E_{\type,p}[\prod_{|x|=1}g(\type(x))]$ on the set of $[0,1]$-valued, measurable functions. As the only $[0,1]$-valued fixed points of $H$ are the constant function $\mathbf{1}$ and the extinction probability $\type \mapsto P_{\type,p}(\IBP(f)\text{ dies out})$, we deduce that $W>0$ 
almost surely on survival. 
Let $c>0$ and $N \in \N$. On the space of the coupling between $\IBP(f)$ and $\IBP(f-\delta)$,
\begin{align*}
\zeta^{\eps}(p,f)-&\zeta^{\eps}(p,f-\delta)=P(\IBP(f) \text{ survives}, \IBP(f-\delta)\text{ dies out})\\
&\le P\big(W \le c, \IBP(f) \text{ survives}\big)+P\Big(W>c, \exists n \ge N\colon |\IBP_n(f)| < \frac{c\rho(f)^n}{2 \max \eigenfct}\Big)\\
&+ P\Big(|\IBP_n(f)| \ge \frac{c\rho(f)^n}{2 \max \eigenfct}\, \forall n \ge N,\IBP(f-\delta) \text{ dies out}\Big)=:\Theta_1(c)+\Theta_2(c,N)+\Theta_3(c,N,\delta).
\end{align*}
Since the offspring distribution of an individual particle is continuous in $\delta$ uniformly on the type space, the probability that $\IBP(f)$ and $\IBP(f-\delta)$ agree until generation $N$ tends to one as $\delta \downarrow 0$. On this event, when $|\IBP_N(f)| \ge C \rho(f)^N$ for some $C>0$, then the probability that the $\IBP(f-\delta)$ subsequently dies out is bounded by
\[
\sup_{\type \in \typespace}P_{\type,p}\big(\IBP(f-\delta) \text{ dies out}\big)^{\lceil C \rho(f)^N \rceil}.
\]
By \eqref{eq:uniformExt}, this expression tends to zero as $N \to \infty$ when $\delta\le\delta_0$. Hence, for all $c>0$, 
\[
0 \le \limsup_{\delta \downarrow 0}\big( \zeta^{\eps}(p,f)-\zeta^{\eps}(p,f-\delta)\big)\le \Theta_1(c) +\limsup_{N \to \infty} \Theta_2(c,N).
\]
On the event $\{W>c\}\cap\{W_n \to W\}$, there is a finite stopping time $N_0$ such that $W_n \ge W/2$ for all $n \ge N_0$ and we deduce that $\rho(f)^{-n} |\IBP_n(f)| \ge W_n/\max \eigenfct \ge c/(2\max \eigenfct)$. Since $W_n$ converges to $W$ almost surely, we conclude that $\lim_{N \to \infty}\Theta_2(c,N)=0$. Finally, $\Theta_1(c)$ tends to zero as $c \downarrow 0$ because $W$ is positive on the event of survival.
\end{proof}

\section{The topology of the damaged graph}

We investigate the empirical indegree distribution and maximal indegree of the damaged network in Section~\ref{sec:degrees}, and typical distances in
Section~\ref{sec:distances}.

\subsection{Degrees}\label{sec:degrees}
The following result for the indegrees $\Z[m,n]$ is immediate from the network construction.

\begin{lemma}\label{lem:degreeCoup}
Fix $n \in \N$, then the random variables $(\Z[m,n] \colon  m \leq n)$ are independent. Fix $\hat{m} \in \N$ and let $(\Z^m[\hat{m},n]\colon m \leq n)$, be independent 
copies of the random variable $\Z[\hat{m},n]$.
\begin{itemize}
\item[(i)] There exists a coupling between $(\Z[m,n]\colon 1 \le m \le \hat{m})$ and $(\Z^m[\hat{m},n]\colon 1 \le m \le \hat{m})$ such that 
\[
\Z[m,n] \ge \Z^m[\hat{m},n]\qquad \mbox{for all }  1 \le m\le \hat{m}.
\]
\item[(ii)] There exists a coupling between $(\Z[m,n]\colon \hat{m} \le m \le n)$ and $(\Z^m[\hat{m},n]\colon \hat{m} \le m \le n)$ such that 
\[
\Z[m,n] \le \Z^m[\hat{m},n]\qquad \mbox{for all } \hat{m} \le m \le n.
\]
\end{itemize}
\end{lemma}
\noindent Our goal is to determine the asymptotic behaviour of $\max_{m \in \vseteps_n}\Z[m,n]$ and
\[
X_{\ge k}^{\eps}(n)= \frac{1}{n-\fst} \sum_{m=\fst+1}^n \mathbbm{1}_{\{k,k+1,\ldots\}}(\Z[m,n]).
\]
Lemma~\ref{lem:degreeCoup} allows us to replace the independent random variables in these sequences by groups of independent and identically distributed random variables.

Dereich and M\"orters observe in~\cite{DerMoe09}, see for example Corollary~4.3, that the indegrees in network $(\graph_n \colon n\in\N)$ are closely related to the pure jump process $(Z_t)_{t \ge 0}$. Since the indegrees are not altered by the targeted attack, the same holds in the damaged network $(\grapheps_n\colon n\in\N)$. %To describe a suitable coupling, we introduce further notation. 
Let $\psi(k):=\sum_{j=1}^{k-1} j^{-1}$ for all $k \in \N$, which we consider as a time change, mapping real time epochs $k$ to an `artificial time'~$\psi(k)$.  The artificial time spent by the process $\Z[m,\cdot]$ 
in state $i$ is 
\[
T_m[i]:=\sup\Big\{ \sum_{j=k}^{l} \frac{1}{j}\colon \Z[m,k]=i=\Z[m,l]\Big\}.
\]
Let $k \in \N_0$ and $n_k$ the last real time that $\Z[m,\cdot]$ spends in $k$, that is, $\Z[m,n_k]=k$, $\Z[m,n_k+1]=k+1$. Then
$\sum_{i=0}^k T_m[i]= \sum_{j=m}^{n_k} j^{-1}=\psi(n_k+1)-\psi(m)$. In particular,
\begin{equation}\label{eq:degViaT}
\Z[m,n] \le k \quad \Leftrightarrow \quad \sum_{i=0}^k T_m[i] \ge \psi(n+1)-\psi(m).
\end{equation}
By definition, there exists a sequence of independent random variables $(T[i]\colon i \in \N_0)$ such that $T[i]$ is exponentially distributed with mean $1/f(i)$ and\vspace{-1mm}
\[
Z_t \le k \quad \Leftrightarrow \quad \sum_{i=0}^k T[i] >t.
\]
The next lemma provides a coupling between the artificial times~$T_m[i]$ and the exponential times~$T[i]$. The proof is identical to the proof of Lemma~4.1 in \cite{DerMoe09} and is 
omitted. We denote by $\tau_{m,i}=\inf\{\psi(k)\colon \Z[m,k]=i\}$ the artificial first entrance time of $\Z[m,\cdot]$ into state $i$. If $\tau_{m,i}=\psi(k)$, we write $\triangle \tau_{m,i}=k^{-1}$.

\begin{lemma}\label{lem:timeCoup}
There exists a constant $\eta >0$ such that for all $m \in \N$ there is a coupling such that, for all $i \in \N_0$ with $f(i) \triangle \tau_{m,i} \le \frac12$,
\[
T[i] - \eta f(i) \triangle \tau_{m,i} \le T_m[i] \le T[i] +\triangle \tau_{m,i} \qquad \text{almost surely,}
\]
and the random variables $((T[i],T_m[i])\colon i \in \N_0)$ are independent.
\end{lemma}

\noindent
By definition $\triangle \tau_{m,i} \le m^{-1}$. Hence, Lemma~\ref{lem:timeCoup} yields a coupling such that when $f(k)/m \le \sfrac12 $, then
\[
\sum_{i=0}^k T[i] - k \eta f(k)/m \le \sum_{i=0}^k T_m[i] \le \sum_{i=0}^k T[i] +k/m.
\]
In particular, for $f(k)/m \le \sfrac{1}{2} $, the equivalence~\eqref{eq:degViaT} implies
\begin{align*}
P\Big(\sum_{i=0}^{k} T[i] \ge \psi(n+1)-\psi(m)+\eta k f(k)/m  \Big)&\le  \P(\Z[m,n] \le k)\\
& \le P\Big(\sum_{i=0}^{k} T[i] \ge \psi(n+1)-\psi(m)+ k/m  \Big).
\end{align*}
If $m\le n$ and $m-\vartheta n=O(1)$ for some $\vartheta \in (0,1]$, then $\psi(n+1)-\psi(m) =\sum_{j=m}^{n} j^{-1}=-\log \vartheta +o(1)$. Hence, for $m \le n$, $m-\vartheta n=O(1)$ and $k=O(\log n)$,
\begin{equation}\label{eq:degDistrApp}
\P(\Z[m,n]\le k)=P\Big(\sum_{i=0}^k T[i] \ge - \log \vartheta +o(1)\Big),
\end{equation}
where the random null sequence $o(1)$ is bounded by a deterministic null sequence of order $O({(\log n)^2}/{n})$.
\pagebreak[3]

We proceed by estimating the distribution function of $\sum_{i=0}^k T[i]$. 
The following identity for the incomplete beta function will be of use.

\begin{lemma}\label{lem:aux_sum}
Let $a>0$, $c>0$ and $k \in \N_0$. Then
\[
\sum_{i=0}^k \binom{k}{i} \frac{(-a)^i}{i+c} = a^{-c} \int_0^a x^{c-1} (1-x)^k \, dx.
\]
\end{lemma}

\begin{proof}
Denote the left-hand side by $\theta(k,a,c)$. For $x >0$, we have
\[
\frac{\partial}{\partial x} \big[x^c \theta(k,x,c)\big] = \frac{\partial}{\partial x}\Big[ \sum_{i=0}^k \binom{k}{i} (-1)^i \frac{x^{i+c}}{i+c}\Big] = \sum_{i=0}^k \binom{k}{i} (-1)^i x^{i+c-1}= x^{c-1} (1-x)^k.
\]
Integrating both sides between $0$ and $a$ and dividing by $a^{c}$, we obtain the claim.
\end{proof}

\begin{lemma}\label{lem:distZt}
For $k \in \N_0$ and $t \ge 0$,
\begin{equation}\label{eq:distZt}
P(Z_t \ge k+1)= P\Big(\sum_{i=0}^{k} T[i] \le t\Big) = B\big(k+1,\tfrac{\beta}{\gamma}\big)^{-1} \int_{e^{-\gamma t}}^1 x^{\frac{\beta}{\gamma}-1} (1-x)^k \, dx.
\end{equation}
\end{lemma}

\begin{proof}
Let $k \in \N_0$. The probability density for $\sum_{i=0}^k T[i]$ is given by (see for example Problem 12, Chapter 1 in \cite{Fel71})
\[
t \mapsto \sum_{i=0}^k  \frac{\prod_{j=0, j \not=i}^k f(j)}{\prod_{j=0, j \not= i}^{k} (f(j)-f(i))} f(i) e^{-f(i) t} \mathbbm{1}_{[0,\infty)}(t).
\]
Using $f(j)=\gamma j +\beta$, we can rewrite for all $i \in \{0,\ldots,k\}$,
\[
\frac{\prod_{j=0, j \not=i}^k f(j)}{\prod_{j=0, j \not= i}^{k} (f(j)-f(i))}= \frac{ \frac{\beta}{f(i)} \frac{k!}{k!} \prod_{j=1}^k f(j)}{\gamma^k (-1)^i i! (k-i)!}=\frac{\beta}{\gamma}\Big(\prod_{j=1}^k \frac{f(j)}{\gamma j}\Big) \binom{k}{i} \frac{(-1)^i}{i+\frac{\beta}{\gamma}}.
\]
We obtain
\[
P\Big(\sum_{i=0}^{k} T[i] \le t\Big)=\frac{\beta}{\gamma} \Big(\prod_{j=1}^k \frac{f(j)}{\gamma j}\Big) \sum_{i=0}^k  \binom{k}{i} \frac{(-1)^i}{i+\frac{\beta}{\gamma}} \big(1-e^{-\beta t} e^{-\gamma i t}\big).
\]
The factor in front of the sum equals $B(k+1,\frac{\beta}{\gamma})^{-1}$. Thus, it remains to show that the sum agrees with the integral in (\ref{eq:distZt}). Using Lemma~\ref{lem:aux_sum}, we derive
\begin{align*}
\sum_{i=0}^k  \binom{k}{i} \frac{(-1)^i}{i+\frac{\beta}{\gamma}} \big(1-e^{-\beta t} e^{-\gamma i t}\big)&= \int_0^1 x^{\frac{\beta}{\gamma}-1} (1-x)^k \, dx- e^{-\beta t} (e^{-\gamma t })^{-\frac{\beta}{\gamma}} \int_0^{e^{-\gamma t}} x^{\frac{\beta}{\gamma}-1} (1-x)^k \, dx\\
&= \int_{e^{-\gamma t}}^1 x^{\frac{\beta}{\gamma}-1} (1-x)^k \, dx. \qedhere
\end{align*}
\end{proof}

Now everything is prepared for the proof of Theorem~\ref{thm:degrees}. % We begin with the mean of the empirical degree distribution.

\begin{proposition}\label{pro:ExpDeg}
For all $k \in \N_0$
\[
\E\big[X_{\ge k+1}^{\eps}(n)\big]\to \int_{\eps}^1 \frac{1}{1-\eps} B\big(k+1,\tfrac{\beta}{\gamma}\big)^{-1} \int_{y^{\gamma}}^1 x^{\frac{\beta}{\gamma} -1} (1-x)^k \, dx \, dy \qquad \text{as }n \to \infty. 
\]
\end{proposition}

\begin{proof}
Let $\delta>0$, $\Delta=\lfloor \delta \eps n\rfloor$, $N=1+\lfloor \frac{n-\fst}{\Delta} \rfloor$, $m_j=\fst+1+j \Delta$ for $j=0,\ldots, N-1$, $m_N=n+1$, $\Delta_j=\Delta$ for $j=1,\ldots,N-1$ and $\Delta_N=m_N-m_{N-1} \in [0,\Delta)$. 
Combining Lemma~\ref{lem:degreeCoup} with \eqref{eq:degDistrApp}, we obtain
\begin{align*}
\frac{1}{n-\fst} \sum_{m=\fst+1}^n \P(\Z[m,n]>k) &\le \frac{1}{n-\fst} \sum_{j=0}^{N-1} \Delta_{j+1} \P(\Z[m_j,n] >k) \\
&\le \frac{1}{n -\fst} \sum_{j=0}^{N-1} \delta \eps n\, P\Big(\sum_{i=0}^k T[i] \le - \log(\eps +j \delta \eps) +o(1)\Big),
\end{align*}
and
\[
\limsup_{n \to \infty} \frac{1}{n-\fst} \sum_{m=\fst+1}^n \P(\Z[m,n]>k) \le \sum_{j=0}^{\lceil \frac{1-\eps}{\delta \eps} \rceil}  \frac{\delta \eps}{1-\eps} P\Big(\sum_{i=0}^k T[i] \le -\log\big(\eps +j \frac{\delta \eps}{1-\eps} (1-\eps)\big)\Big).
\]
Taking $\delta \to 0$ on the right-hand side, we conclude
\begin{equation}\label{ubAsympDeg}
\limsup_{n \to \infty} \frac{1}{n-\fst} \sum_{m=\fst+1}^n \P(\Z[m,n]>k) \le \int_0^1 P\Big(\sum_{i=0}^k T[i] \le -\log\big(\eps +y (1-\eps)\big)\Big) \, dy.
\end{equation}
Similarly,
\begin{align*}
\frac{1}{n-\fst} \sum_{m=\fst+1}^n \P(\Z[m,n]>k )& \ge \frac{1}{n -\fst} \sum_{j=1}^N \Delta_j \P(\Z[m_j-1,n]>k)\\
& \ge \frac{1}{n-\fst} \sum_{j=1}^{N-1} \Delta P\Big(\sum_{i=0}^k T[i] \le - \log(\eps +j \delta \eps)+o(1)\Big),
\end{align*}
and as above we see that $\liminf$ satisfies the reverse inequality in~\eqref{ubAsympDeg}. Lemma~\ref{lem:distZt} yields the claim.
\end{proof}

Dereich and M\"orters (pp~1238--1239 in \cite{DerMoe09}) give a simple argument based on Chernoff's inequality to upgrade the convergence of the expected empirical degree distribution to convergence of the empirical degree distribution itself. The proof remains valid for the damaged network and is therefore omitted.
\smallskip

To establish the claimed tail behaviour of $\mu$, we consider the large $k$ asymptotics of $P(Z_t \ge k+1)$ in~\eqref{eq:distZt}. By Stirling's formula, $B(k+1,\beta/\gamma)^{-1} \asymp k^{\beta/\gamma}$, where we write $a(k) \asymp b(k)$ if there exist constants $0< c\le C<\infty$ such that $ca(k) \le b(k) \le Ca(k)$ for all large $k$.
Moreover, 
\begin{align*}
\frac{1}{1-\eps} \int_{\eps}^1 \int_{y^{\gamma}}^1 x^{\frac{\beta}{\gamma} -1} (1-x)^k \, dx \, dy \asymp\int_{\eps}^1 \int_{y^{\gamma}}^1 (1-x)^k \, dx \, dy=\frac{1}{k+1} \int_{\eps}^1 (1-y^{\gamma})^{k+1} \, dy \asymp \frac{(1-\eps^{\gamma})^{k+1}}{(k+1)^2}.
\end{align*}
In the first estimate we used that $x^{\frac{\beta}{\gamma}-1}$ is bounded from zero and infinity; in the second we employed Laplace's method (see for example Section~3.5 of \cite{Mil06}). In particular, $\mu^{\eps}$ has the stated tails. To complete the proof of Theorem~\ref{thm:degrees}, it remains to derive the asymptotic behaviour of the maximal indegree. The statement follows from the next two lemmas.

\begin{lemma}[Upper bound]\label{lem:degUpBd}
Let $c >-\frac{1}{\log(1 - \eps^{\gamma})}$. Then, 
\[
\P\big(\max_{m \in \vseteps_n} \Z[m,n] \le c \log n\big) \to 1 \qquad \mbox{ as $n\to\infty$.}
\]
\end{lemma}

\begin{proof} Write $k_n=\lfloor c \log n \rfloor$, $\underline{m}=\fst+1$ and $\Delta=n-\fst$. Moreover, let $\Z^m[\underline{m},n]$, $m \le n$, be independent copies of $\Z[\underline{m},n]$. 
Lemma~\ref{lem:degreeCoup} and (\ref{eq:degDistrApp}) yield
\begin{equation}\label{eq:UpBdEst}
\begin{split}
\P\big(\max_{m \in \vseteps_n} \Z[m,n] \le c \log n\big)& \ge \P\big(\max_{m \in \vseteps_n}\Z^m[\underline{m},n] \le k_n\big) = \P\big(\Z[\underline{m},n] \le k_n\big)^{\Delta}\\
& =  P\Big(\sum_{i=0}^{k_n} T[i] \ge -\log \eps +o(1)  \Big)^{\Delta}\\
& = \exp\Big( - \Delta P\Big(\sum_{i=0}^{k_n} T[i] \le -\log \eps +o(1)\Big) (1+o(1)) \Big),
\end{split}
\end{equation}
using a Taylor expansion in the last equality.  As above, uniformly for $t$ in compact subintervals of $(0,\infty)$, 
\[
\int_{e^{-\gamma t}}^1 x^{\frac{\beta}{\gamma}-1} (1-x)^{k} \, dx \asymp \int_{e^{-\gamma t}}^1 (1-x)^{k} \, dx \asymp \exp\Big(k \log(1-e^{-\gamma t}) -\log k \Big).
\]
Thus, Lemma~\ref{lem:distZt} and Stirling's formula yield for $\vartheta \in (0,1)$ and $t=-\log \vartheta+o(1)$, as $k \to \infty$,
\begin{equation}\label{eq:TiAsymp}
P\Big(\sum_{i=0}^{k} T[i] \le t \Big) \asymp \exp\Big( k \log(1-e^{-\gamma t})+(\tfrac{\beta}{\gamma}-1) \log k\Big)= \exp\Big( k \log(1-\vartheta^{\gamma}) (1+o(1))\Big).
\end{equation}
Using this estimate for $k=k_n$ and $\vartheta=\eps$, the exponent on the right-hand side of \eqref{eq:UpBdEst} tends to zero as $n \to \infty$ if
$c > -1/\log(1-\eps^{\gamma})$.
\end{proof}

\begin{lemma}[Lower bound]
Let $c < -\frac{1}{\log(1 - \eps^{\gamma})}$. Then,
\[
\P\big(\max_{m \in \vseteps_n} \Z[m,n] \le c \log n\big) \to 0 \qquad \mbox{ as $n\to\infty$.}
\]
\end{lemma}

\begin{proof} The idea of the proof is to restrict the maximum to an arbitrarily small proportion of the oldest vertices.
Let $\delta>0$ and write $k_n:=\lfloor c \log n\rfloor$ , $\Delta= \lfloor \delta \eps n\rfloor$ and $\overline{m}= \fst+\Delta$. Moreover, let $\Z^m[\overline{m},n]$, $m\le n$, be independent copies of $\Z[\overline{m},n]$. %For two real-valued random variables $X,Y$, we write $X \succeq Y$, when $X$ stochastically dominates $Y$. 
According to  Lemma~\ref{lem:degreeCoup} there is a coupling such that
\[
\max_{m \in \vseteps_n} \Z[m,n] \ge \max_{m=\fst+1,\ldots,\fst +\Delta} \Z[m,n] \ge \max_{m=\fst+1,\ldots,\fst +\Delta} \Z^m[\overline{m},n].
\]
Arguing as in \eqref{eq:UpBdEst}, (\ref{eq:degDistrApp}) yields
\begin{align*}
\P\big(\max_{m \in \vseteps_n} \Z[m,n] \le c \log n\big)
%& \le  \P\big(\max_{m=\fst+1,\ldots,\fst +\Delta} \Z^m[\overline{m},n] \le k_n\big)=  \P( \Z[\overline{m},n] \le k_n)^{\Delta}\\
%& = P\Big(\sum_{i=0}^{k_n} T[i] \ge -\log (\eps (1+\delta)) +o(1) \Big)^{\Delta}\\
%& = 
\le \exp\Big( - \Delta P\Big(\sum_{i=0}^{k_n} T[i] \le -\log (\eps (1+\delta)) +o(1) \Big)\big(1+o(1)\big) \Big).
\end{align*}
Now \eqref{eq:TiAsymp} with $\vartheta=\eps (1+\delta)$ implies that the exponent on the right-hand side tends to $-\infty$ if 
$c< -1/\log(1-(\eps (1+\delta))^{\gamma})$. Since $\delta$ was arbitrary, the claim is established.
\end{proof}

\subsection{Distances}\label{sec:distances}
In this section, we study the typical distance between two uniformly chosen vertices in $\grapheps_n$ and prove Theorem~\ref{thm:distances}.
We write $\triangle \Z[m,n]=\Z[m,n+1]-\Z[m,n]$ and, for $m \ge n$, $\Z[m,n]=0$. In the graph, the indegree of vertex $m$ at time $m$ is zero by definition, but we will also use the distribution of the process $(\Z[m,n]\colon n \ge m)$ for different initial  values. 
Formally, the evolution of $\Z[m,\cdot]$ with initial value $k$ is obtained by using the attachment rule $g(l):=f(k+l)$ and we denote its distribution by~$\P^k$, using~$\E^k$ for
the corresponding expectation; we abbreviate~$\P:=\P^0$, $\E:=\E^0$.
We further write $\threshold:=\inf\{n \in \N\colon f(n)/n \le 1\} \vee 2$. 
Note that $\gamma <1$ implies $\threshold \in \N$. 
We observe some facts about the indegree distribution. These are adaptations of results in \cite{DerMoe13}.

\begin{lemma}[Lemma~2.7 in \cite{DerMoe13}]\label{ZmeanEstim} For all $k \in \N_0$ and $m,n\in \N$ with $k \le m$, $\threshold \le m \le n$,
\begin{equation}\label{triangle_bound}
\P^k(\triangle \Z[m,n]=1) \le \frac{f(k)}{(m-1)^{\gamma} n^{1-\gamma}}.
\end{equation}
\end{lemma}

\begin{proof}
Observe that $(f(\Z[m,n]) \prod_{i=m}^{n-1} \sfrac{1}{1+\gamma/i}\colon n \ge m)$ is a martingale and therefore
\[
\P^k(\triangle \Z[m,n]=1)=\E^k\Big[\frac{f(\Z[m,n])}{n}\Big]=\frac{f(k)}{n} \prod_{i=m}^{n-1} (1+\gamma/i) \le \frac{f(k)}{(m-1)^{\gamma} n^{1-\gamma}}.\qedhere
\]
\end{proof}

\begin{lemma}[Lemma~2.10 in \cite{DerMoe13}]\label{coupFact} For all $k \in \N_0$, $m,m' \in \N$, $\threshold \le m \le m'$, $k \le m$, there exists a coupling of the process
$(\Z[m,n]\colon n\ge m)$ under the conditional probability $\P^k(\cdot|\triangle \Z[m,m']=1)$ and the process $(\Z[m,n]\colon n\ge m)$ under~$\P^{k+1}$ such that, apart from 
time $m'$, the jump times of the first process are a subset of the jump times of the latter.
\end{lemma}

The proof of the lemma is similar to the proof of Lemma~2.10 in \cite{DerMoe13} and we omit it. After these preliminary results, we now begin our analysis of typical distances in the network $(\grapheps_n\colon n\in\N)$. Recall, that for this type of questions, we consider $\grapheps_n$ to be an undirected graph. For $v, w \in \vseteps_n$ and $h \in \N_0$, let
\[
\pathset_h(v,w):=\{(v_0,\ldots,v_h)\colon v_i \in \vseteps_n, v_i \not=v_j \mbox{ for }  \, i \not= j, v_0=v, v_h=w\}
\]
be the set of all self-avoiding paths of length $h$ between $v$ and $w$, and
$\pathset_h(v)=\{\path\colon \path \in \pathset_h(v,w) \text{ for some }w\in \vseteps_n\}$
the set of all self-avoiding paths of length $h$ starting in $v$.
%Notice that $\pathset_0(v,w)=\emptyset$ if $v\not=w$, $\pathset_0(v,v)=\{(v)\}$, and $\pathset_h(v,v)=\emptyset$ for all $h \in \N$.

\begin{definition}
Let $\theta \in (0,\infty)$ and $\graph=(\vset,\eset)$ be an undirected graph with $\vset \subseteq \N$. A self-avoiding path $\path =(v_0,\ldots,v_h)$ in $\graph$ is $\theta$-admissible (or  admissible) if, for all $i \in \{1,\ldots,h\}$, we have $\{v_{i-1},v_i\} \in \eset$ and
\begin{equation}\label{def:admiss}
\big|\big\{w \in \vset\colon v_{i-1} <w\le v_i, \{v_{i-1},w\} \in \eset\big\}\big| \le \theta .
\end{equation}
\end{definition}
Note that \eqref{def:admiss} is automatically satisfied if $v_i<v_{i-1}$. In the graph~$\grapheps_n$, condition \eqref{def:admiss} can be written as $\Z[v_{i-1},v_i] \le \theta$.
We further denote, for $v,w \in \vseteps_n$, $h \in \N_0$ and $\theta \in (0,\infty)$,
\begin{align*}
N_{h}^{\theta}(v,w)&:=\big|\{\path \in \pathset_h(v,w)\colon \path \,\text{is }\theta\text{-admissible in }\grapheps_n\}\big|,\\
N_{h}^{\theta}(v)&:=\big|\{\path \in \pathset_h(v)\colon \path\, \text{is }\theta\text{-admissible in }\grapheps_n\}\big|,
\end{align*}
and, for $h >0$, $N_{\le h}^{\theta}(v)=\sum_{k=0}^{\lfloor h\rfloor} N_{k}^{\theta}(v)$, $N_{\le h}^{\theta}(v,w)=\sum_{k=0}^{\lfloor h\rfloor} N_{k}^{\theta}(v,w)$.  The dependence
of $\pathset_h(v,w)$, $\pathset_h(v)$, $N_h^{\theta}(v,w)$ etc.\ on~$n$ is suppressed in the notation, but it will always be clear from the context which graph is considered.
We write $\IBP^{\eps}(f)$ for the idealized branching process with type space $[\log \eps,0]\times \{\leftp,\rightp\}$ generated with attachment rule $f$ if we want to emphasize $f$ and $\eps$. The proof of the following lemma is deferred to Section \ref{sec:domi}. 

\begin{lemma}\label{lem:ExpPathGTree}
Let $\delta>0$ such that $\gamma (1+\delta) <1$, $\underline{\eps}\in (0,\eps)$, and $(\theta_n \colon n \in \N)$ a sequence of positive numbers
with $\theta_n=o(n)$. For all sufficiently large $n$, $v_0 \in \vseteps_n$, and $h \in \N_0$,
\[
\E\big[N_h^{\theta_n}(v_0)\big] \le E_{(s_n(v_0),\leftp)}\big[\big|\IBP_h^{\underline{\eps}}((1+\delta)f)\big|\big] \qquad \text{where }s_n(v_0):=-\sum_{k=v_0}^{n-1} \frac{1}{k}.
\]
\end{lemma}

We are now in the position to prove Theorem~\ref{thm:distances}.

\begin{proof}[Proof of Theorem~\ref{thm:distances}]
Let $v,w \in \vseteps_n$, $h \in \N$. With $\theta_n:=(\log n)^2$, \eqref{eq:MaxDegThm} yields
\begin{equation}\label{eq:dist1part}
\begin{split}
\P\big(\dist_{\grapheps_n}(v,w) \le h\big)&\le \P\Big(\dist_{\grapheps_n}(v,w) \le h, \max_{m \in \vseteps_n} \Z[m,n] \le \theta_n\Big)+ \P\Big(\max_{m \in \vseteps_n}\Z[m,n] >\theta_n\Big)\\
& \le \P\big(N_{\le h}^{\theta_n}(v,w) \ge 1\big)+o(1),
\end{split}
\end{equation}
where the error is uniform in $v,w$ and $h$. 
Markov's inequality yields, for every $v,w \in \vseteps_n$ with $v \not =w$ and for every $h \in \N$,
\begin{equation}\label{eq:NlehProb}
\P\big(N_{\le h}^{\theta_n}(v,w) \ge 1\big) \le  \E\big[N_{\le h}^{\theta_n}(v,w)\big] =\sum_{k=1}^h \sum_{\path \in \pathset_k(v,w)}\P\big(\path\, \text{is }\theta_n\text{-admissible in }\grapheps_n\big).
\end{equation}
We write $w_i^{+}:=v_{i-1}\vee v_i$, $w_i^{-}:=v_{i-1} \land v_i$ and $\mathcal{E}_i:=\{\triangle \Z[w_i^{-},w_i^+-1]=1, \Z[v_{i-1},v_i]\le \theta_n\}$ for every 
$i\in \{1,\ldots, k\}$.  Then we have
\begin{equation}\label{eq:admissEstim}
\begin{split}
\P(\path \text{ is }\theta_n\text{-admissible in } \grapheps_n)&=\P\Big(\bigcap_{i=1}^k\big\{\triangle \Z[w_i^{-},w_i^+-1]=1, \Z[v_{i-1},v_i]\le \theta_n\big\}\Big)\\
&= \P\Big(\mathcal{E}_{k}\Big|\bigcap_{i=1}^{k-1} \mathcal{E}_i\Big)\,\P((v_0,\ldots,v_{k-1}) \text{ is }\theta_n\text{-admissible in } \grapheps_n).
\end{split}
\end{equation}
To estimate the probability $\P(\mathcal{E}_{k}|\bigcap_{i=1}^{k-1} \mathcal{E}_i)$, we first note that the only edge in the self-avoiding path $\path$ on whose presence
the event $\{v_{k-1},v_k\}\in\eseteps_n$ can depend is $\{v_{k-2},v_{k-1}\}$. The possible arrangements of these two edges are sketched in Figure~4. When $v_{k-2}<v_{k-1}$ (cases A,B,C in Figure 4), then we, in addition, have knowledge of edges whose left vertex is $v_{k-2}$. However, these are always independent of $\{v_{k-1},v_k\}$. If $v_{k-1}<v_k$ (cases A,D,E in Figure 4), then event $\mathcal{E}_k$ requires that $\Z[v_{k-1},v_k]\le \theta_n$. Since edges with left vertex $v_{k-1}$ depend only on edges whose left vertex is also $v_{k-1}$, the only relevant conditioning occurs in cases D and E in Figure 4 by requiring $\{v_{k-1},v_{k-2}\}$ to be present.

\begin{center}
\begin{tikzpicture}
[vertex/.style={circle,draw=black,fill=black,thick,inner sep=0pt,minimum size=1mm}]

\foreach \y in {0,1}
\foreach \x in {0,1,2}
{
\draw (\x*4.5,\y*2) rectangle (\x*4.5+4.5,\y*2+2); %box
\draw[semithick] (\x*4.5,\y*2)+ (0.3,0.7) -- (\x*4.5+4.5-0.3,\y*2+0.7); %line
\node[vertex] (\x-\y-A) at (\x*4.5+0.3+0.5,\y*2+0.7)  {}; %leftmost vertex. 
\node[vertex] (\x-\y-B) at (\x*4.5+0.3+1.95,\y*2+0.7)  {}; %middle vertex
\node[vertex] (\x-\y-C) at (\x*4.5+4.5-0.3-0.5,\y*2+0.7)  {}; %rightmost vertex
\node[circle,draw=black,inner sep=0pt,minimum size=4mm] (\x-\y-D) at (\x*4.5+0.3,\y*2+1.7){}; %for label
}

%twice to the right

\draw[semithick] (0-1-A) to [bend left=70]  (0-1-B); %edges
\draw[semithick] (0-1-B) to [bend left=70]  (0-1-C);
\node [below= 3pt] at (0-1-A) {\small{$v_{k-2}$}}; %vertex labels
\node [below= 3pt] at (0-1-B) {\small{$v_{k-1}$}};
\node [below= 3pt] at (0-1-C) {\small{$v_{k}$}};

\draw[red,dashed] (0-1-A) to [bend left=40]  +(0.5,0);%edges to check number of right-neighbours
\draw[red,dashed] (0-1-A) to [bend left=50]  +(0.8,0);
\draw[red,dashed] (0-1-A) to [bend left=60]  +(1,0);

\draw[red,dashed] (0-1-B) to [bend left=40]  +(0.5,0);
\draw[red,dashed] (0-1-B) to [bend left=50]  +(0.8,0);
\draw[red,dashed] (0-1-B) to [bend left=60]  +(1,0);

\node at (0-1-D) {A};

%left-right forward

\draw[semithick] (0-0-A) to [bend left=70]  (0-0-C); %edges
\draw[semithick] (0-0-A) to [bend left=60]  (0-0-B);
\node [below= 3pt] at (0-0-A) {\small{$v_{k-1}$}}; %vertex labels
\node [below= 3pt] at (0-0-B) {\small{$v_{k-2}$}};
\node [below= 3pt] at (0-0-C) {\small{$v_{k}$}};

\draw[red,dashed] (0-0-A) to [bend left=40]  +(0.8,0);%edges to check number of right-neighbours
\draw[red,dashed] (0-0-A) to [bend left=60]  +(2,0);

\node at (0-0-D) {D};

%twice to the left

\draw[semithick] (2-0-A) to [bend left=70]  (2-0-B);
\draw[semithick] (2-0-B) to [bend left=70]  (2-0-C);
\node [below= 3pt] at (2-0-A) {\small{$v_{k}$}};
\node [below= 3pt] at (2-0-B) {\small{$v_{k-1}$}};
\node [below= 3pt] at (2-0-C) {\small{$v_{k-2}$}};

\node at (2-0-D) {F};

%right-left forward

\draw[semithick] (1-1-A) to [bend left=70]  (1-1-C); %edges
\draw[semithick] (1-1-B) to [bend left=60]  (1-1-C);
\node [below= 3pt] at (1-1-A) {\small{$v_{k-2}$}}; %vertex labels
\node [below= 3pt] at (1-1-B) {\small{$v_{k}$}};
\node [below= 3pt] at (1-1-C) {\small{$v_{k-1}$}};

\draw[red,dashed] (1-1-A) to [bend left=40]  +(0.8,0);%edges to check number of right-neighbours
\draw[red,dashed] (1-1-A) to [bend left=50]  (1-1-B);
\draw[red,dashed] (1-1-A) to [bend left=60]  +(2,0);

\node at (1-1-D) {B};

%right-left backward

\draw[semithick] (2-1-A) to [bend left=70]  (2-1-C); %edges
\draw[semithick] (2-1-B) to [bend left=60]  (2-1-C);
\node [below= 3pt] at (2-1-A) {\small{$v_{k}$}}; %vertex labels
\node [below= 3pt] at (2-1-B) {\small{$v_{k-2}$}};
\node [below= 3pt] at (2-1-C) {\small{$v_{k-1}$}};

\draw[red,dashed] (2-1-B) to [bend left=30]  +(0.5,0);%edges to check number of right-neighbours
\draw[red,dashed] (2-1-B) to [bend left=40]  +(0.8,0);
\draw[red,dashed] (2-1-B) to [bend left=50]  +(1,0);

\node at (2-1-D) {C};

%left-right backward

\draw[semithick] (1-0-A) to [bend left=70]  (1-0-C); %edges
\draw[semithick] (1-0-A) to [bend left=60]  (1-0-B);
\node [below= 3pt] at (1-0-A) {\small{$v_{k-1}$}}; %vertex labels
\node [below= 3pt] at (1-0-B) {\small{$v_{k}$}};
\node [below= 3pt] at (1-0-C) {\small{$v_{k-2}$}};

\draw[red,dashed] (1-0-A) to [bend left=30]  +(0.5,0);%edges to check number of right-neighbours
\draw[red,dashed] (1-0-A) to [bend left=40]  +(0.8,0);
\draw[red,dashed] (1-0-A) to [bend left=50]  +(1,0);

\node at (1-0-D) {E};

\end{tikzpicture}\\

{\small {\bf Figure 4:} Possible interactions of two edges on a self-avoiding path. The red, dashed edges have to be considered to decide if the number of right-neighbours is small enough to declare the path admissible.}
\end{center}

We deduce
\[
\P\Big(\mathcal{E}_k\Big|\bigcap_{i=1}^{k-1} \mathcal{E}_i \Big)=\begin{cases}
\P(\triangle \Z[w_k^{-},w_k^+-1]=1, \Z[v_{k-1},v_k]\le \theta_n) & \text{in A,B,C and F,}\\
\P(\triangle \Z[v_{k-1},v_k-1]=1, \Z[v_{k-1},v_k]\le \theta_n| \triangle \Z[v_{k-1},v_{k-2}-1]=1) & \text{in D and E.}
\end{cases}
\]
Using Lemma \ref{coupFact} and \eqref{triangle_bound}, we can bound the probability in both cases by $f(1)/(\eps n)$. Combining this estimate with \eqref{eq:NlehProb} and \eqref{eq:admissEstim}, we obtain
\[
\P\big(N_{\le h}^{\theta_n}(v,w) \ge 1\big)\le \sum_{k=1}^h \sum_{\path \in \pathset_{k-1}(v)} \frac{f(1)}{\eps n} \P(\path \text{ is }\theta_n\text{-admissible in }\grapheps_n)= \frac{f(1)}{\eps n} \sum_{k=0}^{h-1}\E\big[N_{k}^{\theta_n}(v)\big].
\]
Lemma~\ref{lem:ExpPathGTree} yields for small $\bar{\delta}>0$ and $\underline{\eps}:=\eps -\bar{\delta} >0$, 
\[
\P\big(N_{\le h}^{\theta_n}(v,w) \ge 1\big) \le \frac{f(1)}{\eps n}\, \sum_{k=0}^{h-1} E_{(s_n(v),\leftp)}\Big[\big|\IBP_k^{\underline{\eps}}((1+\bar{\delta})f)\big|\Big]. 
\]
We denote by $\bar{\rho}$ the spectral radius of the operator $A$ associated to $\IBP^{\underline{\eps}}((1+\bar{\delta})f)$ and by $\bar{\eigenfct}$ the corresponding eigenfunction. Choose a constant $C$ such that $C \ge{\max_{\type}\bar{\eigenfct}(\phi)}/{\min_{\type}\bar{\eigenfct}(\type)}$ for all sufficiently small $\bar{\delta}$. This is possible since the eigenfunctions are continuous in $\bar{\delta}$ (this can be seen along the lines of Note 3 to Chapter II on pages 568-569 of \cite{Kat95}). Furthermore, by the continuity of the spectral radius with respect to the operator (see Chapter II.5 in \cite{Kat95}) and, since $\rho_{\eps}(A)>1$ by assumption, $\bar{\rho}>1$ for all small~$\bar{\delta}$.

\pagebreak[3]
\noindent
Hence, for all $v,w \in \vseteps_n$, $v \not=w$, 
\begin{align*}
\P\big(N_{\le h}^{\theta_n}(v,w) \ge 1\big)& \le \frac{f(1)}{\eps n} \sum_{k=0}^{h-1} \frac{1}{\min_{\type} \bar{\eigenfct}(\type)} E_{(s_n(v),\leftp)}\big[\sum_{\heap{x \in\IBP^{\underline{\eps}}((1+\bar{\delta})f)}{|x|=k}}\bar{\eigenfct}(\type)\big]=\frac{f(1)}{\eps n} \sum_{k=0}^{h-1}\bar{\rho}^k  \frac{\bar{\eigenfct}(s_n(v),\leftp)}{\min_{\type} \bar{\eigenfct}(\type)}\\
& \le  \frac{f(1) C}{\eps n} \frac{\bar{\rho}^h}{\bar{\rho}-1}=\frac{f(1) C}{\eps (\bar{\rho}-1)} \exp\big(h \log\bar{\rho}-\log n\big).
\end{align*}
In particular, for $\delta>0$ and $h_n :=(1-\delta^2)\frac{\log n}{\log \bar{\rho}}$, we showed that 
\[
\sup_{v,w \in \vseteps_n, v \not=w} \P(N_{\le h_n}^{\theta_n}(v,w) \ge 1)=o(1).
\]
For independent, uniformly chosen vertices $V_n,W_n$ in $\vseteps_n$, we have $V_n \not=W_n$ with high probability. According to \eqref{eq:dist1part}, this implies $\P(\dist_{\grapheps_n}(V_n,W_n) \le h_n)=o(1)$. Choosing $\bar{\delta}$ so small that $\log \bar{\rho} \le (1+\delta)\log \rho_{\eps}(A)$, it follows that, with high probability,
\[
\dist_{\grapheps_n}(V_n,W_n) \ge (1-\delta^2)\frac{\log n}{\log \bar{\rho}} \ge (1-\delta)\frac{\log n}{\log \rho_{\eps}(A)}.\qedhere
\]
\end{proof}

\section{Approximation by a branching process}\label{sec:coup}
In this section, we compare the connected components in the network to the multitype branching process described in Section~\ref{sec:BP_def}. We begin by coupling the local neighbourhood of a uniformly chosen vertex to the $\IBP$ in Sections \ref{sec:network_tree} and \ref{sec:tree_IBP}. This local consideration allows us to draw conclusions about the existence or nonexistence of the giant component from knowledge of the branching process, see Section~\ref{sec:size_giant}. For the analysis of the typical distances in the network, knowing the local neighbourhood is insufficient. We show in Section \ref{sec:domi} that a slightly larger $\IBP$ dominates the network globally in a suitable way.

\subsection{Coupling the network to a tree}\label{sec:network_tree}
The proof of the coupling follows the lines of~\cite{DerMoe13} for the undamaged network, but unfortunately
we cannot use their results directly as the coupling in  \cite{DerMoe13} makes extensive use of vertices which are removed 
in the damaged network. Note however that the removal of the old vertices significantly reduces the risk of circles in the 
local neighbourhood of a vertex and, therefore, the coupling here will be successful for much longer than the coupling 
in~\cite{DerMoe13}.\smallskip 

In the first step, we couple the neighbourhood of a vertex in $\grapheps_n$ to a labelled tree. We consider trees where every vertex $v$ is equipped with a `tag' in $\vseteps_n$ and a 
`mark' $\mar \in \vseteps_n \cup \{\leftp\}$. We will use the same notation for vertex and tag to emphasize the similarity between the tree and the network, where for a vertex in the network the time at which it was added serves as its tag. For $v_0 \in \vseteps_n$, let $\tree_n(v_0)$ be the random tree with root $v_0$ of label $(v_0,\leftp)$ constructed as follows: Every vertex $v$ produces independently offspring to the left, i.e., with tag $u \in \{\fst+1,\ldots,v-1\}$ with probability
\[
\P(v \text{ has a descendant with tag } u)=\P(\triangle \Z[u,v-1]=1).
\]
All offspring on the left are of mark $v$. Moreover, independently, $v$ produces descendants to its right (i.e.\ with tag at least $v+1$), which are all of mark $\leftp$. The distribution of the cumulative sum of the right descendants depends on the mark of $v$. When $v$ is of mark $\mar=\leftp$, then the cumulative sum is distributed according to the law of $(\Z[v,u]\colon v+1 \le u \le n)$. When $v$ is of mark $\mar= w \in \vseteps_n$, $w>v$, then the cumulative sum follows the same distribution as $(\Z[v,u]-\mathbbm{1}_{[w,\infty)}(u)\colon v+1\le u \le n)$ conditioned on $\triangle \Z[v,w-1]=1$. The percolated version $\tree_{n,p}(v_0)$ is obtained from $\tree_n(v_0)$ by deleting every particle in $\tree_n(v_0)$ together with its line of descent with probability $1-p$, independently for all particles. In particular, with probability $1-p$ the root $v_0$ is deleted and $\tree_{n,p}(v_0)$ is empty.

\pagebreak[3]

We write $\connect_{n,p}^{\eps}(v_0)$ for the connected component in $\grapheps_n(p)$ containing vertex $v_0$.

\begin{proposition}\label{pro:graph_tree}
Suppose $(c_n \colon n \in \N)$ is a sequence of positive integers with $\lim_{n \to \infty}c_n^2/n = 0$. 
Then there exists a coupling of a uniformly chosen vertex $V_n$ in $\vseteps_n$, graph $\grapheps_n(p)$ and tree $\tree_{n,p}(V_n)$ such that
\[
|\connect_{n,p}^{\eps}(V_n)|\land c_n =|\tree_{n,p}(V_n)|\land c_n \qquad \text{with high probability}.
\]
\end{proposition}

To prove Proposition~\ref{pro:graph_tree}, we define an exploration process which we then use to inductively collect information about the tree and the network on the same probability space. We show that the two discovered graphs agree until a stopping time, which is with high probability larger than $c_n$. After that time, the undiscovered part of the tree and the network can be generated independently of each other.

%Let us first fix a vertex $v_0 \in \grapheps_n$. In both graphs, $\grapheps_n(p)$ and $\tree_{n,p}(v_0)$, $v_0$ is removed by percolation with the same probability. If $v_0$ is not removed, then we  distinguish three categories of vertices.

We begin by specifying the exploration process that is used to explore the connected component of a vertex $v_0$ in a labelled graph $\graph$, like $\connect_{n,p}^{\eps}(v_0)$ or $\tree_{n,p}(v_0)$. We distinguish three categories of vertices:

\begin{itemize}
\item \emph{veiled vertices:} vertices for which we have not yet found a connection to the cluster of $v_0$
\item \emph{active vertices:} vertices for which we already know that they belong to the cluster of $v_0$ but for which we have not yet explored all its immediate neighbours
\item \emph{dead vertices:} vertices which belong to the cluster of $v_0$ and for which all immediate neighbours have been explored
\end{itemize}

%Since $v_0$ is removed in both graphs with the same probability, the case that $v_0$ is removed can be perfectly coupled. If $v_0$ is not removed, then we run the following exploration process. 
At the beginning of the exploration only $v_0$ is active and all other vertices are veiled. In the first exploration step we explore all immediate neighbours of $v_0$, declare $v_0$ as dead and all its immediate neighbours as active. The other vertices remain veiled. We now continue from the active vertex $v$ with the smallest tag and explore all its immediate neighbours apart from $v_0$ from where we just came. The exploration is continued until there are no active vertices left.

We couple the exploration processes of the network and the tree started with $v_0 \in \vseteps_n$ up to a stopping time $\stoptime$, such that up to time $\stoptime$ both explored subgraphs (without the marks) coincide. In particular, the explored part of $\connect_{n,p}^{\eps}(v_0)$ is a tree and every tag has been used at most once by the active or dead vertices in $\tree_{n,p}(v_0)$. The event that at least one of these properties fails is called (E). We also stop the exploration, when either the number of dead and active vertices exceeds $c_n$ or when there are no active vertices left. In this case we say that the \emph{coupling is successful}. If we have to stop as a consequence of (E), we say that the \emph{coupling fails}.

\begin{lemma}\label{coupling_succ1}
Suppose that $p \in (0,1]$ and $(c_n \colon n \in \N)$ satisfies $\lim_{n \to \infty} c_n^2/n= 0$. Then
\[
\lim_{n \to \infty} \sup_{v_0 \in \vseteps_n} \P\big(\text{the coupling of }\connect_{n,p}^{\eps}(v_0) \text{ and } \tree_{n,p}(v_0) \text{ fails}\big)=0.
\]
\end{lemma}

In the proof of Lemma~\ref{coupling_succ1} we make use of the following result.

\begin{lemma}[Lemma~2.12 in \cite{DerMoe13}]\label{condEstF} 
Let $(c_n \colon n \in \N)$ satisfy $\lim_{n \to \infty}c_n/n= 0$. Then there exists a constant $C_{\ref{condEstF}}>0$ such that for all sufficiently large $n$, for all disjoint sets $\mathcal{I}_0, \mathcal{I}_1 \subseteq \vseteps_n$ with $|\mathcal{I}_0|\le c_n$ and $|\mathcal{I}_1|\le 1$, and for all $u,v \in \vseteps_n$
\begin{align*}
\P\big(\triangle \Z[v,u]=1\,\big|\,\triangle \Z[v,i]=1 \mbox{ for }  i \in \mathcal{I}_1,& \triangle \Z[v,i]=0 \mbox{ for }  i \in \mathcal{I}_0\big)\\
& \le C_{\ref{condEstF}} \P\big(\triangle \Z[v,u]=1\,\big|\,\triangle \Z[v,i]=1 \mbox{ for }  i \in \mathcal{I}_1\big).
\end{align*}
\end{lemma}

\begin{proof} We have
\begin{align*}
\P\big(\triangle \Z[v,u]=1\,\big| \,\triangle \Z[v,i]=1 \mbox{ for }  i \in \mathcal{I}_1,& \triangle \Z[v,i]=0 \mbox{ for }  i \in \mathcal{I}_0\big)\\
&\le \frac{\P(\triangle \Z[v,u]=1\,|\,\triangle \Z[v,i]=1 \mbox{ for }  i \in \mathcal{I}_1)}{\P(\triangle \Z[v,i]=0 \mbox{ for }  i \in \mathcal{I}_0\,|\,\triangle \Z[v,i]=1 \mbox{ for }  i \in \mathcal{I}_1)}.
\end{align*}
With $n$ so large that $\fst \ge \threshold$, Lemma~\ref{coupFact} and (\ref{triangle_bound}) imply that
\begin{align*}
\P(\triangle \Z[v,i]=0 \mbox{ for }  i \in \mathcal{I}_0\,|\,&\triangle \Z[v,i]=1 \mbox{ for }  i \in \mathcal{I}_1)\ge \P^1(\triangle \Z[v,i]=0 \mbox{ for }  i \in \mathcal{I}_0)\\
&\ge \prod_{i \in \mathcal{I}_0} \P^1(\triangle \Z[v,i]=0)\ge \prod_{i \in \mathcal{I}_0} \Big(1-\frac{f(1)}{(v-1)^{\gamma}i^{1-\gamma}}\Big) \ge \Big(1-\frac{f(1)}{\eps n}\Big)^{c_n}.
\end{align*}
Since $c_n/n$ tends to zero as $n \to \infty$, the right-hand side converges to one.
\end{proof}

\begin{proof}[Proof of Lemma~\ref{coupling_succ1}] We assume that $n$ is so large that $\fst\ge \threshold$.
To distinguish the exploration processes, we use the term \emph{descendant} for a child in the labelled tree and 
the term \emph{neighbour} in the context of $\grapheps_n(p)$. The $\sigma$-algebra generated by 
the exploration until the completion of step $k$ is denoted~$\F_k$. \smallskip

Since the probability of removing $v_0$ is the same in $\connect_{n,p}^{\eps}(v_0)$ and $\tree_{n,p}(v_0)$, this event can be perfectly coupled. 
If $v_0$ is not removed, then we explore the immediate neighbours of $v_0$ in $\grapheps_n(p)$ and the children of the root $v_0$ in the tree. Again these families 
are identically distributed and can be perfectly coupled. \smallskip

Now suppose that we successfully completed exploration step $k$ and are about to start the next step from 
vertex $v$. At this stage every vertex in the tree can be uniquely referred to by its tag and the subgraphs coincide. Denoting by $\mathfrak{a}$ and $\mathfrak{d}$ 
the set of active and dead vertices, respectively, we have $\mathfrak{a}\not=\emptyset$ and $|\mathfrak{a}\cup\mathfrak{d}|<c_n$.  
We continue by exploring the left descendants and neighbours of $v$. Since we always explore the leftmost active vertex, we cannot encounter a dead or active neighbour 
in this step. However, in the tree $\tree_{n,p}(v_0)$ we may find a dead left descendant (i.e.\ an offspring whose tag agrees with the tag of a dead particle); 
we call this event Ia. On Ia, the vertices in the explored part of $\tree_{n,p}(v_0)$ are no longer uniquely identifiable by their tag and we stop. We have 
\[
\P\big(\text{Ia}\,|\,\F_k\big)=\P\big(\exists d \in \mathfrak{d}\colon \text{$d$ is a left descendant of $v$}\,|\,\F_k\big) 
\le \sum_{d \in \mathfrak{d}} \P(\triangle \Z[d,v-1]=1)\le c_n \frac{f(0)}{\eps n}.
\]
In the first inequality, we used subadditivity, the definition of $\tree_{n,p}(v_0)$ and omitted the event that offspring of $v$ are removed by percolation.
Hence, $\P(\text{Ia})=O(c_n/n)$. In the exploration to the left in the tree, we immediately check if a found left descendant has a right descendant which is dead. We denote this event by Ib and stop the exploration as soon as it occurs. The reason is that in the network this event could not happen since we always explore the leftmost active vertex. Therefore, the distribution of left neighbours agrees with the distribution of the left descendants conditioned on having no dead right descendants and we can couple both explorations such that they agree in this case. To estimate the probability of the adverse event Ib, we use the definition of $\tree_{n,p}(v)$ to obtain
\begin{align*}
\P\big(\text{Ib}\,|\,\F_k\big)&\le\P\big(\exists u \in \mathfrak{d}^c, d \in \mathfrak{d}\colon \text{$u$ is a left descendant of $v$, $d$ is a right descendant of $u$}\,|\,\F_k\big)\\
&\le \sum_{u \in \mathfrak{d}^c}\sum_{d \in \mathfrak{d}} \P(\triangle\Z[u,v-1]=1)\P(\triangle \Z[u,d-1]=1\,|\,\triangle\Z[u,v-1]=1).
\end{align*}
By definition of the exploration process, there are at most $c_n$ dead vertices. Therefore, Lemma~\ref{coupFact} and (\ref{triangle_bound}) yield
\[
\P\big(\text{Ib}|\F_k\big) \le c_n \sum_{u \in \mathfrak{d}^c, u\le v-1} \frac{f(0)}{(u-1)^{\gamma}(v-1)^{1-\gamma}} \frac{f(1)}{\eps n} \le c_n\frac{f(0) f(1)}{\eps n} \frac{1}{(v-1)^{1-\gamma}} \sum_{u=1}^{v-1} u^{-\gamma},
\]
which implies in particular that $\P(\text{Ib})=O(c_n/n)$. \smallskip

We turn to the exploration of right descendants, resp.\ neighbours. When vertex $v$ is of mark $\mar\not= \leftp$, then we already know that $v$ has no right descendants, resp.\ neighbours, in $\mathfrak{d}$ since we checked this when $v$ was discovered. 
We denote the event that a right descendant, resp.\ neighbour, is active by IIr and stop the exploration as soon as this event occurs because the tags in $\tree_{n,p}(v_0)$ are no longer unique, 
resp.\ we found a circle in $\connect_{n,p}^{\eps}(v_0)$. According to Lemma~\ref{condEstF} and (\ref{triangle_bound}),
\begin{align*}
\P\big(\text{IIr}\,\big|\,\F_k\big)&\le\P\big(\exists a \in \mathfrak{a}\colon \triangle \Z[v,a-1]=1\,\big|\, \triangle \Z[v,\mar-1]=1,\triangle\Z[v,d-1]=0\, \mbox{ for }  \,d \in \mathfrak{d}\setminus\{\mar\}, \F_k\big)\\
&\le C_{\ref{condEstF}} \sum_{a \in \mathfrak{a}} \P^1(\triangle \Z[v,a-1]=1)\le C_{\ref{condEstF}} c_n \frac{f(1)}{\eps n}.
\end{align*}
Thus, $\P(\text{IIr})=O(c_n/n)$. Conditional on the event that there are no active vertices in the set of right descendants, resp.\ neighbours, the offspring distribution in tree and network 
agree and can therefore be perfectly coupled. When the vertex $v$ is of mark $\mar=\leftp$, then we have not gained any information about its right descendants, yet. The event that there is a dead or active vertex in the right descendants is denoted by II$\ell$a. We stop when this event occurs and use (\ref{triangle_bound}) to estimate
\[
\P\big(\text{II$\ell$a}|\F_k\big)=\P\big(\exists a \in \mathfrak{a}\cup \mathfrak{d}\colon \text{$a$ is a right descendant of $v$}\big|\F_k\big)\le \sum_{a \in \mathfrak{a}\cup \mathfrak{d}} \P(\triangle \Z[v,a-1]=1)\le c_n \frac{f(0)}{\eps n}.
\]
Thus, $\P(\text{II$\ell$a})=O(c_n/n)$. In $\connect_{n,p}^{\eps}(v_0)$, we know that $v$ has no dead right neighbours as this would have stopped the exploration in the moment when $v$ became active. The event that there are active vertices in the set of right neighbours is denoted by II$\ell$b and we stop as soon as it occurs since a cycle is created. Using again (\ref{triangle_bound}), we find
\begin{align*}
\P\big(\text{II$\ell$b}|\F_k\big)&=\P\big(\exists a \in \mathfrak{a}\colon \triangle \Z[v,a-1]=1\big|\triangle \Z[v,d-1]=0 \text{ for }d \in \mathfrak{d},\F_k\big)\\
&\le \sum_{a \in \mathfrak{a}} \P(\triangle \Z[v,a-1]=1)\le c_n \frac{f(0)}{\eps n}.
\end{align*}
As in case $\mar\not=\ell$, the explorations can be perfectly coupled when the adverse events do not occur. We showed that in every step the coupling fails with a probability bounded by $O(c_n/n)$. As there are at most $c_n$ exploration steps until we end the coupling successfully, the probability of failure is $O(c_n^2/n)=o(1)$. In other words, the coupling succeeds with high probability. 
\end{proof}

\begin{proof}[Proof of Proposition~\ref{pro:graph_tree}]
First, consider the statement for a fixed vertex $v_0$. When the coupling is successful and ends because at least $c_n$ vertices were explored, then $|\connect_{n,p}^{\eps}(v_0)|\ge c_n$ and $|\tree_{n,p}(v_0)|\ge c_n$. If the coupling is successful and ends because there are no active vertices left, then $|\connect_{n,p}^{\eps}(v_0)|=|\tree_{n,p}(v_0)|$ since the subgraphs coincide. Since the coupling is successful with high probability by Lemma~\ref{coupling_succ1}, $|\connect_{n,p}^{\eps}(v_0)|\land c_n=|\tree_{n,p}(v_0)|\land c_n$ with high probability. As Lemma~\ref{coupling_succ1} shows the success of the coupling uniformly in the start vertex, the randomization of the vertex $v_0$ to a uniformly chosen vertex $V_n \in \vseteps_n$ is now straightforward.
\end{proof}

\subsection{Coupling the tree to the $\IBP$}\label{sec:tree_IBP}

Coupling the neighbourhood of a vertex to a labelled tree provides a great simplification of the problem since many dependencies are eliminated. However, the offspring distribution in the tree $\tree_{n,p}(V_n)$ is still complicated and depends on~$n$. Since we are mainly interested in the asymptotic size of the giant component, we now couple the tree to the $\IBP$, which does not 
depend on~$n$ and is much easier to analyse. We denote by $|\mathcal{X}^{\eps}(p)|$ the total progeny of the $\IBP$. Recall the definition of $S^{\eps}$ from \eqref{def:Seps}.

\begin{proposition}\label{pro:graph_IBP}
Let $p \in (0,1]$ and $(c_n \colon n \in \N)$ be a sequence of positive integers with $\lim_{n \to \infty}c_n^3/n = 0$.
Then there exists a coupling of a uniformly chosen vertex $V_n$ in $\vseteps_n$, the graph $\grapheps_{n}(p)$ and the percolated $\IBP$ started with a particle of mark $\leftp$ and 
location~$S^{\eps}$ such that, with high probability,
\[
|\connect_{n,p}^{\eps}(V_n)|\land c_n =|\mathcal{X}^{\eps}(p)|\land c_n.
\]
\end{proposition}

\begin{proof}[Proof of Proposition~\ref{pro:graph_IBP}] Throughout the proof, suppose that $n$ is so large that $\fst \ge \threshold$.
Instead of coupling the $\IBP$ directly to the network, we couple a projected version of the $\IBP$ to the tree $\tree_{n,p}(V_n)$. As long as the number of particles is preserved under the projection, this is sufficient according to Proposition~\ref{pro:graph_tree}. To describe the projection, we define $\pi_n^{\eps}\colon [\log \eps,0] \to \vseteps_n$ by
\begin{equation}\label{eq:projDef}
\pi_n^{\eps}(\loc)= v \qquad \Leftrightarrow \qquad s_n(v-1) < \loc \le s_n(v),
\end{equation}
where $s_n(v)=-\sum_{k=v}^{n-1} \sfrac{1}{k}$. Since $s_n(\fst) < \log(\fst/n) \le \log \eps$, every location in $[\log \eps,0]$ can be uniquely identified with a tag in $\vseteps_n$ by the map $\pi_n^{\eps}$. The projected $\IBP$ is again a labelled tree: The genealogical tree of the $\IBP$ with its marks is preserved, the location of a particle $x$ is replaced by the
tag $\pi_n^{\eps}(\loc(x))$. 
If $s_n(\fst+1) <\log \eps$, then no particles of the $\IBP$ are projected onto $\fst+1$. 
Moreover, while for $v \ge \fst+3$ an interval of length $1/(v-1)$ is projected onto $v$, for $\fst+2$ only an interval of length at most $s_n(\fst+2)- \log \eps$ is used. This length is positive but may be smaller than $1/(\fst+1)$. As a consequence, the projected $\IBP$ can have unusually few particles at $\fst+1$ and $\fst+2$ and we treat these two tags separately.
\smallskip
 
The exploration of the two trees follows the same procedure as the exploration described in Section~\ref{sec:network_tree} and we declare the coupling successful and stop as soon as either there are no active vertices left or the number of active and dead vertices exceeds $c_n$. Since both objects are trees, as long as the labels for the starting vertices agree, any failure of the coupling comes from a failure in the coupling of the offspring distributions. For simplicity, we consider only the case $p=1$. The generalization to $p \in (0,1]$ is straightforward. 
\pagebreak[3]

We first show that the labels of the starting vertices can be coupled with high probability.  To this end, note that the distribution of $S^{\eps }$ is chosen such that $\exp(S^{ \eps})$ 
is uniformly distributed on $[\eps,1]$. Since $\log \eps \le s_n(\fst+2) \le s_n(v-1) \le s_n(v) \le 0$,  for $v \ge \fst+3$, we obtain
\[
\P\big(\pi_n^{\eps}(S^{\eps})=v\big)=\P\big(e^{s_n(v-1)}<e^{S^{\eps}}\le e^{s_n(v)}\big) =\sfrac{1}{1-\eps}\big(e^{s_n(v)}-e^{s_n(v-1)}\big)= \sfrac{1}{1-\eps} \, e^{s_n(v-1)} \big(e^{1/(v-1)}-1\big).
\]
The right-hand side is in the interval $[\frac{1}{1-\eps} (\frac{1}{n}- \frac{2}{vn}),\frac{1}{1-\eps} (\frac{1}{n}+ \frac{2}{vn}) ]$. Moreover, the probability that 
$V_n$ or $\pi_n^{\eps}(S^{ \eps})$ is in $\{\fst+1,\fst+2\}$ is of order $O(1/n)$. Hence, $V_n$ and $S^{ \eps}$ can be coupled such that
\[
\P\big(V_n \not= \pi_n^{\eps}(S^{ \eps})\big) \le \sum_{v=\fst+3}^{n} \big|\P\big(\pi_n^{\eps}(S^{ \eps})=v\big)-\sfrac{1}{n - \fst}\big| +O\big(\sfrac{1}{n}\big)= O\big(\sfrac{\log n}{n}\big).
\]
In the next step we study the offspring distributions of a particle $x$ in the $\IBP$ with label $(\loc,\mar)$ and $\pi_n^{\eps}(\loc)=v$.
We start with the offspring to the left. Let $u \in \{\fst+1,\ldots,v\}$. By definition of the $\IBP$, particle $x$ produces a Poissonian number of projected offspring with tag $u$ with parameter
\[
\int_{(s_n(u-1) -\loc) \vee (\log \eps-\loc)}^{(s_n(u)-\loc)  \land 0 } \beta e^{(1-\gamma)t} \, dt.
\]
A vertex with tag $v$ in $\tree_n(V_n)$ produces a Bernoulli distributed number of descendants with tag $u$ with success probability $\P(\triangle \Z[u,v-1]=1)$ when $u<v$, and with success probability zero when $u=v$. It is proved in Lemma~6.3 of \cite{DerMoe13} that for $u \ge \fst+3$ the Poisson distributions can be coupled to the Bernoulli distribution such that they disagree with a probability bounded by a constant multiple of $v^{\gamma-1} u^{-(\gamma+1)}$ for $u<v$ and $1/v$ for $u=v$. For $u \in \{\fst+1,\fst+2\}$ a similar estimate shows that the probability can be bounded by a constant multiple of $1/(\eps n)$. Since the number of descendants with tag in $\{\fst+1,\ldots,v\}$ form an independent sequence of random variables, we can apply the coupling sequentially for each location and obtain a coupling of the $\pi_n^{\eps}$-projected left descendants in the $\IBP$ and the left descendants in $\tree_n(V_n)$. The failure probability of this coupling can be estimated by
\[
\P(\text{left descendants of }v\text{ disagree}) \le \frac{3C}{\eps n}+\frac{C'}{v^{1-\gamma}} \sum_{u = \fst+3}^{v-1} \frac{1}{u^{\gamma+1}} \le \frac{3C}{\eps n} + \frac{C'}{\eps n} \log \Big(\frac{v-1}{\fst}\Big) \le \frac{C''}{n},
\] 
where $C,C',C''$ are suitable positive constants whose value can change from line to line in the sequel. 
We turn to the offspring on the right. Suppose that particle $x$ in the $\IBP$ has mark $\mar=\leftp$. The cumulative sum of $\pi_n^{\eps}$-projected right descendants of $x$ follows the same distribution as $(Z_{s_n(u)-\loc}\colon v\le u \le n)$. The cumulative sum of right descendants of $v$ in $\tree_n(V_n)$ is distributed according to the law of $(\Z[v,u] \colon v \le u \le n)$. The following lemma is taken from \cite{DerMoe13} and we omit its proof.

\begin{lemma}[Lemma~6.2 in \cite{DerMoe13}]\label{uncondCoup} Fix a level $H \in \N$. We can couple the processes $(Z_{s_n(u)-\loc}\colon v\le u \le n)$ and  $(\Z[v,u] \colon v \le u \le n)$ such that for the coupled processes $(\inMC^{\sss(1)}_u\colon v \le u \le n)$ and $(\inMC^{\sss (2)}_u\colon v \le u \le n)$
\[
\P(\inMC^{\sss (1)}_u \not=\inMC^{\sss (2)}_u \text{ for some }u \le \sigma_H) \le C_{\ref{uncondCoup}}\frac{f(H)^2}{v-1},
\]
for some constant $C_{\ref{uncondCoup}}>0$ and $\sigma_H$ the first time one of the processes reaches or exceeds $H$.
\end{lemma}

In the coupling between the tree $\tree_n(V_n)$ and the projected $\IBP$ we consider at most $c_n$ right descendants. Hence, Lemma \ref{uncondCoup} implies that the distributions can be coupled such that
\[
\P(\text{right descendants of }v\text{ disagree}) \le C_{\ref{uncondCoup}} \frac{f(c_n)^2}{v-1} \le C \frac{c_n^2}{\eps n},
\] 
for some $C>0$. When $\mar =\rightp$, then cumulative sum of $\pi_n^{\eps}$-projected right descendants of $x$ follows the same distribution as $(\hat{Z}_{s_n(u)-\loc}-1\colon v\le u \le n)$. The cumulative sum of right descendants of a vertex~$v$ with mark $w\in \vseteps_n$, $w>v$, in $\tree_n(V_n)$ is distributed according to $(\Z[v,u]-\mathbbm{1}_{[w,\infty)}(u)\colon v \le u \le n)$ conditioned on $\triangle \Z[v,w-1]=1$. We can couple these two distributions. Again the proof of the following lemma is up to minor changes given in~\cite{DerMoe13} and therefore omitted.

\begin{lemma}[Lemma~6.6 in \cite{DerMoe13}]\label{condCoup} Fix a level $H \in \N$. We can couple the processes $(\hat{Z}_{s_n(u)-\loc}-1\colon v \le u \le n)$ and  $(\Z[v,u] -\mathbbm{1}_{[w,\infty)}(u)\colon v \le u \le n)$ conditioned on $\triangle \Z[v,w-1]=1$ such that for the coupled processes $(\inMC^{\sss (1)}_u\colon v \le u \le n)$ and $(\inMC^{\sss(2)}_u\colon v \le u \le n)$
\[
\P(\inMC^{\sss(1)}_u \not=\inMC^{\sss (2)}_u \text{ for some }u \le \sigma_H) \le C_{\ref{condCoup}}\frac{f(H)^2}{v-1},
\]
for some constant $C_{\ref{condCoup}}>0$ and $\sigma_H$ the first time one of the processes reaches or exceeds $H$.
\end{lemma}

As we explore at most $c_n$ vertices during the exploration, Lemma \ref{condCoup} implies that we can couple the offspring distribution to the right such that there is a constant $C>0$ with
\[
\P(\text{right descendants of }v\text{ disagree}) \le C_{\ref{condCoup}} \frac{f(c_n)^2}{v-1} \le C \frac{c_n^2}{\eps n}.
\] 
Since we explore at most $c_n$ vertices in total, the probability that the coupling fails can be bounded by a constant multiple of 
${c_n}/{n}+{c_n^3}/{n}$, which converges to zero. Thus, the two explorations can be successfully coupled with high probability
and, as in the proof of Proposition~\ref{pro:graph_tree}, the claim follows.
\end{proof}

\subsection{Dominating the network by a branching process}\label{sec:domi}
Like in the coupling, we begin with a comparison to a tree: For $\theta \in \N$ and $v_0 \in \vseteps_n$, let $\Ftree_n^{\eps,\theta}(v_0)$ be the subtree of $\tree_n(v_0)$, where every particle can have at most $\theta$ offspring to the right. That is, for a particle with tag $v$ and mark $\mar=\leftp$ the cumulative sum of the offspring to the right is distributed according to the law of $(\Z[v,u] \land \theta \colon v+1\le u\le n)$. When $v$ is of mark $\mar =w \in \vseteps_n$, $w>v$, then the cumulative sum follows the same distribution as $((\Z[v,u]-\mathbbm{1}_{[w,\infty)}(u)) \land \theta\colon v+1\le u \le n)$ conditioned on $\triangle \Z[v,w-1]=1$.
Recall from Section \ref{sec:distances} that $N_h^{\theta}(v_0)$ denotes the number of $\theta$-admissible paths of length $h$ in $\grapheps_n$ with starting-point $v_0$.

\begin{lemma}\label{lem:graphTreeDomi}
For all $\theta,h,n \in \N$, $v_0 \in \vseteps_n$,
\[
\E\big[N_h^{\theta}(v_0)\big] \le \E\Big[\big|\big\{\text{particles in generation }h\text{ of }\Ftree_n^{\eps,\theta}(v_0)\big\}\big|\Big].
\]
\end{lemma}

\begin{proof}
Let $\path=(v_0,\ldots,v_h) \in \pathset_h(v_0)$. Using the notation and set-up from the proof of Theorem~\ref{thm:distances}, and the 
definition of the tree $\Ftree_n^{\eps,\theta}(v_0)$, one easily checks that in cases A,B,C,E and F of Figure 4, $\P(\mathcal{E}_h|\cap_{i=1}^{h-1} \mathcal{E}_i )$ agrees with the probability that in tree $\Ftree_n^{\eps,\theta}(v_0)$ a vertex with tag $v_{h-1}$ gives birth to a particle of tag $v_{h}$ given that its parent has tag $v_{h-2}$.
In case D of Figure 4, the tree $\Ftree_n^{\eps,\theta}$ is allowed to have one more offspring on its right, because the edge $\{v_{h-2},v_{h-1}\}$ is not accounted for. Hence $\P(\mathcal{E}_h|\cap_{i=1}^{h-1} \mathcal{E}_i )$ is bounded from above by the probability for the event in the tree. We obtain
\[
\E[N_h^{\theta}(v_0)]=\sum_{\path \in \pathset_h(v_0)}\P(\path \text{ is }\theta\text{-admissible in }\grapheps_n)\le \sum_{\path \in \pathset_h(v_0)}\P(\path \text{ present in }\Ftree_n^{\eps,\theta}(v_0) ).
\] 
Particles in generation $h$ of $\Ftree_{n}^{\eps,\theta}(v_0)$, who have two ancestors with the same tag, are not represented in the sum on the right-hand side. Adding these, we obtain the result.
\end{proof}

\begin{proof}[Proof of Lemma~\ref{lem:ExpPathGTree}]
By Lemma~\ref{lem:graphTreeDomi}, it suffices to dominate $\Ftree_{n}^{\eps,\theta_n}(v_0)$ by the $\IBP^{\underline{\eps}}((1+\delta)f)$ started in $s_n(v_0)$, or, as in the proof of Proposition~\ref{pro:graph_IBP}, by the $\pi_n^{\underline{\eps}}$-projected $\IBP$ defined by \eqref{eq:projDef}. Since both processes are trees starting with the same type of particle, it suffices to compare the offspring distributions. All particles in $\Ftree_n^{\eps,\theta_n}(v_0)$ have a tag $v>\fst$, but the projected $\IBP$ can have offspring with tag $v \in \{\lfloor \underline{\eps} n\rfloor +1,\ldots,\fst\}$. Hence, these offspring are ignored in the following, giving us a lower bound on the projected $\IBP$. We assume that $n$ is so large, that $n \ge \threshold$ and $s_n(\fst+1) \ge \log\underline{\eps}$. 
Let $x$ be a particle in the $\IBP$ of type $(\loc,\mar)$ with $\pi_n^{\underline{\eps}}(\loc)=v$. We begin with the offspring to the left, i.e.\ tag $u \in \{\fst+1,\ldots,v\}$. A particle in $\Ftree_n^{\eps,\theta_n}(v_0)$ with tag $v$ cannot produce particles in $u=v$, therefore, the $\IBP$ clearly dominates. For $u<v$, using~\eqref{triangle_bound}, the probability that a particle 
with tag~$u$ is a child of~$x$, is $\P(\triangle \Z[u,v-1]=1)\big)\leq {\beta} (u-1)^{-\gamma}(v-1)^{-(1-\gamma)}.$
Writing $\bar{f}(k)=(1+\delta)f(k)=\bar{\gamma}k+\bar{\beta}$, for~$k \in \N_0$, the number of particles with tag $u$ produced by $x$ in the projected $\IBP$ follows a Poisson
distribution with parameter
\[
\int_{s_n(u-1)-\loc}^{s_n(u) -\loc} \bar{\beta} e^{(1-\bar{\gamma})t} \, dt \leq \frac{\bar{\beta}}{u-1} e^{-(1-\gamma) \sum_{k=u-1}^{v-1} \frac{1}{k} }\leq 
\frac{\bar{\beta}}{u-1} \Big(\frac{u-2}{v-1}\Big)^{1-\gamma},
\]
where we used that $\loc \le s_n(v)$ and $e^y-1 \ge y$. For $\varrho>0$, $\eta \in [0,1]$, the Poisson distribution with parameter $\varrho$ is dominating the Bernoulli distribution with
parameter~$\eta$ if and only if $e^{-\varrho}\le 1-\eta$. Since $e^{-y} \le 1-y +y^2/2$ for $y \ge 0$, it suffices to show that $\varrho (1-\varrho/2) \ge \eta$. In our case, $\eta= \beta(u-1)^{-\gamma}(v-1)^{-(1-\gamma)}$, $\varrho=\eta (1+\delta) (1-1/(u-1))^{1-\gamma}$ and the inequality holds for all large $n$ and $u \in \vseteps_n$, $u<v$, since $\eta$ is a null sequence.
\smallskip

We turn to the right descendants. The pure jump process corresponding to the attachment rule $\bar{f}$ is denoted by $\bar{Z}$ and we write $P^l$ for the distribution of $\bar{Z}$ when started in $l$, that is, $P^l(\bar{Z}_0=l)=1$. First suppose that $\mar=\leftp$. The cumulative sum of $\pi_n^{\underline{\eps}}$-projected right descendants of $x$ have the distribution of $(\bar{Z}_{s_n(u)-\loc}\colon v \le u \le n)$, where $\bar{Z}_0=0$. The cumulative sum of right descendants of $v$ in $\Ftree_n^{\eps,\theta_n}(v_0)$ is distributed according to the law of $(\Z[v,u]\land \theta_n \colon v \le u \le n)$. We couple these distributions by defining $((\inMC_u^{\sss (1)},\inMC_u^{\sss (2)})\colon v \le u \le n)$ to be the time-inhomogeneous Markov chain which starts in $P^0(\bar{Z}_{s_n(v)-\loc} \in \cdot) \otimes \delta_0$, has the desired marginals and evolves from state $(l,k)$ at time $j$ according to a coupling of $\bar{Z}_{1/j}$ and $\Z[j,j+1]$, which guarantees that $\bar{Z}_{1/j} \ge \Z[j,j+1]$, until $\inMC^{\sss (2)}$ reaches state $\theta_n$, where $\inMC^{\sss (2)}$ is absorbed. To show that this coupling exists, it suffices to show that
\[
e^{-\bar{f}(l)/j} = P^l(\bar{Z}_{1/j}=l) \le \P^k(\Z[j,j+1]=k)=1-f(k)/j \qquad \mbox{ for }  \, j \in \vseteps_n, k \le \theta_n, k \le l.
\]
Since $\bar{f}$ is non decreasing, this inequality follows as above once we show that $\varrho(1-\varrho/2) \ge \eta$ with $\eta =f(k)/j$, $\varrho=\bar{f}(k)/j =\eta (1+\delta)$. Since $k \le \theta_n=o(n)$ and $j \ge \fst$, $\eta$ is a null sequence and the claim follows. Hence, $\inMC_j^{\sss (1)} \ge \inMC_j^{\sss (2)}$ for all $j$ and the domination is established. 
\smallskip

Now suppose that $\mar=\rightp$ and that the location of $x$ parent is projected onto tag $w$. The cumulative sum of $\pi_n^{\underline{\eps}}$-projected right descendants of $x$ has the 
distribution of $(Y_{s_n(u)-\loc}\colon v \le u \le n),$ where $Y$ is a version of $\bar{Z}$ under measure $P^1$. 
The cumulative sum of right descendants of $v$ in $\Ftree_n^{\eps,\theta_n}(v_0)$ is distributed according to the law of 
$((\Z[v,u]-\mathbbm{1}_{[w,\infty)}(u))\land \theta_n \colon v \le u \le n)$ conditioned on $\triangle\Z[v,w-1]=1$.
We couple these distributions as in the $\mar=\leftp$ case, but for times $j \le w-2$, the Markov chain evolves from state $(l,k)$ according to a coupling of $Y_{1/j}$ and $\Z[j,j+1]$ conditioned on $\triangle \Z[j,w-1]=1$, which guarantees that $Y_{1/j} \ge \Z[j,j+1]$, until either $j=w-2$ or $\inMC^{\sss (2)}$ reaches $\theta_n$ and is absorbed. To show that this coupling exists, it suffices to show that
\begin{equation}\label{ts:dominR1}
P^{l+1}(\bar{Z}_{1/j}-1=l) \le \P^k(\Z[j,j+1]=k|\Z[j,w-1]=1) \qquad \mbox{ for }  \, j \in \vseteps_n, k \le \theta_n, k \le l.
\end{equation}
We compute
\begin{align*}
\P^{k}(\Z[j,j+1]=k|\triangle\Z[j,w-1]=1)&=1-\frac{\P^k(\triangle\Z[j,j]=1,\triangle\Z[j,w-1]=1)}{\P^k(\triangle\Z[j,w-1]=1)}\\
&=1- \frac{\frac{f(k)}{j} \P^{k+1}(\triangle\Z[j+1,w-1]=1)}{\P^k(\triangle\Z[j,w-1]=1)}=1-\frac{f(k+1)}{j+\gamma}.
\end{align*}
Since $\bar{f}$ is non-decreasing, \eqref{ts:dominR1} follows when we show that $\varrho (1-\varrho/2) \ge \eta$ with $\eta=f(k+1)/(j+\gamma)$ and $\varrho=\bar{f}(k+1)/j =\eta (1+\delta) (1+\gamma/j)$. Since $k\le \theta_n=o(n)$ and $j\ge \fst$, $\eta$ is a null sequence and (\ref{ts:dominR1}) is proved. In the transition from generation $j=w-1$ to $j=w$, $\inMC^{\sss (2)}$ cannot change its state while $\inMC^{\sss (1)}$ can increase. From generation $j =w$ onwards, the coupling explained in case $\mar=\leftp$ is used. Thus, the Markov chain can be constructed such that $\inMC_j^{\sss (1)}\ge \inMC_j^{\sss (2)}$ for all $j$ and the domination is proved.
\end{proof}

\subsection{Proof of Theorem~\ref{thm:size_giant}}\label{sec:size_giant}
Proposition \ref{pro:graph_IBP} implies the following result.

\begin{corollary}\label{mean_conv}
Let $p \in (0,1]$ and $(c_n \colon n \in \N)$ be a sequence with $\lim_{n \to \infty}c_n^3/n= 0$ and $\lim_{n \to \infty} c_n = \infty$. Then
\[
\E\Big[\frac{1}{n-\fst} \sum_{v=\fst+1}^n \mathbbm{1}\{|\connect_{n,p}^{\eps}(v)|\ge c_n \}\Big]=\P\big(|\connect_{n,p}^{\eps}(V_n)| \ge c_n\big) \to P(|\mathcal{X}^{\eps}(p)|=\infty)=\zeta^{\eps}(p)\quad \text{as }n \to \infty.
\]
\end{corollary}

This convergence can be strengthened to convergence in probability. 

\begin{lemma}\label{lem:Mconv}
Let $p \in (0,1]$ and $(c_n \colon n \in \N)$ be a sequence with $\lim_{n \to \infty} c_n^3/n= 0$ and $\lim_{n \to \infty}c_n = \infty$. Then
\[
M_{n,p}^{\eps}(c_n):=\frac{1}{(1-\eps)n} \sum_{v=\fst+1}^n \mathbbm{1}\{|\connect_{n,p}^{\eps}(v)|\ge c_n\}  \to \zeta^{\eps}(p)\quad \text{in probability, as }n \to \infty.
\]
\end{lemma}

To prove Lemma~\ref{lem:Mconv}, we use a variance estimate for $M_{n,p}^{\eps}(c_n)$. 

\begin{lemma} \label{variance_lem}
Let $p \in (0,1]$ and $(c_n \colon n \in \N)$ be a positive sequence. There exists a constant $C>0$ such that 
$${\rm{Var}}(M_{n,p}^{\eps}(c_n)) \le \frac{C}{n}\, \Big(c_n +\frac{c_n^2}{\eps n}\Big).$$
\end{lemma}

The proof  is almost identical to the proof of Proposition~7.1 in \cite{DerMoe13}. The necessary changes are similar to the changes made for the proofs of 
Proposition~\ref{pro:graph_tree} and \ref{pro:graph_IBP}. We sketch only the main steps.

\begin{proof}[Proof sketch]
Write
\begin{equation}\label{eq:varExpan}
\begin{split}
&{\rm{Var}}\Big(\frac{1}{(1-\eps)n} \sum_{v=\fst+1}^n \mathbbm{1}\{|\connect_{n,p}^{\eps}(v)|\ge c_n\}\Big)\\
&= \frac{1}{(1-\eps)^2 n^2} \sum_{v,w= \fst+1}^n\!\!\Big( \P(|\connect_{n,p}^{\eps}(v)|\ge c_n, |\connect_{n,p}^{\eps}(w)| \ge c_n) -\P(|\connect_{n,p}^{\eps}(v)|\ge c_n)\P(|\connect_{n,p}^{\eps}(w)|\ge c_n)\Big).
\end{split}
\end{equation}
To estimate the probability $\P(|\connect_{n,p}^{\eps}(v)|\ge c_n, |\connect_{n,p}^{\eps}(w)| \ge c_n)$, we run two successive explorations in the graph $\grapheps_n(p)$, the first starting from $v$ and the second starting from $w$. For these explorations, we use the exploration process described below Proposition \ref{pro:graph_tree} but in every step only neighbours in the set of veiled vertices are explored. The first exploration is terminated as soon as either the number of dead and active vertices exceeds $c_n$ or there are no active vertices left. The second exploration, additionally, stops when a vertex is found which was already unveiled in the first exploration. We denote
\[
\Theta_v := \{\text{the first exploration started in vertex }v \text{ stops because }c_n \text{ vertices are found}\}.
\]
Then, for any $v \in \vseteps_n$,  $\P(|\connect_{n,p}^{\eps}(v)|\ge c_n)= \P(\Theta_v)$ and in the proof of Proposition~7.1 of \cite{DerMoe13} it was shown that there exists a constant $C'>0$, independent of $v$ and $n$, such that
\begin{equation}\label{eq:MsqEst}
\begin{split}
\sum_{w=\fst+1}^n &\P\big(|\connect_{n,p}^{\eps}(v)|\ge c_n, |\connect_{n,p}^{\eps}(w)| \ge c_n\big)\\
&\phantom{\P(|\connect_{n,p}^{\eps}(v))}\le \P(\Theta_v) \Big(c_n +\sum_{w=\fst+1}^n \P(\Theta_w) +C'c_n \P^{c_n}(\triangle \Z[\fst+1,\fst+1]=1)\Big).
\end{split}
\end{equation}
Combining \eqref{eq:varExpan} and \eqref{eq:MsqEst} and using \eqref{triangle_bound}, we deduce that there exists a constant $C>0$ with
\[
{\rm{Var}}(M_n^{\eps}(c_n))\le\frac{1}{(1-\eps)^2 n^2} \sum_{v= \fst+1}^n \P(\Theta_v)\Big(c_n + C' \frac{c_nf(c_n)}{\eps n}\Big) \le \frac{C}{n}\Big(c_n + \frac{c_n^2}{\eps n}\Big).\qedhere
\]
\end{proof}

\begin{proof}[Proof of Lemma~\ref{lem:Mconv}]
Using Chebyshev's inequality, Corollary~\ref{mean_conv} and Lemma~\ref{variance_lem} yield the claim.\end{proof}

Lemma~\ref{lem:Mconv} already implies that the asymptotic relative size of a largest component in the network is bounded from above by $\zeta^{\eps}(p)$. To show that the survival probability also constitutes a lower bound, we use the following sprinkling argument.

\begin{lemma}\label{sprinkling}
Let $\eps \in [0,1)$, $p \in (0,1]$, $\delta \in (0,f(0))$ and define $\underline{f}(k):=f(k)-\delta$ for all $k \in \N_0$. Denote by $\underline{\connect}_{n,p}^{\eps}(v)$ the connected component containing $v$ in the network $\underline{\graph}_n^{\eps}(p)$ constructed with the attachment rule~$\underline{f}$. Let $\kappa >0$ and $(c_n \colon n \in \N)$ be a sequence with
\[
\lim_{n \to \infty}[\tfrac{1}{2} \kappa (1-\eps) p \delta c_n -\log n]=\infty \quad \text{and}\quad \lim_{n \to \infty} c_n^2/n=0.
\]
Suppose that
\[
\frac{1}{n-\fst} \sum_{v =\fst+1}^n \mathbbm{1}\{|\underline{\connect}_{n,p}^{\eps}(v)|\ge 2c_n\} \ge \kappa \qquad \text{with high probability}.
\]
Then there exists a coupling of the networks $(\grapheps_n(p))_{n}$ and $(\underline{\graph}_{n}^{\eps}(p))_{n}$ such that $\underline{\graph}_n^{\eps}(p) \le \grapheps_n(p)$ and, with high probability, all connected components in $\underline{\graph}_n^{\eps}(p)$ with at least $2c_n$ vertices belong to one connected component in $\grapheps_n(p)$.
\end{lemma}

Lemma~\ref{sprinkling} in the case $\eps=0$ and $p=1$ is Proposition~4.1 in \cite{DerMoe13}. The proof is valid for $\eps \in [0,1)$, $p \in (0,1]$ up to 
obvious changes and is therefore omitted.

\begin{proof}[Proof of Theorem~\ref{thm:size_giant}]
Choose $c_n=(\log n)^2$. By Lemma~\ref{lem:Mconv}, we have in probability
\[
\limsup_{n \to \infty} \frac{|\connect_{n,p}^{\eps}|}{(1-\eps)n} \le \limsup_{n\to \infty} \max\Big\{\frac{c_n}{(1-\eps)n}, M_{n,p}^{\eps}(c_n)\Big\}\le \zeta^{\eps}(p).
\]
Moreover, for $\delta \in (0,f(0))$, another application of Lemma~\ref{lem:Mconv} implies that $M_{n,p}^{\eps}(2c_n,f-\delta)$ converges to $\zeta^{\eps}(p,f-\delta)$ in probability. Hence, Lemma~\ref{continuityZeta} and Lemma~\ref{sprinkling} imply that for all $\delta'>0$, $|\connect_{n,p}^{\eps}|\ge (n-\fst)  (\zeta^{\eps}(p)-\delta')$ with high probability. This concludes the proof.
\end{proof}

\section{Variations and other models}

We study the preferential attachment network with a non-linear attachment rule in Section \ref{sec:Genf}, and inhomogeneous random graphs and the configuration model in Sections~\ref{sec:IRRG} and~\ref{sec:CM}, respectively.

\subsection{Proof of Theorem~\ref{thm:genf}}\label{sec:Genf}

Theorem~\ref{thm:genf} is an immediate consequence of Theorem~\ref{thm:pc} and a stochastic domination result on the level of the networks. We make use of the notation and terminology introduced in Section~\ref{sec:nonlin_result}.

\begin{proof}[Proof of Theorem~\ref{thm:genf}]
First suppose that $f$ is a L-class attachment rule with $\overline{f}\ge f \ge \underline{f}$ for two affine\vspace{-1mm}\linebreak attachment rules given by
$\overline{f}(k)=\gamma k +\beta_u$ and $\underline{f}(k)=\gamma k +\beta_l$ for $k\in\N_0.$ There exists a natural coupling of the networks generated
by these attachment rules such that 
\[
\overline{\graph}_n \ge \graph_n \ge \underline{\graph}_n \quad \mbox{ for all $n \in \N$.}
\] 
This ordering is retained after percolation and implies and ordering $\overline{p}_{\rm c }(\eps) \le \pc(\eps) \le \underline{p}_{\rm c }(\eps)$ of the critical\vspace{-1mm}\linebreak percolation parameters
of the networks. Applying Theorem~\ref{thm:pc} to $\overline{f}$ and $\underline{f}$, we obtain positive constants $C_1,\ldots,C_4$ such that, for small $\eps \in (0,1)$,
\[
\frac{C_1}{\log(1/\eps)}\le \pc(\eps) \le \frac{C_2}{\log(1/\eps) } \mbox{ if }\gamma=\sfrac{1}{2}, \quad C_3 \eps^{\gamma-1/2}\le \pc(\eps) \le C_4 \eps^{\gamma-1/2}\mbox{ if }\gamma>\sfrac{1}{2},
\]
and the result follows.\smallskip

Now let $f$ be a C-class attachment rule. Concavity of $f$ implies that the increments $\gamma_k:=f(k+1)-f(k)$, for $k \in \N_0$, form a non-increasing sequence converging to~$\gamma$. In particular, with $\underline{f}(k):=\gamma k+f(0)$,\vspace{-1mm}\linebreak we get $f(k)=\sum_{l=0}^{k-1}\gamma_l +f(0)\ge  \gamma k +f(0)=\underline{f}(k)$, for all $k \in \N_0$. To obtain a corresponding upper bound, let $\beta_j:=f(j)-\gamma_j j$. Then $\beta_j >0$ and for all $k \in \N_0$,  
\[
f(k)-\beta_j=f(k)-f(j)+\gamma_j j=\left\{\begin{array}{rr}
-\sum_{l=k}^{j-1}\gamma_l +\gamma_j j \le -(j-k) \gamma_j+\gamma_j j& \text{ if } k \le j\\
\sum_{l=j}^{k-1}\gamma_l +\gamma_j j \le (k-j) \gamma_j+\gamma_j j& \text{ if } k \ge j\end{array}\right\}=\gamma_j k.
\]
Hence, the attachment rule given by $\overline{f}_j(k):=\gamma_j k +\beta_j$, for $k\in\N_0$, satisfies $\overline{f}_j\ge f$, and we can use the same coupling as in the first part of the proof  
to obtain 
\[
\overline{p}_{\rm c }^{(j)}(\eps) \le \pc(\eps) \le \underline{p}_{\rm c }(\eps),
\]
where $\overline{p}_{\rm c }^{(j)}(\eps)$ corresponds to the network with attachment rule~$\overline{f_j}$. Since $\gamma_j \downarrow \gamma$, we have $\gamma_j \in [\sfrac{1}{2},1)$ for large $j$. 
Theorem~\ref{thm:pc} yields, for $\gamma>\sfrac{1}{2}$, constants $C,C_j >0$ such that 
\[
\log C_j +(\gamma_j -1/2)\log \eps \le \log \pc(\eps) \le \log C+(\gamma-1/2) \log \eps.
\]
Dividing by $\log \eps$ and then taking first $\eps \downarrow 0$ and then $j\to \infty$ yields the claim for $\gamma >\sfrac{1}{2}$. In case $\gamma =\frac{1}{2}$ it could happen that $\gamma_j>\frac{1}{2}$ for all $j \in \N$. In this situation, Theorem~\ref{thm:pc} does not give a bound on the right scale. Therefore, we can use only the upper bound on $\pc(\eps)$ which gives the stated result.
\end{proof} 

\subsection{Other models}\label{sec:others}

In this section, we study vulnerability of two other classes of robust network models. 

\subsubsection{Configuration model} \label{sec:CM}

The configuration model is a natural way to construct a network with given degree sequence. Its close connection to the uniformly chosen simple graph with given degree sequence is explained in Section~7.4 in \cite{Hof13pre}. Existence of a giant component in the configuration model has been studied by Molloy and Reed~\cite{MolRee95} and Janson and Luczak~\cite{JanLuc08,Jan09}. 
Recall from Section~\ref{sec:CM_results} that we write $D$ for the weak limit of the degree of a uniformly chosen vertex. Janson and Luczak~\cite{JanLuc08} showed that if \eqref{as:CM} holds and $\P(D=2)<1$, then
\[
(\graph_n^{\sss (\CM)} \colon n \in \N) \text{ has a giant component } \Leftrightarrow \; \E[D (D-1)]>\E D.
\]
Janson \cite{Jan09} found a simple construction that allows to obtain a corresponding result for the network after random or deterministic removal of vertices (or edges), where the retention probability of a vertex can depend on its degree. Let $\boldsymbol{\pi}=(\pi_k)_{k \in \N}$ be a sequence of retention probabilities with $\pi_k \P(D=k)>0$ for some $k$. Every vertex $i$ is removed with probability $1-\pi_{d_i}$ and kept with probability $\pi_{d_i}$, independently of all other vertices. Janson describes the network after percolation as follows \cite[page 90]{Jan09}: For each vertex $i$ replace it with probability $1-\pi_{d_i}$ by $d_i$ new vertices of degree $1$. Then construct the configuration model $\graph_n^{\sss (\CM), \boldsymbol{\pi}}$ corresponding to the new degree sequence and larger number of vertices and remove from this graph uniformly at random vertices of degree $1$ until the correct number of vertices for $\graph_n^{\sss (\CM)}$ after percolation is reached. The removal of these surplus vertices cannot destroy or split the giant component since the vertices are of degree $1$. Hence, it suffices to study the existence of nonexistence of a giant component in $\graph_n^{\sss (\CM), \boldsymbol{\pi}}$.
\smallskip

To construct $\graph_n^{{\sss (\CM)},\eps}(p)$, we remove the $\fst$ vertices with the largest degree from $\graph_n^{\sss (\CM)}$ and then run 
vertex percolation with retention probability $p$ on the remaining graph. In general, this does not fit exactly into the setup of Janson. To emulate the behaviour, we denote by $n_j$ the number of vertices with degree~$j$ in the graph and let 
$K_n=\inf\{ k \in \N_0\colon \sum_{j=k+1}^{\infty} n_j \le \fst\}$. Then all vertices with degree larger than $K_n$ are deterministically removed in $\graph_n^{{\sss (\CM)},\eps}(p)$, i.e.\ $\pi_j=0$ for $j \ge K_n+1$. In addition, we deterministically remove $\fst-\sum_{j=K_n+1}^{\infty} n_j$ vertices of degree $K_n$, while all other vertices are subject to vertex percolation with retention probability $p$. In particular, $\pi_j=p$ for $j \le K_n-1$.
Denote by $F(x):=\P(D\le x)$, for $x \ge 0$, and by $[1-F]^{-1}$ the generalized inverse of $[1-F]$, that is
$[1-F]^{-1}(u)=\inf\{k \in \N_0\colon [1-F](k) \le u\}$ for all $u \in (0,1)$.
One easily checks that $K_n \in \{m,m+1\}$ for all large $n$, where $m:=[1-F]^{-1}(\eps)$. Using this observation, it is not difficult to adapt Janson's proof to show that
\[
(\graph_n^{{\sss(\CM)},\eps}(p)\colon n \in \N) \text{ has a giant component } \Leftrightarrow \; p>\pc(\eps),
\]
where
\[
\pc(\eps):= \frac{\E D}{ \E[D(D-1)\mathbbm{1}_{D \le m}] - m (m-1)(\eps- [1-F](m))}.
\]

\begin{proof}[Proof of Theorem~\ref{thm:CM}]
Let $U$ be a uniformly distributed random variable on $(0,1)$. Then $[1-F]^{-1}(U)$ has the same distribution as $D$ and
\begin{align*}
\E[D(D-1)\mathbbm{1}_{\{D \le m\}}] - m (m-1)(\eps- [1-F](m))&= \E\big[[1-F]^{-1}(U) \big([1-F]^{-1}(U)-1\big) \mathbbm{1}_{\{U \ge \eps\}}\big]\\
&\asymp \E\big[[1-F]^{-1}(U)^2 \mathbbm{1}_{\{U \ge \eps\}}\big].
\end{align*}
The assumption $[1-F](k) \sim C k^{-1/\gamma}$ as $k \to \infty$ implies that $[1-F]^{-1}(u) \sim C^{\gamma} u^{-\gamma}$ as $u \downarrow 0$. Let $u_0>0$ such that
\[
\frac{1}{2} \le \frac{[1-F]^{-1}(u)}{C^{\gamma} u^{-\gamma} } \le \frac{3}{2} \qquad \mbox{ for all } u \le u_0.
\]
Since $[1-F]^{-1}(u)^2$ is not integrable around zero, but bounded on $[u_0,1)$, we deduce
\begin{align*}
\E\big[[1-F]^{-1}(U)^2 \mathbbm{1}_{\{U \ge \eps\}}\big]&= \int_{\eps}^1 [1-F]^{-1}(u)^2 \, du \asymp \int_{\eps}^{u_0} [1-F]^{-1}(u)^2 \, du\\
& \asymp \int_{\eps}^{u_0} u^{-2\gamma} \, du\asymp\begin{cases}
\log(1/\eps) & \text{if }\gamma=\frac{1}{2},\\
\eps^{1-2\gamma} & \text{if }\gamma >\frac{1}{2}.
\end{cases}\qedhere
\end{align*}
\end{proof}

\subsubsection{Inhomogeneous random graphs} \label{sec:IRRG}
The classical Erd\H{o}s-R\'{e}nyi random graph can be generalized by giving each vertex a weight and choosing the probability for an edge between two vertices as an increasing function of their weights. Suppose that $\kappa \colon (0,1]\times (0,1] \to (0,\infty)$ is a symmetric, continuous kernel with
\begin{equation}\label{eq:IRRG_cond}
\int_0^1 \int_0^1 \kappa(x,y) \, dx \, dy<\infty
\end{equation}
and recall from \eqref{def:IRRG} that in the inhomogeneous random graph $\graph_n^{\sss (\kappa)}$, the edge $\{i,j\}$ is present with probability $\frac{1}{n} \kappa(\frac{i}{n},\frac{j}{n}) \land 1$, independently of all other edges. We assume that vertices are ordered in decreasing order of privilege, i.e.\ $\kappa$ is non-increasing in both components.
Bollob{\'a}s et al.\ show in Theorem~3.1 and Example~4.11 of \cite{BolJanRio07} that 
\begin{equation}\label{eq:IRG_gen}
(\graph_n^{{\sss (\kappa)},\eps}(p)\colon n \in \N) \text{ has a giant component } \Leftrightarrow \; p >\pc(\eps):=\|T_{\kappa}\|_{L^2(\eps,1)}^{-1}, 
\end{equation}
where
\[
T_{\kappa}g(x)=\int_{\eps}^1 \kappa(x,y)g(y) \, dy, \qquad \mbox{ for all } x \in (\eps,1),
\]
for all measurable functions $g$ such that the integral is well-defined, and $\|\cdot\|_{L^2(\eps,1)}$ denotes the operator norm on the $L^2$-space with respect to the Lebesgue measure 
on $(\eps,1)$. The same result holds for a version of the Norros-Reittu model in which edges between different vertex pairs are independent and edge $\{i,j\}$ is present with probability $1-e^{-\kappa(i/n,j/n)/n}$ for all $i,j \in \{1,\ldots,n\}$. Consequently, the estimates given in Theorem~\ref{thm:IRG} hold for this model, too. 

\begin{proof}[Proof of Theorem~\ref{thm:IRG}]
Since $\kappa^{\sss (\CL)}$ and $\kappa^{\sss (\PA)}$ are positive, symmetric, continuous kernels satisfying \eqref{eq:IRRG_cond}, the first part of the theorem follows immediately from \eqref{eq:IRG_gen}. By definition 
\[
\|T_{\kappa}\|_{L^2(\eps,1)}=\sup\big\{\|T_{\kappa}g\|_{L^2(\eps,1)}\colon \|g\|_{L^2(\eps,1)}\le 1\big\}.
\]
For a rank one kernel $\kappa(x,y)=\psi(x) \psi(y)$, the operator norm of $T_{\kappa}$ is attained at $\psi/\|\psi\|_{L^2(\eps,1)}$ with $\|T_{\kappa}\|_{L^2(\eps,1)}=\|\psi\|_{L^2(\eps,1)}^2$. Hence,
\[
\|T_{\kappa^{\sss(\CL)}}\|_{L^2(\eps,1)}= \int_{\eps}^1 x^{-2\gamma} \, dx=\begin{cases}
\log( 1/\eps) & \text{if }\gamma =1/2,\\
\frac{1}{1-2\gamma} \big[1-\eps^{1-2\gamma}\big] & \text{if } \gamma \not =\frac{1}{2}, 
\end{cases}
\]
and
\[
\pc^{\sss (\CL)}(\eps)=\begin{cases}
(1-2\gamma)\frac{1}{1-\eps^{1-2\gamma}} & \text{if }\gamma<\frac{1}{2},\\
\frac{1}{\log(1/\eps)} & \text{if }\gamma=\frac{1}{2},\\
(2\gamma -1) \eps^{2\gamma-1} \frac{1}{1-\eps^{2\gamma-1}} & \text{if } \gamma>\frac{1}{2}.
\end{cases}
\]
Now suppose that $\gamma>1/2$. By Cauchy-Schwarz's inequality and the symmetry of $\kappa^{\sss(\PA)}$,
\begin{align*}
\|T_{\kappa^{\sss(\PA)}}\|_{L^2(\eps,1)}^2&\le \int_{\eps}^1 \int_{\eps}^1 \kappa^{\sss (\PA)}(x,y)^2 \, dy \,dx= 2 \int_{\eps}^1 \int_{\eps}^x x^{2(\gamma-1)} y^{-2\gamma} \, dy \,dx\\
& \le 2 \int_{0}^1 x^{2(\gamma-1)} \, dx \int_{\eps}^{\infty} y^{-2\gamma} \, dy = \frac{2}{(2\gamma-1)^2} \eps^{1-2\gamma}.
\end{align*}
For the lower bound, let $c_{\eps}= \sqrt{2\gamma-1} \eps^{\gamma-1/2}$ and $g(x)=c_{\eps} x^{-\gamma}$. Then $\|g\|_{L^2(\eps,1)}\le 1$ and 
\begin{align*}
\|T_{\kappa^{\sss (\PA)}}\||_{L^2(\eps,1)}^2 &\ge \|T_{\kappa^{\sss (\PA)}}g\||_{L^2(\eps,1)}^2 \ge \int_{\eps}^1 \Big(\int_{\eps}^x \kappa^{\sss (\PA)}(x,y) g(y) \, dy\Big)^2 \, dx \\
&= \frac{c_{\eps}^2}{(2\gamma-1)^2} \int_{\eps}^1x^{2(\gamma-1)} [\eps^{1-2\gamma}-x^{1-2\gamma}]^2 \, dx\\
&\ge \frac{c_{\eps}^2}{(2\gamma-1)^2} \int_{\eps}^1x^{2(\gamma-1)} \big[\eps^{2(1-2\gamma)} - 2\eps^{1-2\gamma}x^{1-2\gamma} \big] \, dx\\
&= \frac{\eps^{1-2\gamma}}{(2\gamma-1)^2} \Big[1-\eps^{2\gamma-1} + 2(2\gamma-1) \eps^{2\gamma-1}\log \eps \Big].
\end{align*}
The claim follows. 
\end{proof}

\end{document}